\documentclass[11pt]{amsart}
\usepackage{mathrsfs}
\usepackage{amsmath}
\usepackage{amsfonts}
\usepackage{amssymb}
\usepackage{amsxtra}
\usepackage{mathtools}
\usepackage{dsfont}
\usepackage[dvipsnames]{xcolor}

\usepackage[T1]{fontenc}
\usepackage{newtxmath}
\usepackage{newtxtext}
\usepackage[margin=2.5cm, centering]{geometry}
\usepackage{enumitem}

\usepackage[english,polish]{babel}
\newcommand{\pl}[1]{\foreignlanguage{polish}{#1}}

\usepackage[colorlinks,citecolor=blue,urlcolor=blue,bookmarks=true]{hyperref}
\hypersetup{
pdfpagemode=UseNone,
pdfstartview=FitH,
pdfdisplaydoctitle=true,
pdfborder={0 0 0}, 
pdftitle={Transition densities of subordinators of positive order},
pdfauthor={Tomasz Grzywny, {\L{}ukasz Le\.{z}aj}, and Bartosz Trojan},
pdflang=en-US
}

\newtheorem{theorem}{Theorem}[section]
\newtheorem{proposition}[theorem]{Proposition}
\newtheorem{lemma}[theorem]{Lemma}
\newtheorem{claim}[theorem]{Claim}
\newtheorem{corollary}[theorem]{Corollary}

\newcounter{thm}

\newtheorem{main_theorem}[thm]{Theorem}

\theoremstyle{definition}
\newtheorem{remark}[theorem]{Remark}

\newtheorem{example}[theorem]{Example}

\numberwithin{equation}{section}

\newcommand{\RR}{\mathbb{R}}
\newcommand{\NN}{\mathbb{N}} 
\newcommand{\EE}{\mathbb{E}}
\newcommand{\PP}{\mathbb{P}}

\newcommand{\CC}{\mathbb{C}}
\newcommand{\calA}{\mathcal{A}}
\newcommand{\calB}{\mathcal{B}}

\newcommand{\calF}{\mathcal{F}}

\newcommand{\calL}{\mathcal{L}}

\newcommand{\scrX}{\mathscr{X}}

\newcommand{\bfT}{\mathbf{T}}

\newcommand{\vphi}{\varphi}

\newcommand{\rmd}{{\: \rm d}}
\newcommand{\abs}[1]{{\lvert {#1} \rvert}}

\DeclareMathOperator{\diam}{diam}
\DeclareMathOperator{\supp}{supp}

\def\({\left(} 
\def\){\right)} 
\def\<{\langle} 
\def\>{\rangle}

\newcommand{\ind}[1]{{\mathds{1}_{{#1}}}}

\newcommand{\WUSC}[3]{\textrm{\rm WUSC}(#1,#2,#3)}
\newcommand{\WLSC}[3]{\textrm{\rm WLSC}(#1,#2,#3)}

\begin{document}
\selectlanguage{english}

\title{Transition densities of subordinators of positive order}
\author{Tomasz Grzywny}
\address{
        \pl{
        Tomasz Grzywny\\
        Wydzia{\lll} Matematyki,
        Politechnika Wroc{\lll}awska\\
        Wyb. Wyspia\'{n}skiego 27\\
        50-370 Wroc{\lll}aw\\
        Poland}
}
\email{tomasz.grzywny@pwr.edu.pl}

\author{\pl{\L{}ukasz Le\.{z}aj}}
\address{
        \pl{
		\L{}ukasz Le\.{z}aj\\
        Wydzia{\lll} Matematyki,
        Politechnika Wroc{\lll}awska\\
        Wyb. Wyspia\'{n}skiego 27\\
        50-370 Wroc{\lll}aw\\
        Poland}
}
\email{lukasz.lezaj@pwr.edu.pl}

\author{Bartosz Trojan}
\address{
        Bartosz Trojan\\
        Instytut Matematyczny
        Polskiej Akademii Nauk\\
        ul. \pl{{\'S}niadeckich} 8\\
        00-656 Warszawa\\
        Poland}
\email{btrojan@impan.pl}

\thanks{The authors were partially supported by the National Science Centre (Poland): grant 2016/23/B/ST1/01665.}

\begin{abstract}
	We prove existence and asymptotic behavior of the transition density for a large class of subordinators whose Laplace exponents satisfy lower scaling condition at infinity. Furthermore, we present lower and upper bounds for the density. Sharp estimates are provided if additional upper scaling condition on the Laplace exponent is imposed. In particular, we cover the case when the (minus) second derivative of the Laplace exponent is a function regularly varying at infinity with regularity index bigger than $-2$.
\end{abstract}

\maketitle

\section{Introduction}
Asymptotic behavior as well as estimates of heat kernels have been intensively studied in the last decades. The first 
results obtained by P\`{o}lya \cite{1923P}, and Blumenthal and Getoor \cite{MR0119247} for isotropic $\alpha$-stable
process in $\RR^d$ provided the basis to studies of more complicated processes, e.g. subordinated Brownian motions
\cite{MR3570240, MR2045955}, isotropic unimodal L\'{e}vy processes \cite{MR3165234, MR3646773, GRT2017} and even more general
symmetric Markov processes \cite{CKK11, CK08}. One may, among others, list the articles on heat kernel estimates for jump processes of finite range \cite{CKK08} or with lower intensity of higher jumps \cite{MR2886459, MR3604626}. While a great many of articles with explicit results
is devoted to symmetric processes or those which are, in appropriate sense, similar to the symmetric ones, the nonsymmetric case is in general harder to handle due to lack of familiar
structure. This problem was approached in many different ways, see e.g. \cite{BSK2017, MR1301283, MR1829978, MR3089797,
MR3357585, MR3235175, MR3139314, MR1486930, MR1449834, MR2794975}. For more specific class of stable processes, see
\cite{MR1287843, MR0250372, MR2286060}. Overall, one has to impose some control on the nonsymmetry in order to obtain
estimates in an easy-to-handle form. This idea was applied in the recent paper \cite{GS2017} where the authors considered
the case of the L\'{e}vy measure being comparable to some unimodal L\'{e}vy measure. The methods developed in
\cite{GS2017, GS2019} contributed significantly to this paper. See also \cite{MR3357585, MR1486930} and the references therein.

In this article the central object is a subordinator, that is a one-dimensional L\'{e}vy process with
nondecreasing paths starting at $0$, see Section \ref{sec:3} for the precise definition. The abstract introduction of the
subordination dates back to 1950s and is due to Bochner \cite{MR0030151} and Philips \cite{MR0050797}. In the language of
the semigroup theory, for a Bernstein function $\phi$ and a bounded $C_0$-semigroup $\big( e^{-t\calA} \colon t \geq 0\big)$
with $-\calA$ being its generator on some Banach space $\scrX$, via Bochner integral one can define an operator
$\calB=\phi(\calA)$ such that $-\calB$ also generates a bounded $C_0$-semigroup $\big( e^{-t\calB} \colon t \geq 0\big)$ on
$\scrX$. The semigroup $\big( e^{-t\calB} \colon t \geq 0\big)$ is then said to be \emph{subordinated} to
$\big( e^{-t\calA} \colon t \geq 0\big)$, and although it may be very different from the original one, its properties clearly
follow from properties of both the \emph{parent} semigroup and the involved Bernstein function. See for example
\cite{MR3383800} and the references therein. From probabilistic point of view, due to positivity and monotonicity,
subordinators naturally appear as a random time change functions of L\'{e}vy processes, or more generally, Markov
processes. Namely, if $(X_t \colon t \geq 0)$ is a Markov process and $(T_t \colon t \geq 0)$ is an independent subordinator
then $Y_t=X_{T_t}$ is again a Markov process with a transition function given by
\[
	\PP^x(Y_t \in A) = \int_{[0, \infty)} \PP^x(X_s \in A) \:\PP(T_t \in {\rm d}s).
\]
The procedure just described is called a subordination of a Markov process and can be interpreted as a probabilistic form of the equality $\calB=\phi(\calA)$. Here $\calA$ and $\calB$ are (minus) generators of semigroups associated to processes $X_t$ and $Y_t$, respectively. From analytical point of view, the transition density of $Y_t$ (the integral kernel of $e^{-t\calB}$) can be obtained as a time average of transition density of $X_t$ with respect to distribution of $T_t$. Yet another approach is driven by PDEs, as the transition density is a heat kernel of a generalized heat equation. The generalization can be twofold: either by replacing the Laplace operator with another, possibly nonlocal operator, or by introducing a
more general fractional-time derivative instead of the classical one. The latter case was recently considered
in \cite{MR2074812, MR3672008, MR3849640}. Here the solutions are expressed in terms of corresponding
(inverse) subordinators and thus their analysis is essential.

By taking $\calA=-\Delta$ and changing the time of (i.e. subordinating)
Brownian motion one can obtain a large class of subordinated Brownian motions. A principal example here is an 
$\alpha$-stable subordinator with the Laplace exponent $\phi(\lambda)=\lambda^{\alpha}$, $\alpha \in (0,1)$, which gives
rise to the symmetric, rotation-invariant $\alpha$-stable process and corresponds to the special case of fractional powers
of semigroup $\big(e^{-t\calA^{\alpha}} \colon t \geq 0\big)$. For this reason,
distributional properties of subordinators were often studied with reference to heat kernel estimates of subordinated
Brownian motions (see e.g. \cite{MR3835470, Fahrenwaldt2018}). In \cite{MR0282413} Hawkes investigated the growth of
sample paths of a stable subordinator and obtained the asymptotic behavior of its distribution function. Jain and Pruitt \cite{JainPruitt87} considered tail probability estimates for subordinators and, in the discrete case, nondecreasing random walks. In a more
general setting some related results were obtained in \cite{MR0329049, MR3575536, MR1486930, MR952835}. In
\cite{MR3231629} new examples of families of subordinators with explicit transition densities were given. Finally, in the
recent paper \cite{Fahrenwaldt2018} the author under very restrictive assumptions derives explicit approximate expressions
for the transition density of approximately stable subordinators. 

The result of the paper is asymptotic behavior as well as upper and lower estimates of transition densities of
subordinators satisfying scaling condition imposed on the second derivative of the Laplace exponent $\phi$. 
Our standing assumption on $-\phi''$ is the weak lower
scaling condition at infinity with scaling parameter $\alpha-2$,  for some $\alpha >0$ (see \eqref{wlsc} for definition).
It is worth highlighting that we do not state our assumptions and results in terms of the Laplace exponent $\phi$, as one
could suspect, but in terms of its second derivative and related function $\varphi(x)=x^2(-\phi''(x))$ (see Theorems
\ref{thm:3}, \ref{thm:6} and \ref{thm:4}). Usually the transition density of a L\'{e}vy process is described by the
generalized inverse of the real part of the characteristic exponent $\psi^{-1}(x)$ (e.g. \cite{GS2019}, \cite{MR3139314}),
but in our setting one can show that the lower scaling property implies that $\varphi^{-1}(x) \approx \psi^{-1}(x)$ for 
$x$ sufficient large (see Proposition \ref{prop:7}). In some cases, however, $\varphi$ may be significantly different from
the Laplace exponent $\phi$. However, if one assumes additional upper scaling condition with scaling parameter $\beta-2$
for $\beta$ strictly between $0$ and $1$, then these two objects are comparable (see Proposition \ref{prop:1}).

The main results of this paper are covered by Theorems \ref{thm:3}, \ref{thm:6}, \ref{thm:4}, \ref{thm:2}, and \ref{thm:5}. 
Theorem \ref{thm:3} is essential for the whole paper because it provides not only the existence
of the transition density but also its asymptotic behavior, which is later used in derivation of upper and lower
estimates. The key argument in the proof is the lower estimate on the holomorphic extension of the Laplace exponent $\phi$
(see Lemma \ref{lem:1}) which justifies the inversion of the Laplace transform and allows us to perform the saddle point
type approximation. In Theorem \ref{thm:3} we only use the weak lower scaling property on the second derivative of the
Laplace exponent. In particular, we do not assume the absolute continuity of $\nu({\rm d} x)$. Furthermore, the asymptotic
is valid in some region described in terms of both space and time variable. By freezing one of them, we obtain as corollaries
the results similar to \cite{Fahrenwaldt2018}, see e.g. Corollary \ref{cor:5}. It is also worth highlighting that we
obtain a version of upper estimate on the transition density with no additional assumptions on the L\'{e}vy measure
$\nu(\rm{d}x)$, see Theorem \ref{thm:6}. Clearly, putting some restrictions on $\nu(\rm{d}x)$ results in sharper estimates
(Theorem \ref{thm:4}), but it is interesting that the scaling property sole is enough to get some information.
Our starting point and the main object to work with is the Laplace exponent $\phi$. However, in many cases the primary
object is the L\'{e}vy measure $\nu({\rm d}x)$ and results are presented in terms of or require its tail decay. Therefore,
it would be convenient to have a connection between those two objects. In Proposition \ref{prop:8} we prove that one can
impose scaling conditions on the tail of the L\'{e}vy measure $\nu((x,\infty))$ instead, as they imply the scaling condition on $-\phi''$. 

We also note that the main results of the article hold true when $-\phi''$ is a function regularly varying at infinity with regularity index $\alpha-2$, where $\alpha>0$. This follows easily by Potter bounds for regularly varying functions (see \cite[Theorem 1.5.6]{MR1015093}), which immediately imply both lower and upper scaling properties. Moreover, if additionally $\alpha<1$ then, by Karamata's theorem and monotone density theorem, regular variation of $-\phi''$ with index $\alpha-2$ is equivalent to regular variation of $\phi$ with index $\alpha$. This is not the case for the case $\alpha=1$ where, in general, only the first direction holds true.

Below we present the special case when global upper and lower scaling conditions are imposed with 
$0 < \alpha \leq \beta < 1$, see Theorem \ref{thm:5}.
\begin{main_theorem}
	\label{main_thm:1}
	Let $\bfT$ be a subordinator with the Laplace exponent $\phi$. Suppose that for some $0 < \alpha \leq \beta < 1$,
	the functions
	\[
		(0, \infty) \ni x \mapsto x^{-\alpha} \phi(x),
	\qquad\text{and}\qquad
		(0, \infty) \ni x \mapsto x^{-\beta} \phi(x)
	\]
	are almost increasing and almost decreasing, respectively. We also assume that the L\'{e}vy measure $\nu({\rm d}x)$
	has an almost monotone density $\nu(x)$. Then the probability distribution of $T_t$ has a density $p(t, \:\cdot\:)$.
	Moreover, for all $t \in (0, \infty)$ and $x > 0$,
	\[
		p(t,x) \approx 
		\begin{cases}
		\big(t \big(  -\phi''(w)  \big)\big)^{-1/2} 
		\exp \left\lbrace -t \big( \phi(w) - w\phi'(w) \big) \right\rbrace, 
		&\text{if } 0 < x\phi^{-1}(1/t) \leq 1, \\
		tx^{-1}\phi(1/x), &\text{if } 1 < x\phi^{-1}(1/t),
		\end{cases} 
	\]
	where $w=(\phi')^{-1}(x/t)$.
\end{main_theorem}
We note that a similar result to Theorem \ref{main_thm:1} appeared in \cite{CKKW20} in around the same time as our preprint. Our assumptions, however, are weaker, as we assume almost monotonicity of the L\'{e}vy density instead of monotonicity of the function $t \mapsto t\nu(t)$. Moreover, our estimates are genuinely sharp, i.e. the constants appearing in the exponential factors are the same on both sides of the estimate, while estimates obtained in \cite{CKKW20} are qualitatively sharp, that is the constants in the exponential factors are different.

As a corollary, under the assumption of Theorem \ref{main_thm:1}, we obtain global two-sided estimate on the Green function. Namely, for all $x>0$,
\[
G(x) \approx \frac{1}{x\phi(1/x)}.
\]
See Section \ref{sec:5} and Theorem \ref{thm:8} for details. 

The article is organized as follows: In Section \ref{sec:3} we introduce our framework and collect some facts
concerning Bernstein functions and their scaling properties. Section \ref{sec:1} is devoted to the proof of Theorem
\ref{thm:3} and its consequences. In Section \ref{sec:4} we provide both upper and lower estimates on the transition
density and discuss when these estimates coincide. Some applications of our results to subordination beyond the familiar
$\mathbb{R}^d$ setting  and Green function estimates are presented in Section \ref{sec:5}.

\subsection*{Acknowledgment}
	We thank professor Jerzy Zabczyk for drawing our attention to the problem considered in this paper. The main
	results of this article were presented at the XV Probability Conference held from May 21 to 25, 2018, in B\k{e}dlewo,
	Poland, and at the Semigroups of Operators: Theory and Applications Conference held from September 30 to October 5,
	2018, in Kazimierz Dolny, Poland. We thank the organizers for the invitations. 

\subsection*{Notation}
By $C_1, c_1, C_2, c_2, \ldots$ we denote positive constants which may change from line to line. For two functions
$f,g \colon (0,\infty) \rightarrow [0,\infty)$ we write $f(x) \gtrsim g(x)$, if there is $c>0$ such that $f(x) \geq cg(x)$ for all $x>0$. An analogous rule is applied to the symbol $\lesssim$. We also have  $f(x) \approx g(x)$, if there exists $C \geq 1$ such that
$C^{-1} f(x) \leq g(x) \leq C f(x)$ for all $x > 0$. Finally, we set $a \wedge b=\min \{ a,b \}$ and
$a \vee b = \max \{a,b\}$. 

\section{Preliminaries}
\label{sec:3}
Let $(\Omega, \calF, \PP)$ be a probability space. Let $\bfT = (T_t \colon t \geq 0)$ be a subordinator, that is a L\'{e}vy
process in $\RR$ with nondecreasing paths. Recall that a L\'{e}vy process is a c\`{a}dl\`{a}g stochastic process with
stationary and independent increments such that $T_0 = 0$ almost surely. There is a function
$\psi \colon \RR \rightarrow \CC$, called the \emph{L\'{e}vy--Khintchine exponent} of $\bfT$, such that
for all $t \geq 0$ and $\xi \in \RR$, 
\[
	\EE \big( e^{i \xi T_t} \big) = e^{-t\psi(\xi)}. 
\]
Moreover, there are $b \geq 0$ and $\sigma$-finite measure $\nu$ on $(0, \infty)$ satisfying
\begin{equation}\label{eq:13}
    \int_{(0, \infty)} \min{\{1, s\}} \: \nu({\rm d}s) < \infty,
\end{equation}
such that for all $\xi \in \RR$,
\begin{equation}\label{eq:13a}\begin{aligned}
	\psi(\xi) &= - i \xi b - 
	\int_{(0, \infty)} \big(e^{i \xi x} - 1\big)
	\: \nu ({\rm d} x) \\
	&= - i \xi \bigg( b + \int_{(0,1)} x \: \nu ({\rm d} x)\bigg) -
	\int_{(0, \infty)} \big(e^{i \xi x} - 1 - i\xi x \ind{\{x<1\}}\big)
	\: \nu ({\rm d} x),
\end{aligned}
\end{equation}
which is valid thanks to \eqref{eq:13}. By $\phi\colon [0, \infty) \rightarrow [0, \infty)$ we denote the \emph{Laplace exponent} of $\bfT$, namely
\[
	\EE\big(e^{-\lambda T_t} \big) = e^{-t \phi(\lambda)}
\]
for all $t \geq 0$ and $\lambda \geq 0$. Let $\psi^*$ be the symmetric continuous and nondecreasing majorant of
$\Re \psi$, that is
\begin{align*}
	\psi^*(r) = \sup_{|z| \leq r} \Re \psi(z), \quad r>0.
\end{align*}
Notice that
\[
	\psi^* \big( \psi^{-1}(s) \big)=s, \qquad\text{and}\qquad \psi^{-1} \big( \psi^*(s) \big) \geq s,
\]
where $\psi^{-1}$ is the generalized inverse function defined as
\[
	\psi^{-1}(s) = \sup\big\{ r>0 \colon \psi^*(r)=s \big\}.
\]
To study the distribution function of the subordinator $\bfT$, it is convenient to introduce two concentration functions
$K$ and $h$. They are defined as
\begin{equation}
	\label{eq:30}
	K(r)=\frac{1}{r^2} \int_{(0, r)} s^2 \: \nu({\rm d} s), \quad r > 0,
\end{equation}
and
\begin{equation}
	\label{eq:31}
	h(r) = \int_{(0, \infty)} \min{\big\{ 1 , r^{-2} s^2\big\}} \: \nu({\rm d} s), \quad r > 0.
\end{equation}
Notice that $h(r) \geq K(r)$. Moreover, by the Fubini--Tonelli theorem, we get
\begin{align}
    \label{eq:25}
    h(r)=2 \int_r^\infty K(s) s^{-1} {\: \rm d} s.
\end{align}
In view of \cite[Lemma 4]{MR3225805}, we have
\begin{align}
	\label{eq:27}
	\frac{1}{24} h(r^{-1} ) \leq \psi^*(r) \leq 2h(r^{-1} ).
\end{align}
A function $f\colon [0, \infty) \rightarrow [0, \infty)$ is \emph{regularly varying at infinity} of index $\alpha$,
if for all $\lambda \geq 1$,
\[
	\lim_{x \to \infty} \frac{f(\lambda x)}{f(x)} = \lambda^\alpha.
\]
Analogously, $f$ is \emph{regular varying at the origin} of index $\alpha$, if for all $\lambda \geq 1$,
\[
	\lim_{x \to 0^+} \frac{f(\lambda x)}{f(x)} = \lambda^\alpha.
\]
If $\alpha = 0$ the function $f$ is called \emph{slowly varying}.

We next introduce a notion of scaling conditions frequently used in this article. We say that a function 
$f \colon [0,\infty) \rightarrow [0,\infty)$ has the \emph{weak lower scaling property} at infinity, if there are 
$\alpha \in \RR$, $c \in (0,1]$, and $x_0 \geq 0$ such that for all $\lambda \geq 1$ and $x > x_0$, 
\begin{align}
	\label{wlsc}
	f(\lambda x) \geq c \lambda^{\alpha} f(x).
\end{align}
We denote it briefly as $f \in \WLSC{\alpha}{c}{x_0}$. Observe that
if $\alpha > \alpha'$ then $\WLSC{\alpha}{c}{x_0} \subsetneq \WLSC{\alpha'}{c}{x_0}$.
Analogously, $f$ has the \emph{weak upper scaling property} at infinity, if there are $\beta \in \RR$, $C \geq 1$,
and $x_0 \geq 0$ such that for all $\lambda \geq 1$ and $x > x_0$,
\begin{align}
	\label{wusc}
	f(\lambda x) \leq C \lambda^{\beta}f(x).
\end{align}
In this case we write $f \in \WUSC{\beta}{C}{x_0}$. 

We say that a function $f\colon [0, \infty) \rightarrow [0, \infty)$ has \emph{doubling property} on $(x_0, \infty)$
for some $x_0 \geq 0$, if there is $C \geq 1$ such that for all $x > x_0$,
\[
	C^{-1} f(x) \leq f(2x) \leq C f(x).
\]
Notice that a nonincreasing function with the weak lower scaling has doubling property.
Analogously, a nondecreasing function with the weak upper scaling.

A function $f\colon [0, \infty) \rightarrow [0, \infty)$ is \emph{almost increasing} on $(x_0, \infty)$
for some $x_0 \geq 0$, if there is $c \in (0, 1]$ such that for all $y \geq x > x_0$,
\[
	c f(x) \leq f(y).
\]
It is \emph{almost decreasing} on $(x_0, \infty)$, if there is $C \geq 1$ such that for all $y \geq x > x_0$,
\[
	C f(x) \geq f(y).
\]
In view of \cite[Lemma 11]{MR3165234}, $f \in \WLSC{\alpha}{c}{x_0}$ if and only if the function
\[
	(x_0, \infty) \ni x \mapsto x^{-\alpha} f(x)
\]
is almost increasing. Similarly, $f \in \WUSC{\beta}{C}{x_0}$ if and only if the function
\[
	(x_0, \infty) \ni x \mapsto x^{-\beta} f(x)
\]
is almost decreasing. For a function $f\colon [0,\infty) \rightarrow \CC$ its Laplace transform is defined as
\[
	\calL f(\lambda) = \int_0^\infty e^{-\lambda x} f(x) \rmd x.
\]

\subsection{Bernstein functions}
\label{sec:2}
In this section we recall some basic facts about Bernstein functions. A general reference here is the book
\cite{MR2978140}.

A function $\phi\colon (0, \infty) \rightarrow [0,\infty)$ is \emph{completely monotone} if it is smooth and
\[
	(-1)^n \phi^{(n)} \geq 0
\]
for all $n \in \NN_0$. It is a \emph{Bernstein} function if $\phi$ is a nonnegative smooth function such that $\phi'$
is completely monotone. 

Let $\phi$ be a Bernstein function. In view of \cite[Lemma 3.9.34]{MR1873235}, for all $n \in \NN$ we have 
\begin{equation}
	\label{eq:20}
	\phi(\lambda) \geq \frac{(-1)^{n+1}}{n!} \lambda^n \phi^{(n)}(\lambda), \quad \lambda>0.
\end{equation}
Since $\phi$ is concave, for each $\lambda \geq 1$ and $x > 0$ we have
\[
	\phi(\lambda x) 
	\leq \phi'(x) (\lambda - 1) x + \phi(x),
\]
thus, by \eqref{eq:20},
\begin{equation}
	\label{eq:28}
	\phi(\lambda x) \leq \lambda \phi(x).
\end{equation}
By \cite[Theorem 3.2]{MR2978140}, there are two nonnegative numbers $a$ and $b$, and a Radon measure $\mu$ on
$(0, \infty)$ satisfying 
\[
    \int_{(0, \infty)} \min{\{1, s \}} \: \mu({\rm d} s) < \infty,
\]
and such that
\begin{equation}
    \label{eq:1}
    \phi(\lambda) = a+b\lambda+ \int_{(0, \infty)} \big(1 - e^{-\lambda s}\big) \: \mu({\rm d} s).
\end{equation}
A Bernstein function $\phi$ is called \emph{complete} Bernstein function if the measure $\mu$ has a completely monotone
density with respect to Lebesgue measure.
\begin{proposition}
	\label{prop:2}
	Let $f$ be a completely monotone function. Suppose that $f$ has a doubling property on $(x_0,\infty)$ for some
	$x_0\geq 0$. Then there is $C > 0$ such that for all $x > x_0$,
	\[
		f(x) \geq C x |f'(x)|.
	\]
\end{proposition}
\begin{proof}
	Without loss of generality we can assume $f \not\equiv 0$. Clearly,
	\[
		f(x)-f(x/2)=\int^x_{x/2}f'(s) \rmd s \leq \tfrac{1}{2} x f'(x).
	\]
	Since $f$ is completely monotone, it is a positive function and
	\[
		f(x/2) \geq \tfrac{1}{2} x|f'(x)|,
	\]
	which together with the doubling property, gives
	\[
		f(x)\geq C x|f'(x)|
	\]
	for $x > 2 x_0$. Hence, we obtain our assertion in the case $x_0 = 0$. If $x_0 > 0$ we observe that the function 
	\[
		[x_0, 2 x_0] \ni x \mapsto \frac{x |f'(x)|}{f(x)}
	\]
	is continuous and positive, thus bounded. This completes the proof.
\end{proof}
\begin{proposition}
	\label{prop:CM-WLSC}
	Let $f$ be a completely monotone function. Suppose that $-f'\in \WLSC{\tau}{c}{x_0}$ for some $c \in (0, 1]$,
	$x_0 \geq 0$, and $\tau \leq -1$. Then $f\in \WLSC{1 + \tau}{c}{x_0}$. 

	Analogously, if $-f' \in \WUSC{\tau}{C}{x_0}$ for some $C \geq 1$, $x_0 \geq 0$, and $\tau \leq -1$, then 
	$(f-f(\infty)) \in \WUSC{\tau}{C}{x_0}$.
\end{proposition}
\begin{proof}
	Let $\lambda>1$. For $y > x > x_0$, we have
	\begin{align*}
		f(\lambda x)-f(\lambda y)
		=
		-\int^{\lambda y}_{\lambda x} f'(s) \rmd s 
		&=
		-\lambda \int^y_x f'(\lambda s) \rmd s \\
		&\geq - c \lambda^{1+\tau} \int^y_x f'(s) \rmd s
		= c \lambda^{1+\tau} (f(x)-f(y)),
	\end{align*}
	thus
	\[
		f(\lambda x) \geq c \lambda^{1+\tau} f(x) + f(\lambda y) - c \lambda^{1+\tau} f(y).
	\]
	Since $f$ is nonnegative and nonincreasing, we can take $y$ approaching infinity to get
	\begin{align*}
		f(\lambda x) 
		&\geq c \lambda^{1+\tau} f(x) + \big(1 - c \lambda^{1+\tau}\big) \lim_{y \to \infty} f(y) \\
		&\geq c \lambda^{1+\tau} f(x),
	\end{align*}
	where in the last inequality we have also used that $1 \geq c \lambda^{1+\tau}$. The second part of
	the proposition can be proved in much the same way.
\end{proof}

\begin{proposition}
	\label{prop:WLSC}
	Let $\phi$ be a Bernstein function with $\phi(0)=0$. Then $\phi \in \WLSC{\alpha}{c}{x_0}$ for some
	$c \in (0, 1]$, $x_0 \geq 0$, and $\alpha >0$ if and only if $\phi' \in \WLSC{\alpha-1}{c'}{x_0}$ for some
	$c' \in (0,1]$. Furthermore, if $\phi \in \WLSC{\alpha}{c}{x_0}$ then there is $C \geq 1$ such that for all
	$x > x_0$,
	\begin{align}\label{eq:132}
		x \phi'(x) \leq \phi(x) \leq C x \phi'(x).
	\end{align}
\end{proposition}
\begin{proof}
	Assume first that $\phi'\in \WLSC{\alpha-1}{c}{x_0}$. Without loss of generality we can assume $\phi' \not\equiv 0$.
	We claim that \eqref{eq:132} holds true. In view of \eqref{eq:20}, it is enough to show that there is $C \geq 1$
	such that for all $x > x_0$,
	\[
		\phi(x) \leq C x \phi'(x).
	\]
	First, let us observe that, by the weak lower scaling property of $\phi'$,
	\begin{align}
		\nonumber
		\phi(x) -\phi(x_0) &= 
		\int_{x_0}^{x} \phi'(s) \rmd s \\
		\nonumber
		&\leq c^{-1} \phi'(x)  \int_{x_0}^x \big(s/x\big)^{-1+\alpha} \rmd s \\
		\label{eq:4}
		&\leq \frac{1}{c \alpha} x \phi'(x).
	\end{align}
	Thus we get the assertion in the case $x_0 = 0$. If $x_0 > 0$, it is enough to show that there is $C > 0$ such that 
	for all $x > x_0$,
	\begin{equation}
		\label{eq:3}
		x \phi'(x) \geq C.
	\end{equation}
	Since $\phi' \in \WLSC{\alpha-1}{c}{x_0}$, the function
	\[
		(x_0, \infty) \ni x \mapsto x \phi'(x)
	\]
	is almost increasing. Hence, for $x \geq 2 x_0$ we have
	\[
		x \phi'(x) \geq c 2 x_0 \phi'(2 x_0).
	\]
	To conclude \eqref{eq:3}, we notice that $\phi'(x)$ is positive and continuous in $[x_0, 2x_0]$.
	Now, by \eqref{eq:3} we get
	\[
		x \phi'(x) \geq C \phi(x_0)
	\]
	for all $x > x_0$, which together with \eqref{eq:4}, implies \eqref{eq:132} and the scaling property of $\phi$ follows.
	
	Now assume that $\phi \in \WLSC{\alpha}{c}{x_0}$. By monotonicity of $\phi'$, for $0 < s < t$,
	\[
	\frac{\phi(tx)-\phi(sx)}{\phi(x)} \leq \frac{x(t-s)\phi'(sx)}{\phi(x)}.
	\]
	For $s=1$, by the lower scaling we get
	\[
	\frac{x(t-1)\phi'(x)}{\phi(x)} \geq \frac{\phi(tx)}{\phi(x)}-1 \geq ct^{\alpha}-1,
	\]
	for all $x>x_0$. Thus, for $t = 2^{1/\alpha} c^{-1/\alpha}$, we obtain that 
	$x\phi'(x) \gtrsim \phi(x)$ for all $x>x_0$. Invoking \eqref{eq:20}, we conclude \eqref{eq:132}. In particular, $\phi'$ has the weak lower scaling property.
	This completes the proof.
\end{proof}

\begin{proposition}
	\label{prop:WUSC}
	Let $\phi$ be a Bernstein function. Suppose that $-\phi'' \in \WUSC{\beta-2}{C}{x_0}$ for some $C \geq 1$, $x_0 \geq 0$, and $\beta <1$. Then 
	for all $x > x_0$,
	\[
		\phi'(x) \leq \frac{C}{1-\beta} x(-\phi''(x)) + b.
	\]
\end{proposition}
\begin{proof}
	Without loss of generality we can assume $\phi'' \not\equiv 0$.	By the scaling property, for $x > x_0$ we have 
	\begin{align*}
		\frac{\phi'(x)-b}{x(-\phi''(x))} 
		&= 
		\int_x^{\infty} \frac{t}{x} \frac{(-\phi''(t))}{(-\phi''(x))} \frac{{\rm d} t}{t} \\
		&\leq 
		C \int_x^{\infty} \bigg( \frac{t}{x} \bigg)^{-1+\beta} \frac{{\rm d} t}{t}
		= C \frac{1}{1-\beta},
	\end{align*}
	which concludes the proof.
\end{proof}

\begin{remark}
	\label{rem:2}
	Let $\phi$ be a Bernstein function such that $\phi(0) = 0$. Suppose that $-\phi'' \in \WLSC{\alpha-2}{c}{x_0}$, 
	for some $c \in (0, 1]$, $x_0 \geq 0$, and $\alpha \in (0, 1]$. Since $\phi'$ is completely monotone, by Proposition
	\ref{prop:CM-WLSC}, $\phi' \in \WLSC{\alpha-1}{c}{x_0}$. Therefore, by Proposition \ref{prop:WLSC},
	we conclude that $\phi \in \WLSC{\alpha}{c_1}{x_0}$ for some $c_1 \in (0, 1]$.
\end{remark}

\begin{proposition}\label{prop:11}
	Let $f$ be a completely monotone function. Suppose that 
	\[
		f \in \WLSC{\alpha-1}{c}{x_0} \cap \WUSC{\beta-1}{C}{x_0}
	\]
	for some $c \in (0,1]$, $C \geq 1$, $x_0 \geq 0$ and $0<\alpha \leq \beta <1$. Then
	\[
		-f' \in \WLSC{\alpha-2}{c'}{x_0} \cap \WUSC{\beta-2}{C'}{x_0}
	\]
	for some $c' \in (0,1]$ and $C' \geq 1$.
\end{proposition}
\begin{proof}
	By monotonicity of $-f'$, for $0 < s < t$,
	\begin{equation}
		\label{eq:2}
		\frac{-x(t-s)f'(tx)}{f(x)} \leq \frac{f(sx)-f(tx)}{f(x)} \leq \frac{-x(t-s)f'(sx)}{f(x)}.
	\end{equation}
	Taking $s=1$ in the second inequality, the weak upper scaling property yields
	\[
	\frac{-x(t-1)f'(x)}{f(x)} \geq 1- \frac{f(tx)}{f(x)} \geq 1-ct^{\beta-1},
	\]
	for all $x>x_0$. By selecting $t > 1$ such that $ct^{\beta-1} \leq \frac12$, we obtain 
	$x\big(-f'(x)\big) \gtrsim f(x)$ for $x>x_0$. Similarly, taking $t=1$ in the first inequality in \eqref{eq:2},
	by the weak lower scaling property we get
	\[
		\frac{-x(1-s)f'(x)}{f(x)} \leq \frac{f(sx)}{f(x)}-1 \leq c^{-1}s^{\alpha-1}-1,
	\]
	for all $x>x_0/s$. By selecting $0 < s < 1$ such that $s^{\alpha-1}\geq 2c$, we obtain 
	$x\big( -f'(x)\big) \lesssim f(x)$ for $x>x_0/s$. Hence,
	\begin{align}
		\label{eq:130}
		f(x) \approx x\big( -f'(x) \big),
	\end{align}
	for all $x>x_0/s$. Therefore, lower and upper scaling properties follow from \eqref{eq:130} and the scaling properties
	of $f$. This finishes the proof for the case $x_0 = 0$. If $x_0 > 0$, we notice that both $f$ ad $-f'$ are positive
	and continuous, thus at the possible expense of worsening the constants we get \eqref{eq:130} for all $x>x_0$. 
\end{proof}

Now, by combining Propositions \ref{prop:WLSC} and \ref{prop:11}, we immediately get the following corollary.
\begin{corollary}\label{cor:6}
	Let $\phi$ be a Bernstein function such that $\phi(0)=0$.  Suppose that 
	\[
		\phi \in \WLSC{\alpha}{c}{x_0} \cap \WUSC{\beta}{C}{x_0}
	\]
	for some $c \in (0,1]$, $C \geq 1$, $x_0 \geq 0$ and $0<\alpha \leq \beta <1$. Then 
	\[
		-\phi'' \in \WLSC{\alpha-2}{c'}{x_0} \cap \WUSC{\beta-2}{C'}{x_0}
	\]
	for some $c' \in (0,1]$ and $C' \geq 1$.
\end{corollary}

\begin{lemma}\label{lem:2}
	Let $\phi$ be a Bernstein function. Suppose that $-\phi'' \in \WLSC{\alpha-2}{c}{x_0}$ for some $c \in (0, 1]$, $x_0 \geq 0$, and $\alpha >0$.
	There is a constant $C > 0$ such that for all $x > x_0$,
	\[
		C (-\phi''(x)) \leq \int_{(0,1/x)}s^2\mu(ds).
	\]
	Moreover, the constant $C$ depends only on $\alpha$ and $c$.
\end{lemma}

\begin{proof}
	Let  $f\colon (0, \infty) \rightarrow \RR$ be a function defined as
	\begin{align*}
		f(t)=\int_{(0, t)} s^2 \: \mu ({\rm d} s).
	\end{align*}
	We observe that, by the Fubini--Tonelli theorem, for $x > 0$ we have
	\begin{align*}
		\calL f(x) 
		&= \int_0^\infty e^{-x t} \int_{(0, t)} s^2 \: \mu({\rm d} s) \rmd t \\
		&= \int_{(0, \infty)} s^2 \int_s^\infty e^{-x t} \rmd t \: \mu({\rm d} s)
		= x^{-1} (-\phi''(x)).
	\end{align*}
	Since $f$ is nondecreasing, for any $s > 0$,
	\begin{align*}
		-\phi''(x)=x \calL f(x) 
		&\geq 
		\int_s^{\infty} e^{-t} f\big( t/x \big) \rmd t\\
		&\geq e^{-s} f\big(s /x \big).
	\end{align*}
	Hence, for any $u > 2$,
	\begin{align*}
		-\phi''(x)
		&=\int_0^u e^{-s} f\big( s/x \big) \rmd s 
		+ \int_u^{\infty} e^{-s} f\big(s/x \big) \rmd s \\
		&\leq f\big(u / x \big)
		+ \int_u^{\infty} e^{-s/2} (-\phi''(x/2) ) \rmd s.
	\end{align*}
	Therefore, setting $x = \lambda u > 2 x_0$, by the weak lower scaling property of $-\phi''$,
	\begin{align*}
		f(1/\lambda)
		&\geq -\phi''(u \lambda) - 2 e^{-u/2} (-\phi''(u \lambda/2)) \\
		&\geq \big(2^{\alpha-2}c - 2 e^{-u/2}\big) (-\phi''(u \lambda/2)).
	\end{align*}
	At this stage, we select $u > 2$ such that
	\[
		2^{\alpha-2}c - 2 e^{-u/2} \geq 2^{-2} c.
	\]
	Then again, by the weak lower scaling property of $-\phi''$, for $\lambda>x_0$,
	\begin{align*}
		f(1/\lambda)
        &\geq c 2^{-2} (-\phi''(u \lambda/2))\geq c^2 2^{-\alpha} u^{\alpha-2} (-\phi''(\lambda)),
	\end{align*}
	which ends the proof.
\end{proof}
\begin{lemma}\label{lem:4}
	Let $\phi$ be a Bernstein function. Suppose that $-\phi'' \in \WLSC{\alpha-2}{c}{x_0}$ for some $c \in (0, 1]$, $x_0 \geq 0$, and $\alpha >0$.
	Then there exists a complete Bernstein function $f$ such that $f\approx \phi$ for $x>0$, and $-f''\approx -\phi''$ for
	$x>x_0$.
\end{lemma}
\begin{proof}
	Let us define
	\[
		f(\lambda)=a+b\lambda+\int^\infty_0\frac{\lambda u}{\lambda u +1} \mu({\rm d} u).
	\]
	By \cite[Theorem 6.2 (vi)]{MR2978140} the function $f$ is a complete Bernstein function. Since for $y > 0$,
	\[
		\frac{y}{y +1} \approx \big(1-e^{-y}\big),
	\]
	we get $f(\lambda) \approx \phi(\lambda)$. Moreover,
	\[
		f''(\lambda)=-2\int^\infty_0\frac{u^2}{\lambda u+1}\mu({\rm d}u).
	\]
	Hence, by Lemma \ref{lem:2} we obtain $-f''(\lambda)\approx -\phi''(\lambda)$ for $\lambda>x_0$.
\end{proof}

\section{Asymptotic behavior of densities}
\label{sec:1}
Let $\bfT = (T_t \colon t \geq 0)$ be a subordinator with the L\'{e}vy--Khintchine exponent $\psi$ and the Laplace
exponent $\phi$. Since $\phi$ is a Bernstein function, it admits the integral representation \eqref{eq:1}. As it may be
easily checked (see e.g. \cite[Proposition 3.6]{MR2978140}), we have $\mu = \nu$, $a = 0$, and $\psi(\xi) = \phi(-i\xi)$.
In particular, $\phi(0)=0$. 

In this section we study the asymptotic behavior of the probability density of $T_t$. In the whole section we assume that
$\phi''\not \equiv 0$, otherwise $T_t=b t$ is deterministic.  The main result is Theorem \ref{thm:3}. Let us start by
showing an estimate on the real part of the complex extension $\phi$.
\begin{lemma}
	\label{lem:1}  
	Suppose that $-\phi'' \in \WLSC{\alpha-2}{c}{x_0}$ for some $c \in (0, 1]$, $x_0 \geq 0$, and $\alpha >0$.
	Then there exists $C > 0$ such that for all $w > x_0$ and $\lambda \in \RR$,
	\[
		\Re \big( \phi(w+i\lambda)-\phi(w) \big) 
		\geq C \lambda^2 \big(-\phi''(|\lambda| \vee w)\big).
	\]
\end{lemma}
\begin{proof}
	By the integral representation \eqref{eq:1}, for $\lambda \in \RR$ we have
	\[
		\Re \big( \phi(w+i\lambda)-\phi(w)\big)
        = \int_{(0, \infty)} \big(1-\cos (\lambda s)\big) e^{-w s} \: \nu ({\rm d} s).
	\]
	In particular,
	\[
		\Re \big( \phi(w+i\lambda)-\phi(w)\big) =
		\Re \big( \phi(w-i\lambda)-\phi(w)\big).
	\]
	Thus it is sufficient to consider $\lambda > 0$. We can estimate	
	\begin{align}
		\nonumber
		\Re \big( \phi(w+i\lambda)-\phi(w)\big)
		&\geq
		\int_{(0, 1/\lambda)} \big(1-\cos (\lambda s)\big) e^{-w s} \: \nu ({\rm d} s) \\
		\label{eq:23}
		&\gtrsim \lambda^2 \int_{(0, 1/\lambda)} s^2 e^{-w s} \: \nu({\rm d} s).
	\end{align}
		Due to Lemma \ref{lem:2} we obtain, for $\lambda \geq w$, 
	\begin{align*}
		\Re \big( \phi(w+i\lambda)-\phi(w)\big)
		&\gtrsim
		\lambda^2 \int_{(0, 1/\lambda)} s^2  \: \nu({\rm d} s)\gtrsim
		\lambda^2 (-\phi''(\lambda)).
	\end{align*}
	If $w > \lambda > 0$ then, by \eqref{eq:23}, we have
	\begin{align*}
		\Re \big( \phi(w+i\lambda)-\phi(w)\big)
		&\gtrsim \lambda^2 \int_{(0, 1/w)} s^2 e^{-w s} \mu({\rm d} s) \\
		&\geq e^{-1} \lambda^2 \int_{(0,1/w)} s^2 \mu({\rm d} s)
	\end{align*}
	which, by Lemma \ref{lem:2}, completes the proof.
\end{proof}

\begin{remark}
	\label{rem:1}
	Suppose that $-\phi'' \in \WLSC{\alpha-2}{c}{x_0}$ for some $c \in (0, 1]$, $x_0 \geq 0$, and $\alpha>0$.
	Since
	\[
		K(1/x) \leq ex^2 (-\phi''(x)),
	\]
	by Lemma \ref{lem:2}, we obtain
	\[
		C x^2 (-\phi''(x)) \leq K(1/x) \leq ex^2 (-\phi''(x))
	\]
	for all $x > x_0$.
\end{remark}

\begin{theorem}
	\label{thm:3}
	Let $\bfT$ be a subordinator with the Laplace exponent $\phi$. Suppose that $-\phi'' \in \WLSC{\alpha-2}{c}{x_0}$
	for some $c \in (0, 1]$, $x_0 \geq 0$, and $\alpha>0$. Then the probability distribution of $T_t$ is absolutely
	continuous for all $t>0$. If we denote its density by $p(t, \: \cdot \:)$, then for each $\epsilon > 0$ there is
	$M_0 > 0$ such that 
	\[
		\left\lvert 
		p\big(t,t\phi'(w)\big) \sqrt{2\pi t(-\phi''(w))}
		\exp{\Big\{t \big( \phi(w)-w \phi'(w) \big) \Big\}} - 1 \right\rvert 
		\leq \epsilon,
	\]
	provided that $w > x_0$ and $t w^2 (-\phi''(w)) > M_0$.
\end{theorem}
\begin{proof}
	Let $x=t \phi'(w)$ and $M > 0$. We first show that
	\begin{equation}
		\label{eq:10}
		p(t, x) =
		\frac{1}{2\pi} \cdot \frac{e^{-t \Phi(x/t,0)}}{\sqrt{t (-\phi''(w))}}
        \int_\RR
        \exp{\Bigg\{-t \Bigg( \Phi \bigg(\frac{x}{t}, \frac{u}{\sqrt{t (-\phi''(w))}} \bigg)
        - \Phi\bigg(\frac{x}{t}, 0\bigg) \Bigg) \Bigg\}} \rmd u,
	\end{equation}
	provided that $w > x_0$ and $t w^2 (-\phi''(w)) > M$, where for $\lambda \in \RR$ we have set
	\begin{equation}
		\label{eq:7}
		\Phi \big( x/t, \lambda \big) = \phi(w+i\lambda)-\frac{x}{t}(w+i\lambda).
	\end{equation}
	To do so, let us recall that
	\[
		\EE\big(e^{-\lambda T_t} \big) = e^{-t \phi(\lambda)}, \quad \lambda \geq 0.
	\]
	Thus, by the Mellin's inversion formula, if the limit
	\begin{equation}
		\label{eq:12}
		\lim_{L \to \infty}
		\frac{1}{2\pi i } \int_{w-i L}^{w+i L} e^{-t \phi(\lambda)+\lambda x} \rmd \lambda
		\quad\text{exists,}
	\end{equation}
	then the probability distribution of $T_t$ has a density $p(t, \: \cdot \:)$ and
	\[
		p(t, x) = \lim_{L \to \infty} \frac{1}{2\pi i } \int_{w-i L}^{w+i L} e^{-t \phi(\lambda)+\lambda x} \rmd \lambda.
	\]
	Therefore, our task is to justify the statement \eqref{eq:12}. For $L > 0$, we write
	\begin{align*}
		\frac{1}{2 \pi i}
		\int_{w-i L}^{w+i L} e^{-t \phi(\lambda)+\lambda x} \rmd \lambda
		= \frac{1}{2 \pi} \int_{-L}^{L} e^{-t \Phi\left( x/t,\lambda \right)} \rmd \lambda.
	\end{align*}
	By the change of variables
	\[
		\lambda = \frac{u}{\sqrt{t (-\phi''(w))}},
	\]
	we obtain
	\begin{align*}
		\int_{-L}^{L} 
		e^{-t\Phi\left(x/t,\lambda\right)} \rmd \lambda 
		&= 
		e^{-t \Phi\left(x/t,0\right)} 
		\int_{-L}^{L} 
		\exp{\Big\{-t \Big( \Phi \big(x/t, \lambda \big)-\Phi\big(x/t, 0\big) \Big)\Big\}}
		\rmd \lambda \\ 
		&= 
		\frac{e^{-t \Phi(x/t,0)}}{\sqrt{t (-\phi''(w))}}  
		\int_{-L \sqrt{t(-\phi''(w))}}^{L \sqrt{t(-\phi''(w))}} 
		\exp{\Bigg\{-t \Bigg( \Phi \bigg(\frac{x}{t}, \frac{u}{\sqrt{t (-\phi''(w))}} \bigg) 
		- \Phi\bigg(\frac{x}{t}, 0\bigg) \Bigg) \Bigg\}} \rmd u.
	\end{align*}
	Let us note here that $-\phi''$ is nonincreasing and integrable at infinity, thus, we in fact have $\alpha \leq 1$. We claim that there is $C > 0$ not depending on $M$, such that for all $u \in \RR$,
	\begin{equation}
		\label{eq:9}
		t \Re \Bigg( \Phi \bigg(\frac{x}{t}, \frac{u}{\sqrt{t (-\phi''(w))}} \bigg)
        - \Phi\bigg(\frac{x}{t}, 0\bigg) \Bigg)
		\geq
		C \big( u^2 \wedge (\abs{u}^\alpha M^{1-\alpha/2})\big),
	\end{equation}
	provided that $w > x_0$ and $t w^2 (-\phi''(w)) > M$. Indeed, by \eqref{eq:7} and Lemma \ref{lem:1}, 
	for $w > x_0$ we get
	\begin{equation}
		\label{eq:8}
		t \Re \Bigg( \Phi \bigg(\frac{x}{t}, \frac{u}{\sqrt{t (-\phi''(w))}} \bigg)
        - \Phi\bigg(\frac{x}{t}, 0\bigg) \Bigg)
        \gtrsim
		\frac{\abs{u}^2}{\phi''(w)} 
		\phi''\bigg(\frac{\abs{u}}{\sqrt{t (-\phi''(w))}} \vee w\bigg).
	\end{equation}
	We next estimate the right-hand side of \eqref{eq:8}. If $\abs{u} \leq w \sqrt{t (-\phi''(w))}$, then
	\[
		\frac{\abs{u}^2}{\phi''(w)}
        \phi''\bigg(\frac{\abs{u}}{\sqrt{t (-\phi''(w))}} \vee w\bigg)
		=
		\abs{u}^2.
	\]
	Otherwise, since $-\phi'' \in \WLSC{\alpha-2}{c}{x_0}$, we obtain
	\begin{align*}
		\frac{\abs{u}^2}{\phi''(w)}
        \phi''\bigg(\frac{\abs{u}}{\sqrt{t (-\phi''(w))}} \vee w\bigg)
		&\geq
		c \abs{u}^2
		\bigg(\frac{\abs{u}}{\sqrt{t w^2 (-\phi''(w))}}\bigg)^{-2+\alpha} \\
		&=
		c \abs{u}^\alpha \big(t w^2 (-\phi''(w))\big)^{1-\alpha/2} \\
		&\geq
		c M^{1-\alpha/2} \abs{u}^\alpha.
	\end{align*}
	Hence, we deduce \eqref{eq:9}. To finish the proof of \eqref{eq:12}, we invoke the dominated convergence theorem.
	Consequently, by Mellin's inversion formula we obtain \eqref{eq:10}.

	Our next task is to show that for each $\epsilon > 0$ there is $M_0 > 0$ such that 
	\begin{equation}
		\label{eq:14}
		\Bigg|
		\int_\RR  
		\exp{\Bigg\{-t \Bigg( \Phi \bigg(\frac{x}{t}, \frac{u}{\sqrt{t (-\phi''(w))}} \bigg)
        - \Phi\bigg(\frac{x}{t}, 0\bigg) \Bigg) \Bigg\}}\rmd u
		-
		\int_\RR e^{-\frac{1}{2} u^2} \rmd u
		\Bigg|
		\leq \epsilon,
	\end{equation}
	provided that $w > x_0$ and $t w^2 (- \phi''(w)) > M_0$. In view of \eqref{eq:9}, by taking $M_0>1$ sufficiently
	large, we get
	\begin{equation}
		\label{eq:19a}
		\Bigg|
        \int_{\abs{u} \geq M_0^{1/4}}
        \exp{\Bigg\{-t \Bigg( \Phi \bigg(\frac{x}{t}, \frac{u}{\sqrt{t (-\phi''(w))}} \bigg)
        - \Phi\bigg(\frac{x}{t}, 0\bigg) \Bigg) \Bigg\}}
		\rmd u
		\Bigg|
		\leq
		\int_{\abs{u} \geq M_0^{1/4}}
		e^{-C |u|^\alpha} \rmd u \leq \epsilon,
	\end{equation}
	and
	\begin{equation}
		\label{eq:19b}
		\int_{\abs{u} \geq M_0^{1/4}}
		e^{-\frac{1}{2} u^2} \rmd u \leq \epsilon.
	\end{equation}
	Next, we claim that there is $C > 0$ such that 
	\begin{equation}
		\label{eq:17}
		\bigg|
		t 
		\bigg( \Phi \bigg(\frac{x}{t}, \frac{u}{\sqrt{t (-\phi''(w))}} \bigg)
        - \Phi\bigg(\frac{x}{t}, 0\bigg)
		\bigg)
		-\frac{1}{2} \abs{u}^2
		\bigg|
		\leq
		C \abs{u}^3 M_0^{-\frac{1}{2}}.
	\end{equation}
	Indeed, since
	\[
		\partial_\lambda \Phi\(\frac{x}{t}, 0\) = 0,
	\]
	by the Taylor's formula, we get
	\begin{align}
		\nonumber
		\bigg|
		t 
		\bigg( \Phi \bigg(\frac{x}{t}, \frac{u}{\sqrt{t (-\phi''(w))}} \bigg)
        - \Phi\bigg(\frac{x}{t}, 0\bigg)
		\bigg)
		-\frac{1}{2} \abs{u}^2
		\bigg|
		&=
		\bigg|
		\frac{1}{2} \partial_\lambda^2\Phi \(\frac{x}{t}, \xi\) \frac{\abs{u}^2}{-\phi''(w)} - \frac{1}{2} \abs{u}^2
		\bigg| \\
		\label{eq:16}
		&=
		\frac{\abs{u}^2}{2\abs{\phi''(w)}}
		\big|\phi''(w + i \xi) - \phi''(w) \big|, 
	\end{align}
	where $\xi$ is some number satisfying
	\begin{equation}
		\label{eq:15}
		\abs{\xi} \leq \frac{\abs{u}}{\sqrt{t(-\phi''(w))}}.
	\end{equation}
	Observe that
	\begin{align*}
		\big| \phi''(w+i\xi)- \phi''(w) \big|
		&\leq 
		\int_{(0, \infty)} s^2 e^{-w s} \big|e^{-i \xi s}-1\big|\: \nu({\rm d} s) \\
		&\leq 2
		\abs{\xi} \int_{(0, \infty)} s^3 e^{-w s} \: \nu({\rm d} s) =2 \abs{\xi} \phi'''(w).
	\end{align*}
	Since $-\phi''$ is a nonincreasing function with the weak lower scaling property, it is doubling.
	Thus, by Proposition \ref{prop:2}, for $w > x_0$,
	\[
		-\phi''(w) \gtrsim w \phi'''(w),
	\]
	which together with \eqref{eq:15} give
	\begin{align}
		\nonumber
		\big|
		\phi''(w + i \xi) - \phi''(w)
		\big|
		&\leq
		C \frac{\abs{u}}{\sqrt{t (-\phi''(w))}} \cdot \frac{-\phi''(w)}{w} \\
		\label{eq:18}
		&\leq
		C M_0^{-\frac{1}{2}} \abs{u} (-\phi''(w)),
	\end{align}
	whenever $t w^2 (-\phi''(w)) > M_0$. Now, \eqref{eq:17} easily follows by \eqref{eq:18} and \eqref{eq:16}.

	Finally, since for any $z \in \CC$, 
	\[
		\big|e^z - 1\big| \leq \abs{z} e^\abs{z},
	\]
	by \eqref{eq:17}, we obtain
	\begin{align*}
		&
		\Bigg|
        \int_{\abs{u} \leq M_0^{1/4}}
        \exp{\Bigg\{-t \Bigg( \Phi \bigg(\frac{x}{t}, \frac{u}{\sqrt{t (-\phi''(w))}} \bigg)
        - \Phi\bigg(\frac{x}{t}, 0\bigg) \Bigg) \Bigg\}} 
		-
		e^{-\frac{1}{2} \abs{u}^2} \rmd u
		\Bigg| \\
		&\qquad\qquad\leq
		C M_0^{-\frac{1}{2}} \int_{\abs{u} \leq M_0^{1/4}}
		\exp\bigg\{-\tfrac{1}{2} \abs{u}^2 + C M_0^{-\frac{1}{2}} \abs{u}^3\bigg\} \abs{u}^3 \rmd u 
		\leq \epsilon,
	\end{align*}
	provided that $M_0$ is sufficiently large, which, together with \eqref{eq:19a} and \eqref{eq:19b}, completes the proof
	of \eqref{eq:14} and the theorem follows.
\end{proof}

\begin{remark}
	\label{rem:3}
	If $x_0 = 0$ then the constant $M_0$ in Theorem \ref{thm:3} depends only on $\alpha$ and $c$. If $x_0 > 0$
	it also depends on
	\[
		\sup_{x \in [x_0, 2 x_0]} \frac{x \phi'''(x)}{-\phi''(x)}.
	\]
\end{remark}

By Theorem \ref{thm:3}, we immediately get the following corollaries.
\begin{corollary}
	\label{cor:1}
	Suppose that $-\phi'' \in \WLSC{\alpha-2}{c}{x_0}$ for some $c \in (0, 1]$, $x_0 \geq 0$, and $\alpha >0$.
	Then there is $M_0 > 0$ such that
	\[
		p(t,x) 
		\approx 
		\frac{1}{\sqrt{t (-\phi'' (w))}} \exp\Big\{-t \big(\phi(w) - w \phi'(w)\big)\Big\},
	\]
	uniformly on the set
	\[
		\Big\{
		(t, x) \in \RR_+ \times \RR \colon tb < x < t \phi'(x_0^+) \text{ and } t w^2 (-\phi''(w)) > M_0
		\Big\}
	\]
	where $w = (\phi')^{-1}(x/t)$. 
\end{corollary}

\begin{corollary}
	\label{cor:5}
	Suppose that $-\phi'' \in \WLSC{\alpha-2}{c}{x_0}$ for some $c \in (0, 1]$, $x_0 \geq 0$, and $\alpha >0$. Assume
	also that $b = 0$. Then for any $x>0$,
	\[
		\lim_{t\to\infty} p(t,x) \sqrt{t (-\phi'' (w))} \exp\Big\{t \big(\phi(w) - w \phi'(w)\big)\Big\} = (2\pi)^{-1/2},
	\]
	where $w = (\phi')^{-1}(x/t)$. 
\end{corollary}

By imposing on $-\phi''$ an additional condition of the weak upper scaling, we can further simplify the description of
the set where the sharp estimates on $p(t, x)$ hold.
\begin{corollary}
	\label{cor:2}
	Suppose that $\phi \in \WLSC{\alpha}{c}{x_0} \cap \WUSC{\beta}{C}{x_0}$ for some 
	$c \in (0, 1]$, $C \geq 1$, $x_0 \geq 0$, and $0 < \alpha \leq \beta < 1$. Assume also that $b = 0$. Then there is
	$\delta > 0$ such that
	\[
		p(t, x) \approx
        \frac{1}{\sqrt{t (-\phi'' (w))}} \exp\Big\{-t \big(\phi(w) - w \phi'(w)\big)\Big\},
	\]
	uniformly on the set
	\begin{equation}
		\label{eq:5}
		\Big\{
		(t, x) \in \RR_+ \times \RR \colon
		0 < x \phi^{-1}(1/t) < \delta, \text{ and } 0 \leq t \phi(x_0) \leq 1
		\Big\}
	\end{equation}
	where $w = (\phi')^{-1}(x/t)$.
\end{corollary}
\begin{proof}
	By Proposition \ref{prop:WLSC}, there is $C_1 \geq 1$ such that for all $u > x_0$,
	\[
		\phi(u) \leq C_1 u \phi'(u),
	\]
	thus, for $(t, x)$ belonging to the set \eqref{eq:5},
	\begin{align}
		\nonumber
		\frac{x}{t} 
		< \delta \frac{1}{t \phi^{-1}(1/t)} 
		&= \delta \frac{\phi\big(\phi^{-1}(1/t) \big)}{\phi^{-1}(1/t)} \\
		\label{eq:75}
		&\leq C_1 \delta \phi'\big(\phi^{-1}(1/t)\big).
	\end{align}
	By Proposition \ref{prop:WLSC}, $\phi' \in \WLSC{-1+\alpha}{c}{x_0}$, hence for all $D \geq 1$,
	\[
		\phi'\big(D \phi^{-1}(1/t)\big)
		\geq 
		c D^{-1+\alpha} \phi'\big(\phi^{-1}(1/t)\big).
	\]
	By taking $\delta$ sufficiently small, we get
	\[
		D = \bigg(\frac{c}{C_1 \delta}\bigg)^{\frac{1}{1-\alpha}} \geq 1,
	\]
	thus, by \eqref{eq:75}, we obtain
	\[
		\frac{x}{t} <  \phi'\big(D \phi^{-1}(1/t) \big),
	\]
	which implies that
	\begin{equation}
		\label{eq:76}
		w = (\phi')^{-1} (x/t) > D \phi^{-1}(1/t).
	\end{equation}
	In particular, $w > x_0$. On the other hand, by Propositions \ref{prop:WLSC} and \ref{prop:WUSC}, there is 
	$c_1 \in (0,1]$ such that
	\[
		t w^2 (-\phi''(w)) \geq c_1 t \phi(w).
	\]
	By Remark \ref{rem:2}, $\phi \in \WLSC{\alpha}{c_2}{x_0}$ for some $c_2 \in (0,1]$. Therefore,
	\[
		t\phi(w) = \frac{\phi(w)}{\phi \big( \phi^{-1}(1/t) \big)}
		\geq c_2 \Bigg( \frac{w} {\phi^{-1}(1/t)} \Bigg)^{\alpha},
	\]
	which together with \eqref{eq:76}, gives
	\[
		t w^2 (-\phi''(w)) \gtrsim \delta^{-\frac{\alpha}{1-\alpha}} > M_0,
	\]
	for $\delta$ sufficiently small. Hence, by Corollary \ref{cor:1}, we conclude the proof.
\end{proof}

The following proposition provides a sufficient condition on the measure $\nu$ that entails the weak lower scaling 
property of $-\phi''$, and allows us to apply Theorem \ref{thm:3}.
\begin{proposition}
	\label{prop:8}
	Suppose that there are $x_0 \geq 0$, $C \geq 1$ and $\alpha > 0$ such that for all $0 < r < 1/x_0$ and
	$0 < \lambda \leq 1$,
	\begin{equation}
		\label{eq:91}
		\nu\big((r,\infty)\big) \leq C\lambda^{\alpha}\nu\big((\lambda r,\infty)\big).
	\end{equation}
	Then $-\phi'' \in \WLSC{\alpha-2}{c}{x_0}$ for some $c \in (0, 1]$. 
\end{proposition}
\begin{proof}
	Let us first notice that by the Fubini--Tonelli theorem,
	\begin{align*}
		h(r) 
		&= r^{-2} \int_{(0, \infty)} \min\big\{r^2, s^2\big\} \nu({\rm d} s) \\
		&= r^{-2} \int_0^r t \nu\big((t,\infty)\big) \rmd t.
	\end{align*}
	Thus, by \eqref{eq:91}, for all $0 < r < 1/x_0$ and $0 < \lambda \leq 1$,
	\begin{align}
		\nonumber
		C\lambda^{\alpha}h(\lambda r) 
		&= 
		\frac{2C\lambda^{\alpha}}{r^2} \int_0^r t\nu\big((\lambda t,\infty)\big) \rmd t \\
		\nonumber 
		&\geq \frac{2}{r^2}\int_0^r t\nu\big((t,\infty)\big) \rmd t \\ 
		\label{eq:113} 
		&=h(r).
	\end{align}
	Hence, by \cite[Lemma 2.3]{GS2019}, there is $C' \geq 1$ such that for all $0 < r <1/x_0$,
	\begin{align}
		\label{eq:116}
		K(r) \leq h(r) \leq C'K(r).
	\end{align}
	The integral representation of $\phi$ entails that
	\[
		e^{-1}x^{-2}K(1/x) \leq -\phi''(x) \leq e^2 2^{-2}x^{-2}h(1/x), \quad x>0,
	\]
	thus, by \eqref{eq:116}, we obtain
	\[
		-\phi''(x) \approx x^{-2}h(1/x),
	\]
	for all $x > x_0$. Now, the weak lower scaling property of $-\phi''$ is a consequence of \eqref{eq:113}.
\end{proof}

\section{Estimates on the density}
\label{sec:4}
Let $\bfT = (T_t \colon t \geq 0)$ be a subordinator with the L\'{e}vy--Khintchine exponent $\psi$ and the Laplace exponent
$\phi$. In this section we always assume that $-\phi'' \in \WLSC{\alpha-2}{c}{x_0}$ for some $c \in (0, 1]$, $x_0 \geq 0$,
and $\alpha \in (0, 1]$. In particular, by Theorem \ref{thm:3}, the probability distribution of $T_t$ has a density
$p(t, \: \cdot \:)$. To express the majorant on $p(t, \: \cdot \:)$, it is convenient to set
\[
	\vphi(x) = x^2 (-\phi''(x)), \quad x > 0.
\]
Obviously, $\vphi \in \WLSC{\alpha}{c}{x_0}$. Let $\vphi^{-1}$ denote the generalized inverse function defined as
\[
	\vphi^{-1}(x) = \sup \big\{r > 0 \colon \vphi^*(r) = x \big\}
\]
where
\[
	\vphi^*(r) = \sup_{0 < x \leq r} \vphi(x).
\]
We start by showing comparability between the two concentration functions $K$ and $h$ defined in \eqref{eq:30} and 
\eqref{eq:31}, respectively.
\begin{proposition}
	\label{prop:3}
	Suppose that $-\phi'' \in \WLSC{\alpha-2}{c}{x_0}$ for some $c \in (0, 1]$, $x_0 \geq 0$, and $\alpha>0$.
	Then there is $C \geq 1$ such that for all $0 < r < 1/x_0$,
	\[
		K(r) \leq h(r) \leq C K(r).
	\]
\end{proposition}
\begin{proof}
	Since $h(r) \geq K(r)$, it is enough to show that for some $C \geq 1$ and $0 < r < 1/x_0$,
	\[
		h(r) \leq C K(r).
	\]
	In view of \eqref{eq:25}, we have
	\begin{equation}
		\label{eq:26}
		h(r) = 2 \int_r^\infty K(s) \frac{{\rm d} s}{s} =
		2 \int_r^{1/x_0} K(s) \frac{{\rm d} s}{s} + 
		2 \int_{1/x_0}^{\infty} K(s) \frac{{\rm d} s}{s}.
	\end{equation}
	Let us consider the first term on the right-hand side of \eqref{eq:26}. By Remark \ref{rem:1} we have
	$K(r) \approx \varphi(1/r)$,  for $0<r<1/x_0$, which implies
	\begin{equation*}
		\int_r^{1/x_0} K(s) \frac{{\rm d} s}{s} \lesssim K(r), \quad 0<r<1/x_0.
	\end{equation*}
	This finishes the proof in the case $x_0 = 0$. If $x_0 > 0$ then, for $1/(2x_0)\leq r<1/x_0$, we have
	\[
		K(r) \gtrsim  \vphi(1/r) \gtrsim  \vphi(x_0)>0.
	\]
	Hence, $K(r) \gtrsim 1$ for all $0 < r < 1/x_0$. Since the second term on the right-hand side of \eqref{eq:26}
	is constant, the proof is completed.
\end{proof}

Let us notice that by \eqref{eq:27}, Proposition \ref{prop:3} and Remark \ref{rem:1}, we have
\begin{equation}
	\label{eq:33}
	\psi^*(x) \approx h(1/x) \approx K(1/x) \approx \varphi(x)
\end{equation}
for all $x > x_0$. In particular, there is $c_1 \in (0, 1]$ such that $\psi^* \in \WLSC{\alpha}{c_1}{x_0}$. 
Moreover,
\begin{align*}
    \psi^*(x) \lesssim K(1/x) 
	&= x^2 \int_{(0, 1/x)} s^2 \: \nu({\rm d} s) \\
    &\lesssim \int_{(0, 1/x)} \big(1 - \cos s x\big) \: \nu({\rm d} s),
\end{align*}
thus, for all $x > x_0$,
\begin{equation}
	\label{eq:34}
	\psi^*(x) \lesssim \Re \psi(x).
\end{equation}
Since for $\lambda \geq 1$ and $x > 0$,
\begin{equation}
	\label{eq:59}
	\vphi(\lambda x) \leq \lambda^2 \vphi(x),
\end{equation}
we get
\begin{equation}
	\label{eq:85}
	\vphi^*(\lambda x) \leq \lambda^2 \vphi^*(x).
\end{equation}
\begin{proposition}
	\label{prop:6}
	Suppose that $-\phi'' \in \WLSC{\alpha-2}{c}{x_0}$ for some $c \in (0, 1]$, $x_0 \geq 0$, and $\alpha>0$.
	Then for all $r > 2 h(1/x_0)$,
	\begin{equation}
		\label{eq:41}
		\frac{1}{h^{-1}(r)} \approx \psi^{-1}(r).
	\end{equation}
	Furthermore, there is $C \geq 1$ such that for all $\lambda \geq 1$, and $r > 2 h(1/x_0)$,
	\[
		\psi^{-1}(\lambda r) \leq C \lambda^{1/\alpha} \psi^{-1}(r).
	\]
\end{proposition}
\begin{proof}
	By \cite[(5.1)]{GS2017}, we have
	\[
		\frac{1}{h^{-1}(r/2)} \leq \psi^{-1}(r) \leq \frac{1}{h^{-1}(24 r)}
	\]
	for all $r >0$. On the other hand, by Proposition \ref{prop:3} and \cite[Lemma 2.3]{GS2019}, there is
	$C \geq 1$ such that for all $\lambda \geq 1$ and $r > h(1/x_0)$,
	\begin{equation}
		\label{eq:100}
		\frac{1}{h^{-1}(\lambda r)} \leq C \lambda^{1/\alpha} \frac{1}{h^{-1}(r)}.
	\end{equation}
	Hence, for $r > 2 h(1/x_0)$,
	\begin{equation}
		\label{eq:101}
		C^{-1} 2^{-1/\alpha}
		\frac{1}{h^{-1}(r)} 
		\leq \psi^{-1}(r) 
		\leq  C (24)^{1/\alpha} \frac{1}{h^{-1}(r)},
	\end{equation}
	proving \eqref{eq:41}. The weak upper scaling property of $\psi^{-1}$ is a consequence of 
	\eqref{eq:100} and \eqref{eq:101}.
\end{proof}

\begin{proposition}
	\label{prop:7}
	Suppose that $-\phi'' \in \WLSC{\alpha-2}{c}{x_0}$ for some $c \in (0, 1]$, $x_0 \geq 0$, and $\alpha>0$.
	Then for all $x > x_0$, 
	\begin{equation}
		\label{eq:42}
		\psi^*(x) \approx \vphi^*(x),
	\end{equation}
	and for all $r > \vphi(x_0)$,
	\begin{equation}
		\label{eq:102}
		\psi^{-1}(r) \approx \vphi^{-1}(r).
	\end{equation}
	Furthermore, there is $C \geq 1$ such that for all $\lambda \geq 1$ and $r > \vphi(x_0)$,
	\[
		\vphi^{-1}(\lambda r) \leq C \lambda^{1/\alpha} \vphi^{-1}(r).
	\]
\end{proposition}
\begin{proof}
	We start by showing that there is $C \geq 1$ such that for all $x > x_0$,
	\begin{equation}
        \label{eq:24}
		C^{-1} \psi^*(x) \leq \vphi^*(x) \leq C \psi^*(x).
	\end{equation}
	The first inequality in \eqref{eq:24} immediately follows from \eqref{eq:33}. If $x_0 = 0$ then the second
	inequality is also the consequence of \eqref{eq:33}. In the case $x_0 > 0$, we observe that
	for $x > x_0$, we have 
	\begin{align*}
		\vphi^*(x) 
		&= \max{\Big\{\sup_{0 < y \leq x_0} \vphi(y), \sup_{x_0 \leq y \leq x} \vphi(y)\Big\}} \\
		&\lesssim \max{\big\{\vphi^*(x_0), \psi^*(x)\big\}} \\
		&\leq \bigg(1 + \frac{\vphi^*(x_0)}{\psi^*(x_0)}\bigg)\psi^*(x),
	\end{align*}
	proving \eqref{eq:24}. 
	
	Now, using \eqref{eq:24}, we easily get
	\[
		\psi^{-1}(C^{-1} r) \leq \varphi^{-1}(r) \leq \psi^{-1}(C r)
	\]
	for all $r > C \psi^*(x_0)$. Hence, by Proposition \ref{prop:6},
	\[
		\vphi^{-1}(r) \approx \psi^{-1}(r)
	\]
	for $r > C \max{\big\{\psi^*(x_0), 2 h(1/x_0)\big\}}$. Finally, since both $\psi^{-1}$ and $\varphi^{-1}$ are positive
	and continuous, at the possible expense of worsening the constant, we can extend the area of comparability to conclude
	\eqref{eq:102}. Now, the scaling property of $\varphi^{-1}$ follows by \eqref{eq:102} and Proposition \ref{prop:6}.
\end{proof}

\begin{remark}\label{rem:5}
	Note that, alternatively, one can define the (left-sided) generalized inverse
	\[
	\vphi_{-1}(x) = \inf \{ r>0\colon \vphi_*(r)=x \},
	\]
	where
	\[
	\vphi_*(r) = \inf_{r \leq x} \vphi(x).
	\]
	In such case we have
	\[
	\vphi_* \big( \vphi_{-1}(s) \big) = s, \qquad \text{and} \qquad \vphi_{-1} \big( \vphi_*(s) \big) \leq s.
	\]
	Clearly, for all $x>0$,
	\[
	\vphi_*(x) \leq \vphi(x) \leq \vphi^*(x).
	\]
	Let $u>x_0$ and set
	\[
	r_0 = \inf \{ r>0\colon \vphi^*(r)=u \}.
	\]
	By Proposition \ref{prop:7}, $\vphi^* \in \WLSC{\alpha}{c}{x_0}$ for some $c \in (0,1]$ and $x_0 \geq 0$. Thus, for $\lambda> c^{-1/\alpha}$, we get $\vphi^*(\lambda r_0)> \vphi^*(r_0)$. It follows that for all $u>x_0$,
	\begin{align*}
	\sup \{ r>0\colon \vphi^*(r)=u \} &\leq \lambda \inf \{ r>0\colon \vphi^*(r)=u \} \\
	&\leq \lambda \inf \{ r>0\colon \vphi_*(r)=u \}.
	\end{align*}
	Thus, for all $r>x_0$,
	\[
	\vphi^{-1} \big( \vphi^*(r) \big) \lesssim r.
	\]
\end{remark}

\begin{corollary}
	\label{cor:3}
	Suppose that $-\phi'' \in \WLSC{\alpha-2}{c}{x_0}$ for some $c \in (0, 1]$, $x_0 \geq 0$, and $\alpha>0$.
	Then there is $C > 0$ such that for all $x > x_0$,
	\begin{equation}
		\label{eq:21}
		\big(\phi(x) - x \phi'(x) \big) 
\leq C \vphi(x).
	\end{equation}
\end{corollary}
\begin{proof}
	We have
	\[
		\big(\phi(x) - x \phi'(x) \big) - \big(\phi(x_0) - x_0 \phi'(x_0^+)\big) 
		= \int_{x_0}^x \vphi(u) \frac{{\rm d} u}{u}
		= \int_{x_0/x}^1 \vphi(x u) \frac{{\rm d} u}{u} 
	\]
	where
	\[
	\phi'(x_0^+) = \lim_{x \to x_0^+}\phi'(x).
	\]
	By the weak lower scaling property of $\vphi$, for any $x_0/x < u \leq 1$, we have
	\[
		\vphi(x) \geq c u^{-\alpha} \vphi(x u),
	\]
	thus
	\[
		\big(\phi(x) - x \phi'(x) \big) - \big(\phi(x_0) - x_0 \phi'(x_0^+)\big)
		\lesssim
		\vphi(x) \int_0^1 u^{\alpha-1} \rmd u,
	\]
	which proves \eqref{eq:21} if $x_0=0$. For $x_0>0$ one can use continuity and positivity of $\varphi$.
\end{proof}

\begin{proposition}
	\label{prop:1}
	Suppose that $-\phi'' \in \WLSC{\alpha-2}{c}{x_0} \cap \WUSC{\beta-2}{C}{x_0}$ for some $c \in (0, 1]$, $C \geq 1$,
	$x_0 \geq 0$, and $0 < \alpha \leq \beta < 1$. Assume also that $b = 0$. Then for all $x > x_0$,
	\begin{equation}
		\label{eq:99}
		\vphi^*(x) \approx \phi(x),
	\end{equation}
	and for all $r >\vphi(x_0)$,
	\begin{equation}
		\label{eq:95}
		\vphi^{-1}(r) \approx \phi^{-1}(r).
	\end{equation}
	Furthermore, there is $c' \in (0, 1]$ such that for all $\lambda \geq 1$ and $r > 1/\vphi^*(x_0)$,
	\begin{equation}\label{eq:206}
		\vphi^{-1}(\lambda r) \geq c' \lambda^{1/\beta} \vphi^{-1}(r).
	\end{equation}
\end{proposition}
\begin{proof}
	Let us observe that, by \eqref{eq:20}, Proposition \ref{prop:WLSC} and Proposition \ref{prop:WUSC}, there is
	$c_1 \in (0, 1]$ such that for all $x > x_0$, 
	\begin{equation}
		\label{eq:90}
		2 \phi(x) \geq \vphi(x) \geq c_1 \phi(x).
	\end{equation}
	Now the proof of the lemma is similar to the proof of Proposition \ref{prop:7} therefore it is omitted. 
\end{proof}

\subsection{Estimates from above}
In this section we show the upper estimates on $p(t, \: \cdot \:)$. Before embarking on the proof let us
introduce some notation. Given a set $B \subset \RR$, we define
\[
	\delta(B) = \inf\big\{\abs{x} \colon x \in B \big\},
\]
and
\[
	\diam(B) = \sup\big\{\abs{x-y} \colon x, y \in B \big\}.
\]
Let
\[
	b_r = b + \int_{(0, r)} s \: \nu({\rm d} s), \quad r > 0.
\]
In view of \eqref{eq:13a}, the above definition of $b_r$ is in line with the usual one (see e.g. \cite[formula (4)]{MR3357585} or \cite[formula (1.2)]{GS2019}).
Let us define $\zeta\colon [0, \infty) \rightarrow [0, \infty]$, 
\[
\zeta(s) = \begin{cases}
\infty
&\text{if } s = 0, \\
\vphi^*(1/s) & \text{if } 0 < s \leq x_0^{-1}, \\
A \phi(1/s) & \text{if } x_0^{-1} < s,
\end{cases} 
\]
where $A = \vphi^*(x_0)/\phi(x_0) \in (0, 2]$. 
\begin{theorem}
	\label{thm:6}
	Let $\bfT$ be a subordinator with the L\'{e}vy--Khintchine exponent $\psi$ and the Laplace exponent $\phi$. Suppose
	that $-\phi'' \in \WLSC{\alpha-2}{c}{x_0}$ for some $c \in (0, 1]$, $x_0 \geq 0$, and $\alpha>0$. Then the
	probability distribution of $T_t$ has a density $p(t, \: \cdot \:)$. Moreover, there is  $C > 0$ such that for all $t \in (0, 1/\varphi(x_0))$ and $x \in \RR$,
	\begin{equation}
		\label{eq:44}
        p\Big(t, x + t b_{1/\psi^{-1}(1/t)}\Big)
        \leq
        C
        \vphi^{-1}(1/t) \cdot \min{\big\{ 1, t \zeta(\abs{x}) \big\}}.
    \end{equation}
	In particular, for all $t \in (0, 1/\varphi(x_0))$ and $x \geq 2e t\phi'(\psi^{-1}(1/t))$,
	\begin{equation}
		\label{eq:48}
		p(t,x + t b) \leq C \varphi^{-1}(1/t) \cdot \min{\big\{ 1, t\zeta(x) \big\}}.
	\end{equation}
\end{theorem}
\begin{proof}
	Without loss of generality we can assume $b = 0$. Indeed, otherwise it is enough to consider a shifted process $\widetilde{T_t} = T_t - tb$. Next, let us observe that for any Borel set $B \subset \RR$, 
	we have 
	\begin{align}
		\nonumber
		\nu(B) 
		& \lesssim
		\int_{(\delta(B), \infty)} \big(1 - e^{-s/\delta(B)}\big) \: \nu({\rm d} s) \\
		\label{eq:38}
		&\leq \phi\big(1/\delta(B)\big).
	\end{align}
	Furthermore, for $\delta(B)<1/x_0$, by Proposition \ref{prop:3} and Remark \ref{rem:1},
	\begin{align*}
	\nu(B) &\leq h(\delta(B)) \\ &\lesssim \vphi^*(1/\delta(B)).
	\end{align*}
	Thus, $\nu(B) \lesssim \zeta(\delta(B))$. We claim that $\zeta$ has doubling property on $(0, \infty)$. Indeed, since $-\phi''$ is nonincreasing function with
	the weak lower scaling property, it has doubling property on $(x_0, \infty)$, thus for $0 < s < x_0^{-1}$,
	\[
	\zeta\big(\tfrac{1}{2} s\big) \approx 
	4 s^{-2} (-\phi''(2/s)) \lesssim s^{-2} (-\phi''(1/s)) \lesssim \zeta(s).
	\]
	This completes the argument in the case $x_0 = 0$. If $x_0 > 0$, then by \eqref{eq:28}, for $s > 2 x_0^{-1}$
	we have
	\[
	\zeta\big(\tfrac{1}{2} s\big) =
	A \phi(2/s) \leq 2 A  \phi(1/s) \leq 2 \zeta(s).
	\]
	Lastly, the function
	\[
	\big[\tfrac{1}{2} x_0, x_0] \ni x \mapsto \frac{\vphi^*(2 x)}{\phi(x)}
	\]
	is continuous, thus it is bounded.

	Next, for $s > 0$ and $x \in \RR$,
	\[
		s \vee \abs{x} - \tfrac{1}{2} \abs{x} \geq \tfrac{1}{2} s,
	\]
	thus, by motonicity and doubling property of $\zeta$, we get
	\[
		\zeta\big(s \vee \abs{x} - \tfrac{1}{2} \abs{x} \big) \lesssim \zeta(s).
	\]
	Hence, by \eqref{eq:31} and \eqref{eq:33}, for $r > 0$,
	\begin{align}
		\nonumber
		\int_{(r, \infty)} 
		\zeta\big(s \vee x - \tfrac{1}{2} x \big) \: \nu({\rm d} x)
		&\lesssim
		\zeta(s) h(r) \\
		\label{eq:39}
		&\lesssim
		\zeta(s) \psi^*(1/r).
	\end{align}
	Since $\psi^*$ has the weak lower scaling property and satisfies \eqref{eq:34}, by \cite[Proposition 3.4]{GS2019}
	together with Proposition \ref{prop:6}, there are $C > 0$ and $t_1 \in (0, \infty]$ such that for all
	$t \in (0, t_1)$,
	\begin{equation}
		\label{eq:40}
		\int_\RR e^{-t \Re \psi(\xi)} \abs{\xi} \rmd \xi
		\leq
		C \big(\psi^{-1}(1/t)\big)^2.
	\end{equation}
	If $x_0 = 0$ then $t_1 = \infty$. If  $t_1 <  1/\vphi(x_0)$ we can expand the above estimate for $t_1\leq t<  1/\vphi(x_0)$ using positivity of the right hand side and monotonicity of the left hand side. 

	In view of \eqref{eq:38}, \eqref{eq:39} and \eqref{eq:40}, by \cite[Theorem 1]{MR3357585} with $\gamma=0$, there are
	$C_1,C_2,C_3 > 0$ such that for all $t \in (0 , 1/\vphi(x_0))$ and $x \in \RR$,
	\begin{align*}
		&
		p\Big(t, x+t b_{1/\psi^{-1}(1/t)} \Big) \\
		&\qquad\qquad\leq
		C_1
		\psi^{-1}(1/t) \cdot 
		\min{
		\bigg\{
		1, t \zeta\big(\tfrac{1}{4} \abs{x} \big) 
		+ \exp\Big\{-C_2 \abs{x} \psi^{-1}(1/t) \log \big(1 + C_3 \abs{x} \psi^{-1}(1/t)\big) \Big\}
		\bigg\}}.
	\end{align*}
	Let us consider $x > 0$ and $t \in (0, 1/\vphi(x_0))$ such that $t \zeta(x) \leq 1$. We claim that
	\begin{equation}
		\label{eq:43}
		\exp\Big\{-C_2 x \psi^{-1}(1/t) \log\big(1 + C_3 x \psi^{-1}(1/t)\big)\Big\} \lesssim t \zeta(x).
	\end{equation}
	First suppose that $x>x_0^{-1}$. Let us observe that the function
	\[
		[0, \infty) \ni u \mapsto u \exp\Big\{-C_2 u \log \big(1+C_3 u\big) \Big\}
	\]
	is bounded. Therefore,
	\begin{equation}
		\label{eq:46}
		\exp\Big\{-C_2 x \psi^{-1}(1/t) \log\big(1 + C_3 x \psi^{-1}(1/t)\big)\Big\} 
		\lesssim
		\frac{1}{x \psi^{-1}(1/t)}.
	\end{equation}
	Since $x \phi^{-1}(1/t) \geq 1$, by \eqref{eq:28}, we have
	\begin{equation}
		\label{eq:47}
		t \phi(1/x) 
		= \frac{\phi(1/x)}{\phi\big(x \phi^{-1}(1/t) \cdot 1/x\big)}
		\geq \frac{1}{x \phi^{-1}(1/t)}.
	\end{equation}
	Next, in light of \eqref{eq:20}, for all $y > 0$,
	\[
		\tfrac{1}{2} \vphi^*(y) \leq \phi(y),
	\]
	hence, by the monotonicity of $\phi^{-1}$, 
	\begin{align}
		\nonumber
		\phi^{-1}(1/t) &= \phi^{-1}\big(\tfrac{1}{2} \vphi^*(\vphi^{-1}(2/t))\big) \\
		\nonumber
		&\leq
		\phi^{-1}\big(\phi(\vphi^{-1}(2/t))\big) \\
		\nonumber
		&=
		\vphi^{-1}(2/t) \\
		\label{eq:45}
		&\leq
		C \psi^{-1}(1/t)
	\end{align}
	where in the last step we have used Proposition \ref{prop:7}. Putting \eqref{eq:46}, \eqref{eq:47}, and
	\eqref{eq:45} together, we obtain \eqref{eq:43} as claimed. 
	
	Now let $0<x \leq x_0^{-1}$. Observe that the function
	\[
		[0, \infty) \ni u \mapsto u^2 \exp\Big\{-C_2 u \log \big(1+C_3 u\big) \Big\}
	\]
	is also bounded. Hence,
	\begin{equation}
		\label{eq:148}
		\exp\Big\{-C_2 x \psi^{-1}(1/t) \log\big(1 + C_3 x \psi^{-1}(1/t)\big)\Big\} 
		\lesssim
		\frac{1}{\big(x \psi^{-1}(1/t)\big)^2}.
	\end{equation}
	Since $x \vphi^{-1}(1/t) \geq 1$, using \eqref{eq:85} we get
	\begin{align}
		\label{eq:149}
		t\vphi^*(1/x) 
		= \frac{\vphi^*(1/x)}{\vphi^*\big( x\vphi^{-1}(1/t) \cdot 1/x \big)} \geq \frac{1}{\big( x\vphi^{-1}(1/t) \big)^2}.
	\end{align}
	Hence, putting together \eqref{eq:148} and \eqref{eq:149}, and invoking Proposition \ref{prop:6} we again obtain
	\eqref{eq:43}.
	
	Finally, using doubling property of $\zeta$ we get
	\[
		\zeta\big(\tfrac{1}{4} x\big) \lesssim \zeta(x),
	\]
	thus an another application of Proposition \ref{prop:7} leads to \eqref{eq:44}.

	For the proof of \eqref{eq:48},	we observe that
	\[
		\phi'(\lambda) 
		= \int_{(0, \infty)} x e^{-\lambda x} \: \nu({\rm d} x) \\
		\geq 
		e^{-1} \int_{(0, 1/\lambda)}  x \: \nu({\rm d} x).
	\]
	Thus,
	\[
		b_{1/\psi^{-1}(1/t)} 
		= \int_{(0, 1/\psi^{-1}(1/t))} x \: \nu({\rm d}x) \\
		\leq e \phi'(\psi^{-1}(1/t)).
	\]
	Hence, by monotonicity and doubling property of $\zeta$, for $x > 2e t \phi'(\psi^{-1}(1/t))$, we obtain
	\[
		\zeta\Big(x - t b_{1/\psi^{-1}(1/t)}\Big) \leq \zeta\bigg(\frac{x}{2}\bigg) 
		\lesssim \zeta (x),
	\]
	and the theorem follows.
\end{proof}

Now we define $\eta\colon [0, \infty) \rightarrow [0, \infty]$, 
\[
	\eta(s) = s^{-1}\zeta(s)=  \begin{cases}
	\infty
	&\text{if } s = 0, \\
	s^{-1} \vphi^*(1/s) & \text{if } 0 < s \leq x_0^{-1}, \\
	A s^{-1} \phi(1/s) & \text{if } x_0^{-1} < s,
	\end{cases} 
\]
{where $A = \vphi^*(x_0)/\phi(x_0) \in (0, 2]$. Notice that, by \eqref{eq:20}, if $2 t \zeta(\abs{x}) \leq 1$ then 
$t \vphi^*(1/\abs{x}) \leq 1$, and so
\begin{align*}
	\eta(\abs{x}) 
	&=  \abs{x}^{-1} \zeta(\abs{x}) \\
	& \leq \vphi^{-1}(1/t) \zeta(\abs{x}).
\end{align*}
Therefore,
\[
	\min{\big\{\vphi^{-1}(1/t), t \eta(\abs{x}) \big\}} \leq
	4 \vphi^{-1}(1/t) \cdot \min{\big\{1, t \zeta(\abs{x}) \big\}}.
\]
\begin{theorem}
	\label{thm:4}
	Let $\bfT$ be a subordinator with the L\'{e}vy--Khintchine exponent $\psi$ and the Laplace exponent $\phi$.
	Suppose that $-\phi'' \in \WLSC{\alpha-2}{c}{x_0}$ for some $c \in (0,1]$, $x_0 \geq 0$, and $\alpha>0$. 
	We also assume that the L\'{e}vy measure $\nu$ has an almost monotone density $\nu(x)$. Then the probability
	distribution of $T_t$ has a density $p(t, \: \cdot \:)$. Moreover, there is $C > 0$
	such that for all $t\in(0,1/\vphi(x_0))$ and $x \in \RR$,
	\begin{equation}
		\label{eq:98}
		p\Big(t,x+t b_{1/\psi^{-1}(1/t)}\Big) 
		\leq 
		C \min{\big\{\vphi^{-1}(1/t), t \eta(\abs{x}) \big\}}.
	\end{equation}
	In particular, for all $t \in (0, 1/\vphi(x_0))$ and $x \geq 2e t\phi'(\psi^{-1}(1/t))$,
	\begin{equation}
		\label{eq:97}
		p(t, x + t b) \leq C \min{\big\{\vphi^{-1}(1/t), t \eta(x) \big\}}.
	\end{equation}
\end{theorem}
\begin{proof}
	Without loss of generality we can assume $b = 0$. Let us observe that for any $\lambda > 0$,
	\[
		\phi(\lambda) \geq \int_0^{1/\lambda} \big(1 - e^{-\lambda s} \big) \nu(s) \rmd s
		\gtrsim \nu(1/\lambda) \lambda^{-1},
	\]
	and
	\[
		-\phi''(\lambda) \geq \int_0^{1/\lambda} s^2 e^{-\lambda s} \nu(s) \rmd s 
		\gtrsim \nu(1/\lambda) \lambda^{-3}.
	\]
	Hence, 
	\begin{equation}\label{eq:30a}
	\nu(x) \lesssim \eta(x) \qquad \text{for all } x > 0.
	\end{equation} 
	Since $\eta$ is nonincreasing, for any Borel subset $B \subset \RR$,
	\begin{align}
				\nu(B) 
		&\lesssim \int_{B \cap (0,\infty)} \eta(x) \rmd x \label{eq:49}
		\lesssim \eta\big(\delta(B)\big) \diam(B).
	\end{align}
	Arguing as in the proof of Theorem \ref{thm:6} we conclude that $\eta$ has doubling property on $(0,\infty)$. 
Using that and monotonicity of $\eta$, for $s > 0$ and $x \in \RR$,
	\[
		\eta\big(s \vee x - \tfrac{1}{2} x\big) \leq \eta\big(\tfrac{1}{2} s \big) \lesssim \eta(s). 
	\]
	Therefore, by \eqref{eq:33}, for $r > 0$,
	\begin{equation}
		\label{eq:51}
		\int_r^\infty \eta\big(s \vee x - \tfrac{1}{2} x\big) \nu(x) \rmd x
		\lesssim
		\eta(s) \psi^*(1/r).
	\end{equation}
	Since $\psi^*$ has the weak lower scaling property and satisfies \eqref{eq:34}, by \cite[Theorem 3.1]{GS2019}
	and Proposition \ref{prop:6}, there are $C > 0$ and $t_1 \in (0, \infty]$ such that for all $t \in (0, t_1)$,
	\begin{equation}
		\label{eq:50}
		\int_\RR e^{-t\Re \psi(\xi)} \rmd \xi \leq C \psi^{-1}(1/t).
	\end{equation}
	If $x_0 = 0$ then $t_1 = \infty$. If  $t_1 <  48/\vphi(x_0)$ we can expand the above estimate for $t_1\leq t<  48/\vphi(x_0)$ using positivity of the right hand side and monotonicity of the left hand side. 
	
	In view of \eqref{eq:49}, \eqref{eq:51}, and \eqref{eq:50}, by \cite[Theorem 5.2]{GS2017}, there is $C > 0$
	such that for all $t \in (0, 1/\vphi(x_0))$ and $x \in \RR$,
	\[
		p\Big(t,x+tb_{1/\psi^{-1}(1/t)}\Big) 
		\leq C 
		\psi^{-1}(1/t) \cdot 
		\min{ \Big\{ 1, t \big(\psi^{-1}(1/t)\big)^{-1} \eta(\abs{x})+\big(1+\abs{x}\psi^{-1}(1/t)\big)^{-3} \Big\}}.
	\]
	We claim that
	\begin{equation}
		\label{eq:62}
		\frac{\psi^{-1}(1/t)}{\big(1+\abs{x}\psi^{-1}(1/t)\big)^3} 
		\lesssim t \eta(\abs{x})
	\end{equation}
	whenever $t \eta(\abs{x}) \leq \tfrac{A}{2} \vphi^{-1}(1/t)$.

	First, let us show that for any $\epsilon \in (0, 1]$, the condition
	$t \eta(\abs{x}) \leq \tfrac{A \epsilon}{2} \vphi^{-1}(1/t)$ implies that
	\begin{equation}
		\label{eq:61}
		t \vphi^*\bigg(\frac{1}{\abs{x}}\bigg) \leq \epsilon \abs{x} \vphi^{-1}\bigg(\frac{1}{t}\bigg).
	\end{equation}
	Indeed, by \eqref{eq:20}, we have $\abs{x} \eta(\abs{x}) \geq \tfrac{A}{2} \vphi^*(1/\abs{x})$, thus
	\[
		\epsilon \abs{x} \vphi^{-1}\(\frac{1}{t}\) 
		\geq \frac{2}{A} t \abs{x} \eta(\abs{x}) 
		\geq t \vphi^*\(\frac{1}{\abs{x}}\).
	\]
	Notice also that $\epsilon^{1/3} \abs{x} \vphi^{-1}(1/t) \geq 1$ since otherwise, by \eqref{eq:85},
	\[
		1 < t \vphi^*\bigg(\frac{1}{\epsilon^{1/3} \abs{x}}\bigg)
		< \frac{1}{\epsilon^{2/3}} t \vphi^*\bigg(\frac{1}{\abs{x}}\bigg),
	\]
	which entails that $\epsilon^{2/3} < t \vphi^*(1/\abs{x})$, i.e. $\epsilon^{1/3} \abs{x} \vphi^{-1}(1/t) < 
	\epsilon^{-2/3} t \vphi^*(1/\abs{x})$ contrary to \eqref{eq:61}.

	To show \eqref{eq:62}, let us suppose that $t \eta(\abs{x}) \leq \tfrac{A}{2} \vphi^{-1}(1/t)$, thus
	$\abs{x} \vphi^{-1}(1/t) \geq 1$. By \eqref{eq:85}, we have
	\[
		t \vphi^*(1/\abs{x}) 
		= \frac{\vphi^*(1/\abs{x})}{\vphi^*\big(\abs{x} \vphi^{-1}(1/t) \cdot 1/\abs{x}\big)}
		\geq
		\frac{1}{(\abs{x} \vphi^{-1}(1/t))^2},
	\]
	which, by Proposition \ref{prop:7}, gives
	\[
		t \abs{x} \eta(\abs{x}) \geq \tfrac{A}{2} t \vphi^*(1/\abs{x}) 
		\gtrsim \frac{|x|\psi^{-1}(1/t)}{(1 + \abs{x} \psi^{-1}(1/t))^3},
	\]
	proving \eqref{eq:62}, and \eqref{eq:98} follows. The inequality \eqref{eq:97} holds by the same argument as in the proof of Theorem  \ref{thm:6}.
\end{proof}

\begin{remark}
	In statements of Theorems \ref{thm:6} and \ref{thm:4}, we can replace $b_{1/\psi^{-1}(1/t)}$
	by $b_{1/\vphi^{-1}(1/t)}$. Indeed, let us observe that if $0 < r_1 \leq r_2 < 1/x_0$ then
	\begin{align}
		\nonumber
		\big| b_{r_1} - b_{r_2} \big| 
		&\leq \int_{(r_1, r_2]} s \: \nu({\rm d} s) \\
		\nonumber 
		&\leq r_1^{-1} r_2^2 h(r_2) \\
		\label{eq:83}
		&\lesssim r_1^{-1} r_2^2 \psi^*(1/r_2),
	\end{align}
	where in the last estimate we have used \eqref{eq:33}. Hence, by \eqref{eq:42}, we get
	\begin{equation}
		\label{eq:84}
		\big| b_{r_1} - b_{r_2} \big| \lesssim r_1^{-1} r_2^2 \vphi^*(1/r_2).
	\end{equation}
	Therefore, by \eqref{eq:83}, \eqref{eq:84}, and Proposition \ref{prop:7}, there is $C \geq 1$ such that
	\begin{equation}
		\label{eq:87}
		\Big| b_{1/\psi^{-1}(1/t)} - b_{1/\vphi^{-1}(1/t)} \Big| \leq C \frac{1}{t \vphi^{-1}(1/t)},
	\end{equation}
	provided that $0 < t < 1/\vphi(x_0)$. Now, let us suppose that
	$8 C^2 t \zeta\big(\abs{x}\big) \leq 1$. Then, by \eqref{eq:28} and \eqref{eq:85},
	\begin{align*}
		\frac{1}{t} 
		\geq 8 C^2 \zeta\big(\abs{x} \big) 
		\geq 4C^2 \vphi^*\bigg(\frac{1}{\abs{x}}\bigg) 
		\geq \vphi^*\bigg(\frac{2 C}{\abs{x}} \bigg),
	\end{align*}
	that is
	\begin{equation}
		\label{eq:86}
		\abs{x} \geq \frac{2 C}{\vphi^{-1}(1/t)}.
	\end{equation}
	Hence, by \eqref{eq:87},
	\[
		\Big|x + t\Big( b_{1/\psi^{-1}(1/t)} - b_{1/\vphi^{-1}(1/t)}\Big)\Big| 
		\geq 
		\abs{x} - \frac{C}{\vphi^{-1}(1/t)} \geq \frac{\abs{x}}{2},
	\]
	which together with monotonicity and the doubling property of $\zeta$, gives
	\[
		\zeta\Big(\big| x + t \big(b_{1/\psi^{-1}(1/t)} - b_{1/\vphi^{-1}(1/t)}\big)\big| \Big)
		\lesssim \zeta\big( \abs{x} \big).
	\]
	Similarly, if $t \eta(\abs{x}) \leq \tfrac{A \epsilon}{2} \vphi^{-1}(1/t)$, then 
	\[
		\abs{x} \vphi^{-1}(1/t) \geq \epsilon^{-1/3},
	\]
	thus, by taking $\epsilon = (2C)^{-3}$, we obtain \eqref{eq:86}. Hence, by monotonicity and the doubling property of
	$\eta$, we again obtain
	\[
		\eta\bigg(\Big| x + t\Big( b_{1/\psi^{-1}(1/t)} - b_{1/\vphi^{-1}(1/t)}\Big) \Big|\bigg)
		\lesssim
		\eta(\abs{x}).
	\]
\end{remark}

\subsection{Estimates from below}
In this section we develop estimates from below on the density $p(t, \: \cdot \:)$. The main result is Theorem \ref{thm:2}.
Its proof is inspired by the ideas from \cite{MR1486930}, see also \cite{GS2019}. Thanks to Theorem \ref{thm:3}, we can
generalize results obtained in \cite{MR1486930} to the case when $-\phi''$ satisfies the weak lower scaling of index
$\alpha-2$ for $\alpha>0$ together with a certain additional condition. We use the following variant of
the celebrated Pruitt's result \cite[Section 3]{MR632968} adapted to subordinators.
\begin{proposition}
	\label{prop:13}
	Let $\bfT$ be a subordinator with the L\'{e}vy--Khintchine exponent 
	\[
		\psi(\xi) = -i\xi b - \int_{(0, \infty)} \big(e^{i\xi x} - 1\big) \: \nu({\rm d} x).
	\]
	Then there is an absolute constant $c>0$ such that for all $\lambda > 0$ and $t > 0$,
	\[
		\PP \Big( \sup_{0 \leq s \leq t} \big|T_s- s b_\lambda \big|\geq \lambda \Big) 
		\leq ct h(\lambda).
	\]
\end{proposition}
\begin{proof}
	We are going to apply the estimates \cite[(3.2)]{MR632968}. To do so, we need to express the L\'{e}vy--Khintchine
	exponent of $T_s - s b_\lambda$ in the form used in \cite[Section 3]{MR632968}, namely
	\[
		\tilde{\psi}(\xi) = \psi(\xi) + i \xi b_\lambda
		=
		-i\xi \bigg(b - b_\lambda + \int_{(0, \infty)} \frac{y}{1 + |y|^2} \: \nu({\rm d} y) \bigg) 
		- \int_{(0, \infty)} \bigg(e^{i\xi y} - 1 - \frac{i y \xi}{1 + |y|^2} \bigg) \: \nu({\rm d} y).
	\]
	Since
	\[
		\int_{(0, \lambda]} \frac{y |y|^2}{1+|y|^2} \: \nu({\rm d} y) 
		- \int_{(\lambda, \infty)} \frac{y}{1+|y|^2} \: \nu({\rm d} y)
		=
		\int_{(0, \lambda]} y \: \nu({\rm d} y) - \int_{(0, \infty)} \frac{y}{1+|y|^2} \: \nu({\rm d} y),
	\]
	we have
	\begin{align*}
		M(\lambda) &= \frac{1}{\lambda} \bigg| 
		b - b_\lambda + 
		\int_{(0, \infty)} \frac{y}{1 + |y|^2} \: \nu({\rm d} y)
		+
		\int_{(0, \lambda]} \frac{y |y|^2}{1+|y|^2} \: \nu({\rm d} y)
		-
		\int_{(\lambda, \infty)} \frac{y}{1+|y|^2} \: \nu({\rm d} y)
		\bigg| \\
		&=
		\frac{1}{\lambda}
		 \bigg|
		b - b_\lambda +  \int_{(0, \lambda]} y \: \nu({\rm d} y)
		\bigg| = 0.
	\end{align*}
	Hence, by \cite[(3.2)]{MR632968},
	\[
		\PP \Big( \sup_{0 \leq s \leq t} \big|T_s- s b_\lambda \big|\geq \lambda \Big)
		\leq
		c s h(\lambda),
	\]
	as desired.
\end{proof}

\begin{theorem}
	\label{thm:2}
	Let $\bfT$ be a subordinator with the Laplace exponent $\phi$. Suppose that $-\phi'' \in \WLSC{\alpha-2}{c}{x_0}$
	for some $c \in (0, 1]$, $x_0 \geq 0$, and $\alpha>0$, and assume that one of the following conditions holds
	true:
	\begin{enumerate}
		\item $-\phi'' \in \WUSC{\beta-2}{C}{x_0}$ for some $C\geq 1$ and $\alpha \leq \beta <1$, or
		\item $-\phi''$ is a function regularly varying at infinity with index $-1$.
		If $x_0 = 0$ we also assume that $-\phi''$ is regularly varying at zero with index $-1$.
	\end{enumerate}
	Then there exists $\rho_0 > 0$, so that for all $0 < \rho_1 < \rho_0$, $0 < \rho_2$ there is $C>0$ such that for
	all $t \in (0, 1/\varphi(x_0))$, and all $x > 0$ satisfying
	\[
		-\frac{\rho_1}{\vphi^{-1}(1/t)} \leq
		x -  t \phi'\big(\vphi^{-1}(1/t)\big)
		\leq
		\frac{\rho_2}{\vphi^{-1}(1/t)},
	\]
	we have
	\begin{equation}\label{eq:22}
		p\left(t, x\right) \geq C \vphi^{-1}(1/t).
	\end{equation}
\end{theorem}

\begin{remark}
	From the proof of Theorem \ref{thm:2} it stems that if $x_0=0$, one can obtain the same statement under the condition that 
	$-\phi''$ is $(-1)$-regular at infinity and satisfies upper scaling at $0$ with $\alpha \leq \beta < 1$.
	Alternatively, one can assume that $-\phi''$ satisfies upper scaling at infinity with $\alpha \leq \beta < 1$
	and varies regularly at zero with index $-1$. The same remark applies to Proposition \ref{prop:4}.
\end{remark}

\begin{proof}
	First let us observe that it is enough to prove that \eqref{eq:22} holds true for a certain $M \geq 1$,
	all $t \in (0,1/\vphi(x_0))$ and all $x > 0$ satisfying
	\[
		-\frac{\rho_1}{\vphi^{-1}(M/t)} \leq
		x -  t \phi'\big(\vphi^{-1}(1/t)\big)
		\leq
		\frac{\rho_2}{\vphi^{-1}(M/t)}.
	\]
	Indeed, since $\vphi^{-1}$ is nondecreasing and has upper scaling property (see Proposition \ref{prop:7}), it has
	a doubling property. Hence, the lemma will follow immediately with possibly modified $\rho_0$.
	
	Without loss of generality we can assume that $b = 0$. Let $\lambda > 0$, whose value will be specified later. 
	We decompose the L\'{e}vy measure $\nu({\rm d} x)$ as follows: Let $\nu_1({\rm d} x)$ be the restriction of 
	$\tfrac12 \nu({\rm d} x)$ to the interval $(0,\lambda]$, and
	\[
		\nu_2({\rm d} x)=\nu({\rm d} x)-\nu_1({\rm d} x).
	\]
	We set
	\[
		\phi_1(u) = \int_{(0, \infty)} \big(1 - e^{-u s}\big) \: \nu_1({\rm d} s), 
		\qquad
		\phi_2(u) = \int_{(0, \infty)} \big(1 - e^{-u s}\big) \: \nu_2({\rm d} s).
	\]
	Let us denote by $\bfT^{(j)}$ the subordinator having the Laplace exponent $\phi_j$, for
	$j \in \{ 1, 2 \}$. Let $\psi_j(\xi) = \phi_j(-i\xi)$. Notice that $\tfrac12 \nu \leq \nu_2 \leq \nu$, thus
	\[
		\tfrac12 \phi \leq \phi_2 \leq \phi,
	\]
	and for every $n \in \NN$,
	\begin{equation}
		\label{eq:88}
		\tfrac12 (-1)^{n+1} \phi^{(n)} 
		\leq (-1)^{n+1} \phi_2^{(n)} 
		\leq (-1)^{n+1} \phi^{(n)}.
	\end{equation}
	Therefore, for all $u > 0$,
	\begin{equation}
		\label{eq:67}
		\tfrac12 \vphi(u) \leq \vphi_2(u) \leq \vphi(u).
	\end{equation}
	Next, by Theorem \ref{thm:3}, the random variables $T_t^{(2)}$ and $T_t$ are absolutely continuous. 
	Let us denote by $p^{(2)}(t, \:\cdot\:)$ and $p(t, \:\cdot\:)$ the densities of $T^{(2)}_t$ and $T_t$, respectively.

	Let $M \geq 2 M_0+1$, where $M_0$ is determined in Corollary \ref{cor:1} for
	the process $\bfT^{(2)}$. For $0 < t < 1/\varphi(x_0)$, we set
	\[
		x_t = t \phi_2'\big(\vphi^{-1}(M/t)\big).
	\]
	Since $\vphi^{-1}(M/t) > x_0$, we have
	\[
		\frac{x_t}{t} 
		= \phi_2'\big(\vphi^{-1}(M/t)\big) 
		\leq \phi'_2(x_0).
	\]
	Let
	\[
		w_2 = (\phi_2')^{-1}(x_t/t) =\vphi^{-1}(M/t).
	\]
	Then, by \eqref{eq:67} we get
	\begin{align*}
		\vphi_2(w_2) \geq \frac12 \vphi \big(  \vphi^{-1}(M/t) \big) = \frac{M}{2t} \geq \frac{M_0}{t}.
	\end{align*}
	Moreover, by Corollary \ref{cor:3} together with \eqref{eq:67} we get
	\begin{align*}
		t\big( \phi_2(w_2) - w_2 \phi_2'(w_2) \big)
		\lesssim
		t \vphi_2\big(w_2\big) \lesssim
		1.
	\end{align*}
	Hence, by Corollary \ref{cor:1},
	\begin{equation}
		\label{eq:89}
        p^{(2)}(t, x_t) \gtrsim \frac{1}{\sqrt{t (-\phi_2'')(w_2)}}.
	\end{equation}
	Notice that, by \eqref{eq:88} and Remark \ref{rem:3}, the implied constant in \eqref{eq:89} is independent of
	$t$ and $\lambda$. Since
	\begin{align*}
		(-\phi_2'')(w_2) \leq (-\phi'')\big( \vphi^{-1}(M/t)\big) =\frac{M}{t \big(\vphi^{-1}(M/t)\big)^{2}},
	\end{align*}
	by \eqref{eq:89} and monotonicity of $\vphi^{-1}$, we get
	\begin{equation}
		\label{eq:52}
		p^{(2)}(t, x_t) \geq C_1 \vphi^{-1} (1/t),
	\end{equation}
	for some constant $C_1 > 0$.
	
	Next, by the Fourier inversion formula 
	\begin{align*}
		\sup_{x \in \RR} \big| \partial_x p^{(2)}(t,x) \big|
		&\lesssim
		\int_\RR
		e^{-t \Re \psi_2(\xi)} \abs{\xi} \rmd \xi \\
		&\lesssim
		\int_\RR
		e^{-\frac{t}{2} \Re \psi(\xi)} \abs{\xi} \rmd \xi,
	\end{align*}
	thus, by \cite[Proposition 3.4]{GS2019}, and Propositions \ref{prop:6} and \ref{prop:7} we see
	that there is $C_2 > 0$ such that for all $t \in (0, 1/\varphi(x_0))$,
	\[
		\sup_{x \in \RR} \big| \partial_x p^{(2)}(t,x) \big|
		\leq
		C_2 \big(\vphi^{-1}(1/t)\big)^2.
	\]
	By the mean value theorem, for $y \in \RR$, we get
	\[
		\big|
		p^{(2)}(t, y + x_t) - p^{(2)}(t, x_t) 
		\big|
		\leq
		C_2 \abs{y} \big(\vphi^{-1}(1/t)\big)^2.
	\]
	Hence, for $y \in \RR$ satisfying
	\[
		\abs{y} \leq \frac{C_1} {2 C_2 \vphi^{-1}(1/t)},
	\]
	by \eqref{eq:52}, we get
	\begin{align*}
		p^{(2)}(t, y + x_t) 
		&\geq p^{(2)}(t, x_t) - C_2 \abs{y}  \big(\vphi^{-1}(1/t)\big)^2 \\
		&\geq \frac{C_1}{2} \vphi^{-1}(1/t).
	\end{align*}
	Therefore,
	\begin{align*}
		p(t,x) &= 
		\int_{\RR} p^{(2)}(t,x-y) \PP \big(  T_t^{(1)}\in \rmd y \big) \\
		&\geq
		\frac{C_1}{2} \vphi^{-1}(1/t) \cdot	
		\PP \bigg( \big|x - x_t - T_t^{(1)} \big| \leq \frac{C_0}{\varphi^{-1}(1/t)} \bigg) \\
		&\geq
		\frac{C_1}{2} \vphi^{-1}(1/t) \cdot	
		\PP \bigg( \big|x - x_t - T_t^{(1)} \big| \leq \frac{C_0}{\varphi^{-1}(M/t)} \bigg) \\
		&= \frac{C_1}{2} \vphi^{-1}(1/t) \cdot
		\PP \bigg( \Big|x - \tilde{x}_t - \big( \tfrac12 tb_{\lambda} - (\tilde{x}_t - x_t) \big) 
		- \big(T_t^{(1)}-\tfrac12 tb_{\lambda}\big) \Big| 
		\leq \frac{C_0}{\varphi^{-1}(M/t)} \bigg),
	\end{align*}
	where we have set $C_0=C_1(2C_2)^{-1}$ and
	\[
		\tilde{x}_t = t\phi' \big(\vphi^{-1}(M/t)\big).
	\]
	Let $\rho_0 = \tfrac{1}{2} C_0$ and
	\begin{align}
		\label{eq:123}
		\lambda = \frac{1}{\vphi^{-1}(M/t)}.
	\end{align}
	We have
	\begin{align*}
		\frac{1}{2} tb_{\lambda} - (\tilde{x}_t - x_t) 
		&= \frac12tb_{\lambda} - t\phi_1'(1/\lambda)  \\
		&= \frac{t}{2} \int_{(0,\lambda]} s\big(1-e^{-s/\lambda} \big) \: \nu({\rm d} s).
	\end{align*}
	Thus, $\tfrac12 tb_{\lambda} - (\tilde{x}_t - x_t)$ is nonnegative and in view of \eqref{eq:33}
	and \eqref{eq:123},
	\begin{align}
		\nonumber
		\tfrac12 tb_{\lambda} - (\tilde{x}_t - x_t) 
		&\leq 
		C_3 t \lambda \vphi(1/\lambda) \\
		\label{eq:29}
		&= \frac{C_3M}{\vphi^{-1}(M/t)}, 
	\end{align}
	for some constant $C_3>0$. Next, setting
	\[
		\rho(t) = \lambda^{-1}\Big( \tfrac12 tb_{\lambda} - (\tilde{x}_t - x_t)\Big),
	\]
	we get
	\begin{align}
		\nonumber
		&
		\inf_{t \in (0, 1/\varphi(x_0))}{\bigg\{ 
		\PP \bigg( 
		\Big| x -\tilde{x}_t - \lambda\rho(t) - \big( T_t^{(1)} - \tfrac12 tb_{\lambda}\big) \Big| 
		\leq C_0 \lambda \bigg)
		\colon
		x \geq 0, -\rho_1 \lambda \leq x- \tilde{x}_t 
		\leq 
		\rho_2\lambda \bigg\}} \\
		\label{eq:56}
		&\qquad\qquad\geq
		\inf_{t \in (0, 1/\varphi(x_0))}{
		\Big\{
		\PP\( \Big| y - \lambda^{-1} \big( T_t^{(1)} -\tfrac12 tb_{\lambda} \big) \Big| \leq C_0 \)
		\colon
		-\rho_1 -\rho(t) \leq y \leq \rho_2  \Big\}}.
	\end{align}
	Hence, the problem is reduced to showing that the infimum above is positive. Let us consider a collection
	$\{Y_t \colon t \in (0, 1/\varphi(x_0)) \}$ of infinitely divisible nonnegative random variables
	$Y_t = \lambda^{-1}\big( T^{(1)}_t-\tfrac12 tb_{\lambda}\big)$. The L\'{e}vy measure corresponding to $Y_t$ 
	is
	\begin{equation}
		\label{eq:53}
		\mu_t(B) = t \nu_1\big(\lambda B\big)
	\end{equation}
	for any Borel subset $B \subset \RR$. Since for each $R > 1$,
	\begin{align*}
		b_{R \lambda}^{(1)} 
		&= \int_{(0, R \lambda]} y \: \nu_1({\rm d} y) \\
		&= \frac{1}{2} \int_{(0, \lambda]} y \: \nu({\rm d} y) = \frac{1}{2} b_\lambda,
	\end{align*}
	by Proposition \ref{prop:13}, 
	\begin{align*}
		\PP\big(|Y_t| \geq R\big) &=
		\PP\(\Big| T_t^{(1)}-\tfrac12 tb_{\lambda} \Big| \geq R \lambda \) \\
		&\lesssim t \int_{(0, \infty)} \min{\big\{1, R^{-2} \lambda^{-2} s^2\big\}} \: \nu_1({\rm d} s),
	\end{align*}
	thus,
	\begin{align*}
		\PP\big(|Y_t| \geq R\big)
		&\lesssim 
		t \lambda^{-2} R^{-2} \int_{(0, \lambda]} s^2 \: \nu({\rm d} s) \\
		&\lesssim
		t R^{-2} h(\lambda) \\
		&\lesssim
		t R^{-2} \vphi(1/\lambda)
	\end{align*}
	where in the last estimate we have used \eqref{eq:33}. Therefore, recalling \eqref{eq:123} we conclude that the
	collection is tight. Next, let $\big((Y_{t_n}, y_n) \colon n \in \NN\big)$ be a sequence realizing the infimum
	in \eqref{eq:56}. By the Prokhorov theorem, we can assume that $(Y_{t_n} \colon n \in \NN)$ is weakly
	convergent to the random variable $Y_0$. We note that $Y_{t_n}$ has the probability distribution supported in
	$\big[-\tfrac12t_n\lambda_n^{-1}b_{\lambda_n}, \infty\big)$ where $\lambda_n$ is defined as $\lambda$ corresponding
	to $t_n$.

	Suppose that $(t_n : n \in \NN)$ contains a subsequence convergent to $t_0 > 0$. Then $Y_0 = Y_{t_0}$ and
	the support of its probability distribution equals $\big[-\tfrac12t_0\lambda_0^{-1}b_{\lambda_0}, \infty\big)$. 
	Since $\rho(t_0) \leq \tfrac12t_0\lambda_0^{-1}b_{\lambda_0}$, we easily conclude that the infimum in \eqref{eq:56}
	is positive.

	Hence, it remains to investigate the case when $(t_n \colon n \in \NN)$ has no positive accumulation points.
	If zero is the only accumulation point then $(\lambda_n \colon n \in \NN)$ has a subsequence convergent to zero.
	Otherwise $(t_n)$ diverges to infinity, thus $x_0 = 0$ and $(\lambda_n)$ contains a subsequence diverging to infinity.
	In view of \eqref{eq:29}, $\rho(t)$ is uniformly bounded in $t$. Thus, after taking a subsequence we may
	and do assume that there exists a limit
	\[
		\tilde{\rho} = \lim_{n \to \infty} \rho(t_n).
	\]
	By compactness we can also assume that $(y_n \colon n \in \NN)$ converges to
	$y_0 \in [-\rho_1-\tilde{\rho}, \rho_2]$. Consequently, to prove that the infimum in \eqref{eq:56} is positive
	it is sufficient to show that
	\begin{equation}
		\label{eq:35}
		\PP\big(|y_0 - Y_0| \leq \tfrac{1}{2} C_0 \big) > 0.
	\end{equation}
	Observe that \eqref{eq:35} is trivially satisfied if the support of the probability distribution of $Y_0$ is
	the whole real line. Therefore, we can assume that $Y_0$ is purely non-Gaussian. In view of
	\cite[Theorem 8.7]{MR1739520}, it is also infinitely divisible.

	Given $w \colon \RR \rightarrow \RR$ a continuous function satisfying
	\begin{align}
		\label{eq:125}
		\big| w(x) - 1 \big| \leq C' \abs{x}, \qquad\text{and}\qquad
		\big| w(x)\big| \leq C' \abs{x}^{-1},
	\end{align}
	we write the L\'{e}vy--Khintchine exponent of $Y_{t_n}$ in the form
	\[
		\psi_n(\xi) = -i \xi \gamma_n - \int_{(0, \infty)} \big(e^{i\xi s} -1 -i\xi s w(s) \big)
		\: \mu_{t_n}({\rm d} s)
	\]
	where
	\[
		\gamma_n = \int_{(0, \infty)} s w(s) \: \mu_{t_n}({\rm d} s) - \tfrac12 \lambda_n^{-1}t_n b_{\lambda_n}.
	\]
	Since $(Y_{t_n} : n \in \NN)$ converges weakly to $Y_0$, there are $\gamma_0 \in \RR$ and $\sigma$-finite measure
	$\mu_0$ on $(0, \infty)$ satisfying
	\[
		\int_{(0, \infty)} \min{\big\{1, s^2 \big\}} \: \mu_0({\rm d} s) < \infty,
	\]
	such that the L\'{e}vy--Khintchine exponent of $Y_0$ is
	\[
		\psi_0(\xi) = -i\xi \gamma_0 - \int_{(0, \infty)} \big(e^{i\xi s} -1 -i\xi s w(s) \big)
		\:
        \mu_0({\rm d} s)
	\]
	where
	\begin{equation}
		\label{eq:82}
		\gamma_0 = \lim_{n \to \infty} \gamma_n.
	\end{equation}
	Moreover, for any bounded continuous function $f\colon \RR \rightarrow \RR$ vanishing in a neighborhood of zero,
	we have
	\begin{equation}
		\label{eq:81}
		\lim_{n \to \infty} \int_{(0, \infty)} f(s) \: \mu_{t_n}({\rm d} s) 
		= \int_{(0, \infty)} f(s) \: \mu_{0}({\rm d} s).
	\end{equation}
	Next, let us fix $w$ satisfying \eqref{eq:125} which equals $1$ on $[0,1]$. In view of \eqref{eq:53} and
	the definition of $\nu_1$ the support of $\mu_{t_n}$ is contained in $[0, 1]$. Hence, $\gamma_n=0$ for every
	$n \in \NN$ and consequently, $\gamma_0=0$. We also conclude that $\supp \mu_0 \subset [0,1]$.
	
	At this stage we consider the cases (i) and (ii) separately. In (ii) we need to distinguish two possibilities: 
	if $(t_n)$ tends to zero then also $(\lambda_n)$ approaches to zero, and we impose that $-\phi''$ is a function
	regularly varying at infinity with index $-1$; otherwise $(t_n)$ tends to infinity as well as $(\lambda_n)$, thus
	$x_0 = 0$ and we additionally assume that $-\phi''$ is a function regularly varying at zero with index $-1$. For the sake of
	clarity of presentation, we restrict attention to the first possibility only. In the second one the reasoning is
	analogous. We show that the support of the probability distribution of $Y_0$ is the whole real line. 
	By \cite[Theorem 24.10]{MR1739520}, the latter can be deduced from
	\begin{equation}
		\label{eq:54}
		\int_{(0, \infty)} \min\{1, s\} \: \mu_0({\rm d} s) = \infty.
	\end{equation}
	Since $\supp \mu_0 \subset [0, 1]$, for each $\epsilon \in (0,1)$ we can write
	\[
		\int_{(0, \infty)} \min\{1, s\} \: \mu_0({\rm d} s) 
		\geq 
		\int_{(\epsilon/2, 1]} s \: \mu_0({\rm d} s),
	\]
	thus to conclude \eqref{eq:54}, it is enough to show that
	\begin{equation}
		\label{eq:68}
		\int_{(\epsilon/2, 1]} s \: \mu_0({\rm d} s) \gtrsim \log \epsilon^{-1}.
	\end{equation}
	For the proof, for any $\epsilon \in (0,1)$ we define the following
	bounded continuous function
	\begin{align}\label{eq:65}
		f_{\epsilon}(s) = \begin{cases}
			0 & \text{if } s<\epsilon/2,\\
			2s-\epsilon, & \text{if } \epsilon/2 \leq s < \epsilon,\\
			s & \text{if } \epsilon \leq s < 1,\\
			1 & \text{if }s \geq 1.
		\end{cases}
	\end{align}
	We have, in view of \eqref{eq:81},
	\begin{equation}\label{eq:66}
	\int_{(\epsilon/2,1]} s \: \mu_0({\rm d}s) \geq \int_{(0,1]}f_{\epsilon}(s) \: \mu_0({\rm d}s) = \lim_{n \to \infty} \int_{(0,1]}f_{\epsilon}(s) \: \mu_{t_n}({\rm d}s) \geq \liminf_{n \to \infty} \int_{(\epsilon,1]} s\:\mu_{t_n}({\rm d}s).
	\end{equation}
	Let us estimate the last integral. We write
	\begin{align*}
		\int_{(\epsilon, 1]} s \: \mu_t (\rm{d} s) \nonumber 
		&= 
		t\lambda^{-1} \int_{(\lambda \epsilon, \lambda]}s \: \nu_1({\rm d}s) \\ 
		\label{eq:131}
		&= 
		\tfrac{1}{2} t \lambda^{-1}\int_{(\lambda \epsilon, \lambda ]}s \: \nu({\rm d}s).
	\end{align*}
	By the Fubini--Tonelli theorem, we get
	\[
		\int_{[\lambda\epsilon,\lambda)} s \:\nu({\rm d}s) = 
		\int_{\lambda \epsilon}^{\lambda}u^{-2} \int_{(0,u]} s^2\:\nu({\rm d}s) \:{\rm d}u
		+ \lambda K(\lambda) - \lambda \epsilon K(\lambda \epsilon).
	\]
	Thus,
	\begin{equation}
		\label{eq:72}
		2 \int_{[\epsilon,1)} s \:\mu_t({\rm d}s) 
		= t\lambda^{-1} \int_{\lambda \epsilon}^{\lambda} K(u)\:{\rm d}u 
		+ tK(\lambda) - t\epsilon K(\lambda \epsilon).
	\end{equation}
	Setting $z = 1/\lambda$, by \eqref{eq:33} and \eqref{eq:123}, we obtain
	\[
		t K(\lambda) \approx t \vphi(z) \approx 1.
	\]
	Moreover, since $\vphi$ is $1$-regularly varying function at infinity, we have
	\[
		t\epsilon K(\lambda \epsilon) \approx t \epsilon \vphi(z/\epsilon) = M \epsilon 
		\frac{\vphi(z/\epsilon)}{\vphi(z)}\to M,
	\]
	as $z$ tends to infinity. Therefore, it remains to estimate the integral in \eqref{eq:72}. Using \eqref{eq:33} we get
	\begin{align*}
		t\lambda^{-1} \int_{\lambda \epsilon}^{\lambda} K(u)\:{\rm d}u
		&\approx 
		\frac{z}{\vphi(z)} \int_{\epsilon z^{-1}}^{z^{-1}}\vphi \big( u^{-1} \big)\:{\rm d}u \\ 
		&=\frac{z}{\vphi(z)}\int_z^{\epsilon^{-1}z}u^{-2}\vphi(u)\:{\rm d}u \\ 
		&= \frac{\phi'(z)-\phi' \big(\epsilon^{-1}z\big)}{z \big(-\phi''(z)\big)}.
	\end{align*}
	Since $-\phi''(s)=s^{-1} \ell(s)$ for a certain function $\ell$ slowly varying at infinity, 
	by \cite[Theorem 1.5.6]{MR1015093},
	\[
		\frac{\phi'(z)-\phi' \big(\epsilon^{-1}z\big)}{z \big(-\phi''(z)\big)} 
		= 
		\int_1^{\epsilon^{-1}} \frac{\ell(zt)}{\ell(z)} \frac{{\rm d}t}{t} 
		\to \log \epsilon^{-1},
	\]
	as $z$ tends to infinity. Hence, 
	\[
		\liminf_{n \to \infty} \int_{(\epsilon, 1]} s \: \mu_{t_n}({\rm d} s) \gtrsim \log \epsilon^{-1},
	\]
	which by \eqref{eq:66} implies \eqref{eq:68}.

	Next, let us consider the case (i) that is when $-\phi'' \in \WUSC{\beta-2}{C}{x_0}$ with $C \geq 1$ and 
	$\alpha\leq\beta<1$. We claim that for all $\epsilon \in (0, 1)$,
	\begin{equation}
		\label{eq:64}
		\int_{(0, \epsilon)} s^2 \: \mu_0({\rm d} s) >0.
	\end{equation}
	To see this, it is enough to show that there is $C > 0$ such that for all $\epsilon \in (0, 1]$ and 
	$t \in (0, 1/\varphi(x_0))$,
	\begin{equation}
		\label{eq:63}
		\int_{(0, \epsilon)} s^2 \: \mu_t({\rm d} s) \geq C \epsilon^{2-\alpha}.
	\end{equation}
	For the proof, we select a continuous function on $\RR$ such that
	\[
		\ind{(-1, 1)} \leq \eta \leq \ind{(-2, 2)},
	\]
	and for each $\tau > 0$ set
	\[
		\eta_\tau(x) = \eta(\tau^{-1} x).
	\]
	Since for $0 < 2 \tau < \epsilon$, 
	\begin{align*}
		\int_{(0, \infty)} s^2 \big(\eta_\epsilon(s) - \eta_\tau(s)\big) \: \mu_t({\rm d} s)
		+ \int_{(0, 2\tau)} s^2 \eta_\tau(s) \: \mu_t({\rm d} s)
		&\geq
		\int_{(0, \epsilon)} s^2  \: \mu_t({\rm d} s),
	\end{align*}
	by \eqref{eq:63} and \eqref{eq:81},
	\[
		\int_{(0, \infty)} s^2 \big(\eta_\epsilon(s) - \eta_\tau(s)\big) \: \mu_0({\rm d} s) +
		\limsup_{n \to \infty} 
		\int_{(0, \infty)} s^2 \eta_\tau(s) \: \mu_{t_n}({\rm d} s)
		\geq C \epsilon^{2-\alpha}.
	\]
	Since $Y_{t_n}$ and $Y_0$ are purely non-Gaussian, by \cite[Theorem 8.7(2)]{MR1739520},
	\[	
		\lim_{\tau \to 0^+}
		\limsup_{n \to \infty} 
		\int_{(-\tau, \tau)} s^2 \: \mu_{t_n} ({\rm d} s) = 0,
	\]
	thus,
	\[
		\int_{(0, \epsilon)} s^2 \:\mu_0({\rm d} s) \geq C \epsilon^{2-\alpha},
	\]
	which entails \eqref{eq:64}. 

	We now turn to showing \eqref{eq:63}. We have
	\begin{align*}
		\int_{(0, \epsilon)} s^2 \: \mu_t({\rm d} s) 
		&=
		t \lambda^{-2} \int_{(0, \lambda \epsilon)} s^2 \: \nu_1({\rm d} s) \\
		&=
		\tfrac{1}{2} t \lambda^{-2} \int_{(0, \lambda \epsilon)} s^2 \: \nu({\rm d} s)\\
		&=
		\tfrac{1}{2} t \epsilon^2 K(\lambda \epsilon),
	\end{align*}
	thus, by \eqref{eq:33} and the weak lower scaling property of $\vphi$,
	\begin{align*}
		\int_{(0, \epsilon)} s^2 \:\mu_t({\rm d} s) 
		&\gtrsim t \epsilon^2 \vphi\big(\epsilon^{-1} \lambda^{-1}\big) \\
		&\gtrsim t \epsilon^{2-\alpha} \vphi(1/\lambda),
	\end{align*}
	which, together with the definition of $\lambda$, implies \eqref{eq:63}.

	Since the support of the probability distribution of $Y_0$ is not the whole real line, 
	by \cite[Lemma 2.5]{MR1486930}, the inequality \eqref{eq:64} implies that 
	\begin{equation}
		\label{eq:127}
		\int_{(0, \infty)} \min\{1, s\} \: \mu_0({\rm d} s) < \infty
	\end{equation}
	and the support of $Y_0$ equals $[\chi, \infty)$ where
	\begin{equation}
		\label{eq:127a}
		\chi = \gamma_0 - \int_{(0, \infty)} sw(s) \: \mu_0({\rm d} s) = -\int_{(0, 1]} s \: \mu_0({\rm d} s).
	\end{equation}
	To conclude \eqref{eq:35}, it is enough to show that $\chi \leq -\tilde{\rho}$. Since
	$\rho(t_n) \leq \frac{1}{2}t_n \lambda_n^{-1} b_{\lambda_n}$, the latter can be deduced from
	\begin{equation}
		\label{eq:79}
		\begin{aligned}
		\chi 
		&= -\lim_{n \to \infty} \frac{1}{2} t_n \lambda_n^{-1} b_{\lambda_n} \\
		&= -\lim_{n \to \infty} \int_{(0, 1]} s \: \mu_{t_n}({\rm d} s)
		\end{aligned}
	\end{equation}
	where the last equality is a consequence of \eqref{eq:53} since
	\begin{align}
		\int_{(0, 1]} s \: \mu_t (\rm{d} s) \nonumber 
		&= 
		t\lambda^{-1} \int_{(0, \lambda]} s \: \nu_1({\rm d}s) \\ 
		\label{eq:103}
		&= 
		\tfrac{1}{2} t \lambda^{-1}\int_{(0, \lambda] }s \: \nu({\rm d}s). 
	\end{align}
	Therefore, the problem is reduced to showing \eqref{eq:79}. By the monotone convergence theorem and \eqref{eq:81} we have
	\begin{equation}
		\label{eq:60}
		\begin{aligned}
		\chi &= -\lim_{\epsilon \to 0^+} \int_{(0,1]} f_{\epsilon}(s) \: \mu_0({\rm d}s) \\
		&= -\lim_{\epsilon \to 0^+} \lim_{n \to \infty}\int_{(0,1]} f_{\epsilon}(s) \: \mu_{t_n}({\rm d}s),
		\end{aligned}
	\end{equation}
	and
	\begin{equation}
		\label{eq:80}
		\lim_{\epsilon \to 0^+}
		\int_{(0,1]} f_{\epsilon}(s) \: \mu_{t_n}({\rm d}s)
		=
		\int_{(0, 1]} s \: \mu_{t_n}({\rm d} s),
	\end{equation}
	where $f_{\epsilon}$ is as in \eqref{eq:65}. Hence, we just need to justify the change in the order of limits. In view of the Moore--Osgood theorem \cite[Chapter VII]{Graves}, 
	it is enough to show that the limit in \eqref{eq:80} is uniform with respect to $n \in \NN$. 

	We write
	\begin{align*}
		\bigg|
		\int_{(0, 1]} s \: \mu_t({\rm d} s)
		-
		\int_{(0, 1]} f_\epsilon(s) \: \mu_t({\rm d} s)
		\bigg|
		&\leq
		\int_{(0, \epsilon/2]} s \: \mu_t({\rm d} s) 
		+
		\int_{(\epsilon/2, \epsilon]} (\epsilon-s) \: \mu_t({\rm d} s) \\
		&\leq
		\int_{(0, \epsilon]} s \: \mu_t({\rm d} s).
	\end{align*}	
	By \eqref{eq:103} and the Fubini--Tonelli theorem, we have
	\begin{align*}
		2 t^{-1} \lambda 
		\int_{(0, \epsilon]} s \: \mu_t({\rm d} s)
		=
		\int_{(0, \lambda \epsilon]}s \: \nu({\rm d}s) 
		&= 
		\int_0^{\lambda \epsilon} u^{-2} \int_{(0,u]} s^2 \: \nu({\rm d} s) \rmd u 
		+
		\lambda \epsilon K (\lambda \epsilon) \\ 
		&\approx 
		\int_0^{\lambda \epsilon} \varphi\big(u^{-1}\big) \rmd u
		+
		\lambda \epsilon \varphi\big(\lambda^{-1}\epsilon^{-1}\big).
	\end{align*}
	By almost monotonicity of $\vphi$,
	\begin{align}
		\nonumber
		\int_{(0, \epsilon]} s \: \mu_t ({\rm d} s) 
		&
		\approx
		t \lambda^{-1} \int_0^{\lambda \epsilon} \varphi\big(u^{-1}\big) \rmd u
		+
		t \epsilon \varphi\big(\lambda^{-1}\epsilon^{-1}\big) \\
		\label{eq:124}
		&\approx
		t \lambda^{-1} \int_0^{\lambda \epsilon} \varphi\big(u^{-1}\big) \rmd u.
	\end{align}
	Now, setting $z = \varphi^{-1}(M/t)$, by \eqref{eq:123}, we get
	\begin{align}
		\nonumber
		t \lambda^{-1} \int_0^{\lambda \epsilon} \varphi\big(u^{-1}\big) \rmd u
		&= 
		t\varphi^{-1}(M/t)
		\int_0^{\epsilon/\varphi^{-1}(M/t)} \varphi \big(u^{-1}\big) \rmd u \\
		\nonumber
		&\approx 
		\frac{z}{\varphi(z)}\int_0^{\epsilon z^{-1}} \varphi \big( u^{-1} \big) \rmd u \\
		\nonumber
		&= 
		\frac{z}{\varphi(z)}\int_{\epsilon^{-1} z}^{\infty} u^{-2}\varphi(u) \rmd u \nonumber \\
		\label{eq:126}
		&= \frac{\phi'\big(\epsilon^{-1} z\big)}{z \big(-\phi''(z)\big)}.
	\end{align}
	In view of Proposition \ref{prop:WUSC}, by the upper scaling of $-\phi''$, there is $c > 0$ such that for all
	$z > x_0$, 
	\[
		\frac{\phi'\big(\epsilon^{-1} z\big)}{z \big(-\phi''(z)\big)}
		\leq c \epsilon^{1-\beta}.
	\]
	Hence, the limit in \eqref{eq:80} is uniform with respect to $n \in \NN$ which justifies
	\eqref{eq:79}. This completes the proof of \eqref{eq:35} and the lemma follows.
\end{proof}

\begin{theorem}
	\label{thm:7}
	Let $\bfT$ be a subordinator with the Laplace exponent $\phi$. Suppose that $\phi \in \WLSC{\alpha}{c}{x_0}
	\cap \WUSC{\beta}{C}{x_0}$ for some $c \in (0, 1]$, $C \geq 1$, $x_0 \geq 0$, and $0 < \alpha \leq \beta < 1$. 
	We also assume that $b=0$. Then for all $0<\chi_1<\chi_2$ there is $C' \geq 1$ such that for all
	$t \in (0, 1/\varphi(x_0))$ and $x > 0$ satisfying
	\[
		\chi_1 \leq x \phi^{-1}(1/t) \leq \chi_2,
	\]
	we have
	\begin{equation}
		\label{eq:11}
		C'^{-1} \phi^{-1}(1/t) \leq p(t, x) \leq C' \phi^{-1}(1/t).
	\end{equation}	
\end{theorem}
\begin{proof}
	First let us notice that Corollary \ref{cor:6} implies that $-\phi'' \in \WLSC{\alpha-2}{c}{x_0}
	\cap \WUSC{\beta-2}{C}{x_0}$. Therefore, the hypothesis of Theorem \ref{thm:2} is satisfied.

	It is enough to show the first inequality in \eqref{eq:11} since the latter is an easy consequence of 
	\eqref{eq:98} and Proposition \ref{prop:1}. For $t \in (0, 1/\varphi(x_0))$ and $M \geq 1$, we set
	\[
		x_t =  t  \phi' \big( \varphi^{-1}(M/t) \big).
	\]
	By Proposition \ref{prop:1}, the function $\vphi^{-1}$ possesses the weak lower scaling property. Moreover,	
	there is $C_1 \geq 1$ such that for all $r > \max{\big\{\vphi^*(x_0), \phi(x_0)\big\}}$,
	\begin{equation}
		\label{eq:96}
		C_1^{-1} \vphi^{-1}(r) \leq \phi^{-1}(r) \leq C_1 \vphi^{-1}(r).
	\end{equation}
	Hence, by Proposition \ref{prop:WUSC}, there is $C_2 \geq 1$, such that
	\begin{equation}
		\label{eq:92}
		x_t \leq C_2 M^{1-1/\beta} \frac{1}{\vphi^{-1}(1/t)}.
	\end{equation}
	We select $M \geq 1$ satisfying
	\[
		C_1 C_2 M^{1-1/\beta} < \chi_1.
	\]
	Let $\rho_1 = \rho_0/2$ where $\rho_0$ is determined in Theorem \ref{thm:2}. Then, by \eqref{eq:96}
	and \eqref{eq:92}, we have
	\begin{align}
	\nonumber
		x_t - \frac{\rho_1}{\vphi^{-1}(1/t)}
		&\leq
		C_1 C_2 M^{1-1/\beta} \frac{1}{\phi^{-1}(1/t)} \\
		\label{eq:93}
		&< \frac{\chi_1}{\phi^{-1}(1/t)}.
	\end{align}
	Now set $\rho_2 = C_1 \chi_2$. Then, by \eqref{eq:96}, we have
	\begin{equation}
		\label{eq:94}
		x_t + \frac{\rho_2}{\vphi^{-1}(1/t)} 
		>
		\frac{\rho_2}{C_1 \phi^{-1}(1/t)} = \frac{\chi_2}{\phi^{-1}(1/t)}.
	\end{equation}
	Putting \eqref{eq:94} and \eqref{eq:93} together, we conclude that
	\[
	\bigg[\frac{\chi_1}{\phi^{-1}(1/t)}, \frac{\chi_2}{\phi^{-1}(1/t)} \bigg]
	\subseteq
	\(
	x_t - \frac{\rho_1}{\vphi^{-1}(1/t)}, x_t + \frac{\rho_2}{\vphi^{-1}(1/t)}
	\).
	\]
	Therefore, by Theorem \ref{thm:2}, for all $t \in (0, 1/\vphi(x_0))$ and $x > 0$ satisfying
	\[
	\chi_1 \leq x \phi^{-1}(1/t) \leq \chi_2,
	\]
	we have
	\[
	p(t,x) \gtrsim \vphi^{-1}(1/t).
	\]
	In view of \eqref{eq:96}, this completes the proof of the theorem. 
\end{proof}

\begin{proposition}
	\label{prop:4}
	Let $\bfT$ be a subordinator with the Laplace exponent $\phi$. Suppose that $-\phi'' \in \WLSC{\alpha-2}{c}{x_0}$
	for some $c \in (0, 1]$, $x_0 \geq 0$, and $\alpha>0$, and assume that one of the following conditions holds true:
	\begin{enumerate}
		\item $-\phi'' \in \WUSC{\beta-2}{C}{x_0}$ for some $C\geq 1$ and $\alpha \leq \beta <1$, or
		\item $-\phi''$ is a function regularly varying at infinity with index $-1$. If $x_0 = 0$, we also assume
		that $-\phi''$ is regularly varying at zero with index $-1$.
	\end{enumerate}
	We also assume that the L\'{e}vy measure $\nu({\rm d} x)$ has an almost monotone density $\nu(x)$. Then
	the probability distribution of $T_t$ has a density $p(t, \: \cdot \:)$. Moreover, there are $\rho_0 > 0$, and
	$C > 0$ such that for all $t \in (0, 1/\varphi(x_0))$ and
	\[
		x \geq 2t\phi'\big(\vphi^{-1}(1/t) \big) + \frac{2 \rho_0}{\vphi^{-1}(1/t)},
	\]
	we have
	\[
		p(t, x) \geq C t \nu(x).
	\]
\end{proposition}
\begin{proof}
	Let $\lambda > 0$. We begin by decomposing the L\'{e}vy measure $\nu({\rm d} x)$. Let 
	$\nu_1({\rm d} x) = \nu_1(x) \rmd x$ and $\nu_2({\rm d}x) = \nu_2(x) \rmd x$ where
	\[
		\nu_1(x) = \nu(x) - \nu_2(x), \qquad\text{and}\qquad
		\nu_2(x) = \tfrac{1}{2} \nu(x) \ind{[\lambda, \infty)}(x).
	\]
	For $u > 0$, we set
	\[
		\phi_1(u) = b u + \int_{(0, \infty)} \big(1-e^{-u s}\big) \: \nu_1({\rm d} s),
		\qquad\text{and}\qquad
		\phi_2(u) = \int_{(0, \infty)} \big(1 - e^{-u s}\big) \: \nu_2({\rm d} s).
	\]
	Let $\bfT^{(j)}$ be the L\'{e}vy process having the Laplace exponent $\phi_j$, for $j \in \{1, 2\}$.
	Since $\tfrac{1}{2} \nu \leq \nu_1 \leq \nu$, we have
	\[
		\tfrac{1}{2} \phi \leq \phi_1 \leq \phi,
	\]
	and for all $n \in \NN$,
	\begin{equation}
		\label{eq:73}
		\tfrac{1}{2} (-1)^{n+1} \phi^{(n)} 
		\leq (-1)^{n+1} \phi_1^{(n)}
		\leq (-1)^{n+1} \phi^{(n)}.
	\end{equation}
	Thus
	\[
		\tfrac{1}{2} \vphi \leq \vphi_1 \leq \vphi,
	\]
	and so for all $u > 0$,
	\begin{equation}
		\label{eq:77}
		\vphi^{-1}_1(u/2) \leq \vphi^{-1}(u) \leq \vphi^{-1}_1(u).
	\end{equation}
	In particular, $-\phi_1''$ has the weak lower scaling property. Therefore, by Theorem \ref{thm:3}, 
	$T_t^{(1)}$ and $T_t$ are absolutely continuous. Let us denote by $p(t, \:\cdot\:)$ and $p^{(1)}(t, \:\cdot\:)$
	the densities of $T_t$ and $T_t^{(1)}$, respectively. Observe that $\bfT^{(2)}$ is a compound Poisson process with
	the probability distribution denoted by $P_t({\rm d} x)$. By \cite[Remark 27.3]{MR1739520},
	\begin{equation}
		\label{eq:71}
		P_t({\rm d} x) \geq t e^{-t\nu_2(\RR)} \nu_2(x) \rmd x. 
	\end{equation}
	We apply Theorem \ref{thm:2} to the process $\bfT^{(1)}$. For $t > 0$, we set
	\[
		x_t = t \phi'_1\big(\vphi_1^{-1}(1/t)\big).
	\]
	Then there are $C > 0$ and $\rho_0 > 0$, such that for all $t \in (0, 1/\varphi(x_0))$ and $x \geq 0$ satisfying
	\[
		x_t - \frac{\rho_0}{\vphi_1^{-1}(1/t)}
		\leq
		x
        \leq
		x_t + 
        \frac{\rho_0}{\vphi_1^{-1}(1/t)},
	\]
    we have
    \[
        p^{(1)}(t, x) \geq C \vphi_1^{-1}(1/t).
    \]
	Therefore, if
	\[
		\lambda = x_t + \frac{\rho_0}{\vphi_1^{-1}(1/t)},
	\]
	then
	\begin{equation}
		\label{eq:70}
		\int_0^\lambda p^{(1)}(t, x) \rmd x  
		\gtrsim 1.
	\end{equation}
	Next, if $\lambda \geq \rho_0/\vphi^{-1}(1/t)$ then, by \eqref{eq:33},
	\begin{align}
		\nonumber
		t \nu_2(\RR) 
		&= \tfrac{1}{2} t \int_\lambda^\infty \: \nu(x) \rmd x \\
		\nonumber
		&\leq \tfrac{1}{2} t h\big(\rho_0/\vphi^{-1}(1/t)\big) \\
		\nonumber
		&\lesssim t h\big(1/\vphi^{-1}(1/t)\big)\\
		\label{eq:69}
		&\lesssim 1,
	\end{align}
	where the penultimate inequality follows either by monotonicity of $h$ or by \cite[Lemma 2.1 (4)]{GS2019}.
	Finally, by \eqref{eq:71} and \eqref{eq:69}, for $x \geq 2 \lambda$ we can compute
	\begin{align*}
		p(t, x) 
		&= \int_\RR p^{(1)}(t, x - y) P_t({\rm d} y) \\
		&\gtrsim t \int_\RR p^{(1)}(t, x-y) \nu_2(y) \rmd y \\
		&=\tfrac{1}{2} t \int_\lambda^x p^{(1)}(t, x-y) \nu(y) \rmd y.
	\end{align*}
	Hence, by the monotonicity of $\nu$, we get
	\begin{align*}
		p(t, x) &\gtrsim 
		t \nu(x) \int_0^{x - \lambda} p^{(1)}(t, y) \rmd y\\
		&\geq
		t \nu(x) \int_0^\lambda p^{(1)}(t, y) \rmd y \\
		&\gtrsim
		t \nu(x),
	\end{align*}
	where in the last estimate we have used \eqref{eq:70}. Using \eqref{eq:73}, and \eqref{eq:77},
	we can easily show that
	\[
    	\lambda = x_t + \frac{\rho_0}{\vphi_1^{-1}(1/t)} \leq  t \phi'\big(\vphi^{-1}(1/t)\big)
		+\frac{\rho_0}{\vphi^{-1}(1/t)},
	\]
	and the proposition follows.
\end{proof}

\subsection{Sharp two-sided estimates}
In this section we present sharp two-sided estimates on the density $p(t, \: \cdot\:)$ assuming both the weak lower
and upper scaling properties on $-\phi''$. First, following \cite[Lemma 13]{MR3165234}, we prove an auxiliary result.

\begin{proposition}\label{prop:5}
	Assume that the L\'{e}vy measure $\nu({\rm d}x)$ has an almost monotone density $\nu(x)$. Suppose that 
	$-\phi'' \in \WUSC{\gamma}{C}{x_0}$ for some $C \geq 1$, $x_0 \geq 0$ and $\gamma <0$. Then there are $a \in (0,1]$
	and $c \in (0,1]$ such that for all $0 < x < a/x_0$,
	\[
		\nu(x) \geq c x^{-3} \big( -\phi''(1/x) \big).
	\]
\end{proposition}
\begin{proof}
	Let $a \in (0,1]$. Recall that by \eqref{eq:30a} we have $\nu(s) \leq C_1 s^{-3} \big( -\phi''(1/s) \big)$ for any $s>0$. Hence, for any $u>0$,
	\begin{align}
		\nonumber
		-\phi''(u) 
		&= \int_0^{au^{-1}} s^2e^{-us}  \nu(s) \rmd s 
		+ \int_{au^{-1}}^{\infty} s^2e^{-us}  \nu(s) \rmd s \\
		\label{eq:36}
		&\leq C_1 \int_0^{au^{-1}} s^{-1}e^{-us} \big( -\phi''(1/s) \big) \rmd s 
		+ C_2 \nu(au^{-1}) \int_{au^{-1}}^{\infty} s^2e^{-us} \rmd s
	\end{align}
	where $C_2$ is a constant from the almost monotonicity of $\nu$. If $u > x_0$, then by the scaling property of
	$-\phi''$ we obtain
	\begin{align*}
		C_1 \int_0^{au^{-1}} s^{-1}e^{-us} \big( -\phi''(1/s) \big) \rmd s 
		&\leq C\int_0^{au^{-1}} s^{-1}e^{-us} (su)^{-\gamma} \big( -\phi''(u) \big) \rmd s \\
		&\leq C \big( -\phi''(u) \big) \int_0^as^{-1-\gamma} e^{-s} \rmd s.
	\end{align*}
	By selecting $a \in (0, 1]$ such that
	\[
		2C \int_0^a s^{-1-\gamma}e^{-s} \rmd s \leq 1,
	\]
	we get
	\[
		\int_0^{au^{-1}} s^{-1}e^{-us} \big( -\phi''(1/s) \big) \rmd s
		\leq
		\tfrac{1}{2} \big( -\phi''(u) \big).
	\]
	Since
	\[
		\int_{au^{-1}}^{\infty} s^2 e^{-us} \rmd s = u^{-3} e^{-a}(a^2+2a+2),
	\]
	by \eqref{eq:36}, we obtain 
	\[
		\nu(au^{-1}) \geq \frac{e^a}{2(a^2+2a+2)}u^3 \big( -\phi''(u) \big),
	\]
	provided that $u > x_0$. Now, by the monotonicity of $-\phi''$ we conclude the proof.
\end{proof}
In view of Propositions \ref{prop:WLSC} and \ref{prop:WUSC}, we immediately obtain the following corollary.
\begin{corollary}
	\label{cor:4}
	Assume that the L\'{e}vy measure $\nu({\rm d}x)$ has an almost monotone density $\nu(x)$. Suppose that $b=0$ and
	$\phi \in \WLSC{\alpha}{c}{x_0} \cap \WUSC{\beta}{C}{x_0}$ for some $c \in (0,1]$, $C \geq 1$, $x_0 \geq 0$
	and $0<\alpha \leq \beta < 1$. Then there are $a \in (0,1]$ and $c' \in (0,1]$ such that for all $0 < x<a/x_0$,
	\[
		\nu(x) \geq c' x^{-1}\phi(1/x).
	\]
\end{corollary}
We are now ready to prove our main result in this section.
\begin{theorem}
	\label{thm:5}
	Let $\bfT$ be a subordinator with the Laplace exponent $\phi$. Suppose that $\phi \in \WLSC{\alpha}{c}{x_0}
	\cap \WUSC{\beta}{C}{x_0}$ for some $c \in (0, 1]$, $C \geq 1$, $x_0 \geq 0$, and $0 < \alpha \leq \beta < 1$. 
	We also assume that $b=0$ and that the L\'{e}vy measure $\nu({\rm d} x)$ has an almost monotone density $\nu(x)$.
	Then  there is
	$x_1 \in (0, \infty]$ such that for all $t \in (0, 1/\varphi(x_0))$ and $x \in (0,x_1)$,
	\[
		p(t, x) \approx
		\begin{cases}
			\big(t (-\phi''(w)) \big)^{-\frac{1}{2}} \exp\big\{-t \big(\phi(w) - w \phi'(w)\big) \big\}
			& \text{if } 0 < x \phi^{-1}(1/t) \leq 1, \\
			 t x^{-1}\phi(1/x)
			& \text{if } 1 <  x \phi^{-1}(1/t),
		\end{cases}
	\]
	where $w = (\phi')^{-1}(x/t)$. If $x_0 = 0$ then $x_1 = \infty$.
\end{theorem}
\begin{proof}
	First let us note that by Corollary \ref{cor:6}, $-\phi'' \in \WLSC{\alpha-2}{c}{x_0}
	\cap \WUSC{\beta-2}{C}{x_0}$. Therefore, we are in position to apply Proposition \ref{prop:4}. By Corollary \ref{cor:2}, for $\chi_1 = \min{\{1, \delta\}}$, we have
	\[
		p(t, x)
		\approx \big(t (-\phi''(w)) \big)^{-\frac{1}{2}} \exp\big\{-t \big(\phi(w) - w \phi'(w)\big) \big\},
	\]
	whenever $0 < x \phi^{-1}(1/t) \leq \chi_1$. Next, by Proposition \ref{prop:WUSC} and \eqref{eq:33}, for
	$t \in (0, 1/\varphi(x_0))$, we get
	\[
		t \phi'\big(\psi^{-1}(1/t)\big) \lesssim  \frac{1}{\psi^{-1}(1/t)},
	\]
	thus, by Propositions \ref{prop:7} and \ref{prop:1}, there is $C_1 > 0$ such that
	\begin{equation*}
		2 e t \phi'\big(\psi^{-1}(1/t)\big) 
		+
		\frac{2 \rho'_0}{\vphi^{-1}(1/t)}
		\leq C_1 \frac{1}{\phi^{-1}(1/t)}
	\end{equation*}
	where $\rho_0'$ is the value of $\rho_0$ determined in Proposition \ref{prop:4}. Let 
	$\chi_2 = \max{\{1, C_1, \chi_1\}}$. By Proposition \ref{prop:4}, and Corollary \ref{cor:4}, there is $a \in (0,1]$
	such that if $x \phi^{-1}(1/t) > \chi_2$ and $0 < x< a/x_0$, then 
	\begin{align*}
		p(t, x) 
		&\gtrsim t \nu(x) \\ 
		&\gtrsim tx^{-1}\phi(1/x).
	\end{align*}
	Furthermore, by \eqref{eq:97}, if $x \phi^{-1}(1/t) > \chi_2$, then
	\begin{align*}
		p(t,x) &\lesssim t\eta (x) \\ &\lesssim tx^{-1}\phi(1/x)
	\end{align*} 
	where in the last step we have also used \eqref{eq:99}. Lastly, by Theorem \ref{thm:7} there is $C_2 \geq 1$
	such that for all $t \in (0, 1/\varphi(x_0))$ and $x > 0$ satisfying
	\[
		\chi_1 \leq x \phi^{-1}(1/t) \leq \chi_2,
	\]
	we have
	\begin{equation}
		\label{eq:129}
		C_2^{-1} \phi^{-1}(1/t) \leq p(t, x) \leq C_2 \phi^{-1}(1/t).
	\end{equation}
	We next claim that the following holds true.
	\begin{claim}
		\label{clm:3}
		There exist $0< c_1 \leq 1 \leq c_2$ such that for all $t \in (0,c_1/\vphi(x_0))$ and $x>0$ satisfying
		\[
			\chi_1 \leq x\phi^{-1}(1/t) \leq \chi_2,
		\]
		we have
		\begin{equation}
			\label{eq:106}
			t\phi' \big( \phi^{-1}(c_2/t) \big) \leq x \leq t \phi' \big( \phi^{-1}(c_1/t) \big).
		\end{equation}
	\end{claim}
	By Proposition \ref{prop:1}, there is $C_3 \geq 1$ such that for $r > \vphi(x_0)$,
	\[
		C_3^{-1} \vphi^{-1}(r) \leq \phi^{-1}(r) \leq C_3 \vphi^{-1}(r).
	\]
	Let $c_2 = (\chi_1 c' C_3^{-2})^{-\beta/(1-\beta)} \in [1, \infty)$, where $c'$ is taken from \eqref{eq:206}.
	Then
	\[
		c_2^{-1}
		\phi^{-1}(c_2/t) 
		\geq C_3^{-2} c' c_2^{-1+1/\beta} \phi^{-1}(1/t) = \chi_1^{-1} \phi^{-1}(1/t).
	\]
	Consequently, by Proposition \ref{prop:WLSC},
	\begin{align}
		\nonumber
		x 
		\geq \frac{\chi_1}{\phi^{-1}(1/t)} 
		&\geq t \frac{\phi \big( \phi^{-1} ( c_2/t) \big)}{\phi^{-1} (c_2/t)} \\
		\label{eq:107}
		&\geq t \phi' \big( \phi^{-1}(c_2/t) \big).
	\end{align}
	Moreover, there is $C_4 \geq 1$ such that $C_4 x\phi'(x) \geq \phi(x)$ provided that $x > x_0$. Therefore,
	if $\chi_2 \leq C_4^{-1}$, then
	\begin{align}
		\nonumber
		\frac{\chi_2}{\phi^{-1}(1/t)} 
		&= \chi_2 t \frac{\phi\big(\phi^{-1}(1/t)\big)}{\phi^{-1}(1/t)} \\
		\label{eq:108}
		&\leq
		t \phi'\big( \phi^{-1}(1/t) \big),
	\end{align}
	which yields \eqref{eq:106} with $c_1 = 1$. Otherwise, if $\chi_2 > C_4^{-1}$, then we set 
	$c_1 =\big(C_4 \chi_2 C_3^2 (c')^{-1}\big)^{-\beta/(1-\beta)} \in (0,1]$. Hence, by Proposition \ref{prop:1}, for all $t \in (0,c_1/\vphi(x_0))$,
	\begin{align*}
		\frac{C_4 \chi_2}{c_1} \phi^{-1}(c_1/t) 
		\leq
		C_4 \chi_2 C_3^2 (c')^{-1} c_1^{-1+1/\beta} \phi^{-1}(1/t) = \phi^{-1}(1/t).
	\end{align*}
	Therefore,
	\begin{align*}
		x 
		&\leq \frac{\chi_2}{\phi^{-1}(1/t)}  \\
		&\leq t 
		\frac{\chi_2}{c_1} \cdot \frac{\phi^{-1}(c_1/t)}{\phi^{-1}(1/t)} 
		\cdot \frac{\phi \big( \phi^{-1}(c_1/t) \big)}{\phi^{-1}(c_1/t)} \\
		&\leq t\phi' \big( \phi^{-1}(c_1/t) \big),
	\end{align*}
	which combined with \eqref{eq:107} and \eqref{eq:108}, implies \eqref{eq:106}.
	
	Now, using Claim \ref{clm:3} and Propositions \ref{prop:7} and \ref{prop:1} we deduce that for $t \in (0,c_1/\vphi(x_0))$ and
	$\chi_1 \leq x \phi^{-1}(1/t) \leq \chi_2$,
	\begin{equation}
		\label{eq:109}
		w \leq \phi^{-1}(c_2/t) \lesssim \phi^{-1}(1/t), 
	\end{equation}
	and
	\begin{equation}\label{eq:110}
		w \geq \phi^{-1}(c_1/t) \gtrsim \phi^{-1}(1/t).
	\end{equation}
	Hence, $tw\phi'(w) \approx 1$ and
	\begin{equation}
		\label{eq:111}
		\exp\big\{-t \big(\phi(w) - w \phi'(w)\big) \big\} \approx 1.
	\end{equation}
	Next, by Propositions \ref{prop:WUSC} and \ref{prop:2},
	\[
		w^2 \big( -\phi''(w) \big) \approx w\phi'(w),
	\]
	thus, by \eqref{eq:109} and \eqref{eq:110}, we obtain
	\[
		\frac{1}{\sqrt{t\big( -\phi''(w) \big)}} \approx \frac{w}{\sqrt{tw\phi'(w)}} \approx \phi^{-1}(1/t),
	\]
	which, together with \eqref{eq:111}, implies that 
	\[
		\big(t (-\phi''(w)) \big)^{-\frac{1}{2}} \exp\big\{-t \big(\phi(w) - w \phi'(w)\big) \big\} \approx \phi^{-1}(1/t),
	\]
	for $t \in (0,c_1/\vphi(x_0))$ and $\chi_1 \leq x\phi^{-1}(1/t) \leq \chi_2$. In view of \eqref{eq:129}, the theorem follows in the case $x_0=0$. Now, it remains to observe that in the case $x_0>0$ we may use positivity and continuity to conclude the claim for all $t \in (0,1/\vphi(x_0))$.
\end{proof}
\section{Applications}\label{sec:5}
\subsection{Subordination}
Let $(\scrX, \tau)$ be a locally compact separable metric space with a Radon measure $\mu$ having full support on $\scrX$.
Assume that $(X_t \colon t \geq 0)$ is a homogeneous in time Markov process on $\scrX$ with density $h(t, \: \cdot \:,\: \cdot \:)$, that is
\[
	\PP\(\left. X_t \in B \right| X_0 = x\) = \int_B h(t, x, y) \: \mu({\rm d} y)
\]
for any Borel set $B \subset \scrX$, $x \in \scrX$ and $t > 0$. Assume that for all $t > 0$ and $x, y \in \scrX$,
\begin{equation}
	\label{eq:115}
	t^{-\frac{n}{\gamma}} \Phi_1\Big(\tau(x, y) t^{-\frac{1}{\gamma}} \Big)
	\leq
	h(t, x, y) 
	\leq
	t^{-\frac{n}{\gamma}} \Phi_2\Big(\tau(x, y) t^{-\frac{1}{\gamma}} \Big)
\end{equation}
where $n$ and $\gamma$ are some positive constants, $\Phi_1$ and $\Phi_2$ are nonnegative nonincreasing function on
$[0, \infty)$ such that $\Phi_1(1) > 0$ and
\begin{equation}
	\label{eq:121}
	\sup_{s \geq 0}{\Phi_2(s) (1+s)^{n + \gamma}} < \infty.
\end{equation}
By $H(t, x, y)$ we denote the heat kernel for the subordinate process $\(X_{T_t} \colon t \geq 0\)$, that is
\[
	H(t, x, y) = \int_0^\infty h(s, x, y) G(t, {\rm d} s),
\]
where
\[
	G(t, s) = \PP\big(T_t \geq s \big).
\]
Suppose that $\phi \in \WLSC{\alpha}{c}{x_0} \cap \WUSC{\beta}{C}{x_0}$ for some $c \in (0, 1]$,
$C \geq 1$, $x_0 > 0$, and $0 < \alpha \leq \beta < 1$. We also assume that
\[
	\lim_{x \to \infty} \phi'(x) = b = 0,
\]
and that the L\'{e}vy measure $\nu({\rm d} x)$ has an almost monotone density $\nu(x)$. 
\begin{claim}
	\label{clm:2}
	For all $x, y \in \scrX$ satisfying $\tau(x, y)^{-\gamma} > x_0$, and any $t \in (0, 1/\vphi(x_0))$,
	\[
		H(t, x, y) 
		\approx
		\begin{cases}
		t \phi\big(\tau(x,y)^{-\gamma} \big) \tau(x, y)^{-n} & \text{if } 
		0 < t \phi\big(\tau(x, y)^{-\gamma}\big) \leq 1, \\
		\big(\phi^{-1}(1/t) \big)^{\frac{n}{\gamma}} & \text{if } 1 \leq t \phi\big(\tau(x, y)^{-\gamma}\big).
		\end{cases}
	\]
\end{claim}
By Proposition \ref{prop:WLSC}, $\phi' \in \WLSC{\alpha-1}{c}{x_0} \cap \WUSC{\beta-1}{C}{x_0}$.
Let $0 < r < \phi'(x_0^+)$. If $0 < \lambda \leq C$ then by setting
\[
	D = C^\frac{1}{1-\beta} \lambda^{-\frac{1}{1-\beta}},
\]
the weak upper scaling property of $\phi'$ implies that
\[
	\lambda r = 
	\lambda \phi'\big((\phi')^{-1} (r)\big) \geq \phi'\big(D (\phi')^{-1} (r) \big).
\]
Therefore,
\begin{equation}
	\label{eq:104}
	(\phi')^{-1}(\lambda r) \leq C^\frac{1}{1-\beta} \lambda^{-\frac{1}{1-\beta}} (\phi')^{-1}(r).
\end{equation}
Analogously, we can prove the lower estimate: If $0 < \lambda \leq c$ then by setting
\[
	D = c^{\frac{1}{1-\alpha}} \lambda^{-\frac{1}{1-\alpha}},
\]
we obtain
\[
	\lambda r = \lambda \phi'\big((\phi')^{-1} (r)\big) \leq \phi'\big(D (\phi')^{-1} (r) \big),
\]
and consequently,
\begin{equation}
	\label{eq:105}
	(\phi')^{-1}(\lambda r) \geq c^{\frac{1}{1-\alpha}} \lambda^{-\frac{1}{1-\alpha}} (\phi')^{-1}(r).
\end{equation}
Since $(\phi')^{-1}$ is nonincreasing the last inequality is valid for all $0 < \lambda \leq 1$. Let 
\begin{align*}
	H(t, x, y) &= 
	\left(
	\int_0^{\frac{1}{\phi^{-1}(1/t)}}
	+
	\int_{\frac{1}{\phi^{{-1}(1/t)}}}^\infty
	\right)
	h(s, x, y) G(t, {\rm d} s) \\
	&=
	I_1(t, x, y) + I_2(t, x, y).
\end{align*}
By Theorem \ref{thm:5},
\begin{align}\label{eq:114a}
	I_1
	\approx \frac{1}{\phi^{-1}(1/t)} \int_0^{1} 
	h\bigg(\frac{u}{\phi^{-1}(1/t)}, x , y\bigg) 
	\frac{1}{\sqrt{t (-\phi''(w))}} \exp\Big(-t \big(\phi(w) - w\phi'(w)\big) \Big)
	\rmd u
\end{align}
where
\[
	w = (\phi')^{-1}\bigg(\frac{u}{t \phi^{-1}(1/t)} \bigg).
\]
Recall that, by Proposition \ref{prop:WLSC}, for all $r > x_0$ we have 
\begin{equation}
	\label{eq:114}
	r \phi'(r) \leq \phi(r) \leq C_1 r \phi'(r).
\end{equation}
We can assume that
\[
	t \phi\(2 (CC_1)^{\frac{1}{1-\beta}} x_0 \) < 1.
\]
By \eqref{eq:114} and the weak upper scaling of $\phi'$, we get
\begin{align*}
	\phi'\Big(\phi^{-1}(1/t)\Big) \leq
	\frac{1}{t \phi^{-1}(1/t)} 
	&\leq C_1 \phi'\big(\phi^{-1}(1/t)\big)\\
	&\leq \phi'\Big((CC_1)^{-\frac{1}{1-\beta}} \phi^{-1}(1/t)\Big),
\end{align*}
thus
\[
	(\phi')^{-1}\bigg(\frac{1}{t \phi^{-1}(1/t)} \bigg) \approx \phi^{-1}(1/t).
\]
Hence, by \eqref{eq:104} and \eqref{eq:105}, we obtain
\begin{equation}
	\label{eq:112}
	u^{-\frac{1}{1-\alpha}} \phi^{-1}(1/t) \lesssim w \lesssim u^{-\frac{1}{1-\beta}} \phi^{-1}(1/t),
	\quad u \in (0, 1].
\end{equation}
Moreover, since $w > x_0$, by \eqref{eq:114} and Proposition \ref{prop:1},
\begin{align*}
	w \phi'(w)
	\gtrsim
	\phi(w) - w \phi'(w) 
	&=
	\int_0^{w} \vphi(u) \frac{{\rm d} u}{u} \\ 
	&\geq \int_{x_0}^{w} \vphi(u) \frac{{\rm d} u}{u}
	\gtrsim w \phi'(w).
\end{align*}
Thus, \eqref{eq:112} entails that
\begin{equation}
	\label{eq:119}
	u^{-\frac{\alpha}{1-\alpha}} 
	\lesssim 
	t \big(\phi(w) - w \phi'(w)\big) 
	\lesssim u^{-\frac{\beta}{1-\beta}}, \quad u \in (0, 1].
\end{equation}
Next, by Proposition \ref{prop:1} and \eqref{eq:114}, we get
\begin{align*}
	\frac{1}{\sqrt{t (-\phi''(w))}} 
	\approx
	\frac{w}{\sqrt{t \phi(w)}} 
	\approx
	\sqrt{u^{-1} \phi^{-1}(1/t) w}.
\end{align*}
Therefore, by \eqref{eq:112},
\begin{equation}
	\label{eq:120}
	u^{-\frac{2-\alpha}{2(1-\alpha)}} \phi^{-1}(1/t) \lesssim \frac{1}{\sqrt{t (-\phi''(w))}}
	\lesssim
	u^{-\frac{2-\beta}{2(1-\beta)}} \phi^{-1}(1/t),
	\quad u \in (0, 1].
\end{equation}
Now, by \eqref{eq:114a} and \eqref{eq:115} together with \eqref{eq:119} and \eqref{eq:120}, we can estimate
\begin{align}
	\label{eq:107a}
	I_1 &\lesssim 
	\big(\phi^{-1}(1/t) \big)^{\frac{n}{\gamma}}
	\int_0^{1}
	\Phi_2\(u^{-\frac{1}{\gamma}} A^{\frac{1}{\gamma}}\)
	u^{-\frac{n}{\gamma} - \frac{2-\beta}{2(1-\beta)}} 
	\exp\Big(-C'' u^{-\frac{\alpha}{1-\alpha}} \Big) \rmd u,
	\intertext{and}
	\label{eq:107b}
	I_1 &\gtrsim 
	\big(\phi^{-1}(1/t) \big)^{\frac{n}{\gamma}}
	\int_0^{1} 
	\Phi_1\(u^{-\frac{1}{\gamma}} A^{\frac{1}{\gamma}}\)
	u^{-\frac{n}{\gamma} - \frac{2-\alpha}{2(1-\alpha)}} 
	\exp\Big(-C' u^{-\frac{\beta}{1-\beta}} \Big) \rmd u,
\end{align}
where
\[
	A = \tau(x, y)^\gamma \phi^{-1}(1/t).
\]
Suppose that $A \leq 1$. Since $\Phi_1$ and $\Phi_2$ are nonincreasing, by \eqref{eq:107a} and \eqref{eq:107b}, 
we easily see that
\[
	I_1 \approx \big(\phi^{-1}(1/t) \big)^{\frac{n}{\gamma}}.
\]
We also have
\begin{align*}
	I_2 \lesssim 
	\int_{\frac{1}{\phi^{-1}(1/t)}}^{\infty} 
	s^{-\frac{n}{\gamma}} p(t, s) \rmd s 
	\lesssim \big(\phi^{-1}(1/t) \big)^{\frac{n}{\gamma}}.
\end{align*}
Therefore,
\[
	H(t, x, y) \approx \big(\phi^{-1}(1/t) \big)^{\frac{n}{\gamma}}.
\]
We now turn to the case $A > 1$. By \eqref{eq:121} and \eqref{eq:107a},
\begin{align}
	\nonumber
	I_1 
	&\lesssim 
	\big(\phi^{-1}(1/t) \big)^{\frac{n}{\gamma}}
	A^{-\frac{n}{\gamma}-1}
	\int_0^{1}
	u^{-\frac{\beta}{2(1-\beta)}} \exp\Big(-C'' u^{-\frac{\alpha}{1-\alpha}} \Big) \rmd u \\
	\label{eq:78}
	&\lesssim
	A^{-1} \tau(x, y)^{-n}.
\end{align}
It remains to estimate $I_2$. Let us observe that for all $r > x_0$, if $u \geq 1$ then by the weak upper scaling of
$\phi$, we have
\[
	\phi(r) \leq \phi(r u) \leq C u^\beta \phi(r).
\]
On the other hand, if $0 < u \leq 1$ then by \eqref{eq:28} and the monotonicity of $\phi$, we get
\[
	u \phi(r) \leq \phi(r u) \leq \phi(r).
\]
Therefore, for all $u > 0$ and $r > x_0$,
\begin{equation}
	\label{eq:118}
	\min{\{1, u\}} \phi(r) \leq \phi(r u) \leq C \max{\{1, u^\beta\}} \phi(r).
\end{equation}
Since $\tau(x, y)^{-\gamma} > x_0$, by Theorem \ref{thm:5}, \eqref{eq:115}, and estimates \eqref{eq:118}, we get 
\begin{align*}
	I_2 &\lesssim
	t \phi\big(\tau(x, y)^{-\gamma}\big) \tau(x, y)^{-n}
	\int_{1/A}^\infty
	\Phi_2\(u^{-\frac{1}{\gamma}} \) 
	u^{-\frac{n}{\gamma}-1} \max{\left\{1, u^{-\beta} \right\}} \rmd u,
	\intertext{and}
	I_2 &\gtrsim
	t \phi\big(\tau(x, y)^{-\gamma}\big) \tau(x, y)^{-n}
	\int_{1/A}^\infty 
	\Phi_1\(u^{-\frac{1}{\gamma}}\)
	u^{-\frac{n}{\gamma} - 1} \min{\{1, u\} } \rmd u.
\end{align*}
By \eqref{eq:121}, we have
\[
	\int_0^1 \Phi_2\(u^{-\frac{1}{\gamma}} \) u^{-\frac{n}{\gamma}-\beta-1} \rmd u
	\lesssim
	\int_0^1 u^{-\beta} \rmd u < \infty,
\]
thus,
\[
	I_2 \approx t \phi\big(\tau(x, y)^{-\gamma}\big) \tau(x, y)^{-n}.
\]
Finally, since $A > 1$, by \eqref{eq:28}, we have
\[
	t \phi\big(\tau(x, y)^{-\gamma}\big) = t \phi\big(A^{-1} \phi^{-1}(1/t)\big) \geq A^{-1},
\]
hence, by \eqref{eq:78},
\[
	I_1 \lesssim t \phi\big(\tau(x, y)^{-\gamma}\big) \tau(x, y)^{-n},
\]
proving the claim.

\begin{example}
	Let $(\scrX, \tau)$ be a nested fractal with the geodesic metric on $\scrX$. Let $d_w$ and $d_f$ be the walk dimension
	and the Hausdorff dimension of $\scrX$, respectively. Let $(X_t \colon t \geq 0)$ be the diffusion on $\scrX$
	constructed in \cite[Section 7]{barlow1998}. By \cite[Theorem 8.18]{barlow1998}, the corresponding heat kernel
	satisfies \eqref{eq:115} with $n = d_f$, $\gamma = d_w$, and
	\[
		\Phi_1(s) = \Phi_2(s) = \exp\(-s^{\frac{\gamma}{\gamma-1}} \).
	\]
	Let $\bfT$ be a subordinator with the Laplace exponent
	\[
		\phi(s) = s^\alpha \log^\sigma(2+s),
	\]
	where $\alpha \in (0, 1)$ and $\sigma \in \RR$. Then, by Claim \ref{clm:2}, the process $(X_{T_t} \colon t \geq 0)$
	has density $H(t, x, y)$ such that for all $x, y \in \scrX$ and $t > 0$,
	\begin{itemize}
		\item if $t > \tau(x, y)^{\alpha\gamma} \log^{-\sigma}\(2 + \tau(x, y)^{-\gamma}\)$ then
	\[
		H(t, x, y) \approx t^{-\frac{n}{\alpha\gamma}} \log^{-\frac{\sigma n}{\alpha\gamma}}\big(2+t^{-1}\big),
	\]
		\item if $t < \tau(x, y)^{\alpha\gamma} \log^{-\sigma}\(2 + \tau(x, y)^{-\gamma}\)$ then
	\[
		H(t, x, y) \approx t \tau(x, y)^{-\alpha \gamma - n} \log^{\sigma}\big(2 + \tau(x, y)^{-\gamma}\big).
	\]
	\end{itemize}	
\end{example}

\begin{example}
	Let $(\scrX, \tau)$ be a complete manifold without boundary, having nonnegative Ricci curvature. Then
	by \cite{MR834612}, the heat kernel corresponding to the Laplace--Beltrami operator on $\scrX$ satisfies
	estimates \eqref{eq:115} with
	\[
		\Phi_1(s) = e^{-C_1 s^2}, \qquad \Phi_2(s) = e^{-C_2 s^2}.
	\]
	Now, one can take $\bfT$ with a L\'{e}vy-Khintchine exponent regularly varying at infinity with index $\alpha \in (0,1)$ and apply Claim \ref{clm:2} to obtain the asymptotic behavior of subordinate process.
\end{example}

\subsection{Green function estimates}\label{subsec:1}
Let $\bfT = (T_t \colon t \geq 0)$ be a subordinator with the Laplace exponent $\phi$. If $-\phi''$ has the weak lower scaling
property of index $\alpha-2$ for some $\alpha > 0$, then the probability distribution of $T_t$ has a density
$p(t, \: \cdot \:)$, see Theorem \ref{thm:3}. In this section we want to derive sharp estimates on the Green function
based on Sections \ref{sec:1} and \ref{sec:4}. Let us recall that the Green function is
\[
	G(x) = \int_0^{\infty} p(t,x) \rmd t, \qquad x>0.
\]
We set
\[
	f(x) = \frac{\vphi(x)}{\phi'(x)}, \qquad x>0.
\]
Let us denote by $f^{-1}$ the generalized inverse of $f$, i.e.
\[
	f^{-1}(x) = \sup \{ r>0\colon f^*(r)=x \}
\]
where
\[
	f^*(r) = \sup_{0<x\leq r}f(x).
\]
Notice that by \eqref{eq:20} and Proposition \ref{prop:WLSC}, for all $x > x_0$,
\begin{align}
	\label{eq:138}
	f^*(x) \lesssim x.
\end{align}
In view of \eqref{eq:33} and Proposition \ref{prop:7}, the function $\vphi$ is almost increasing, thus
by monotonicity of $\phi'$, $f$ is almost increasing as well. Therefore, there is $c_0 \in (0, 1]$ such that for all
$x > x_0$,
\begin{equation}
	\label{eq:74}
	c_0 f^*(x) \leq f(x) \leq f^*(x).
\end{equation}
Moreover, $f$ has the doubling property on $(x_0,\infty)$. Since $\vphi$ belongs to $\WLSC{\alpha}{c}{x_0}$, by monotonicity of $\phi'$, we conclude that $f$ belongs 
	to $\WLSC{\alpha}{c}{x_0}$. It follows that $f^{-1} \in \WUSC{1/\alpha}{C}{f^*(x_0)}$ for some $C \geq 1$ and since $f^{-1}$ is increasing, we infer that $f^{-1}$ also has doubling property on $(f^*(x_0),\infty)$.
\begin{proposition}
	\label{prop:9}
	Suppose that $b=0$ and $-\phi'' \in \WLSC{\alpha-2}{c}{x_0}$ for some $c \in (0,1]$, $x_0 \geq 0$ and $\alpha>0$.
	Then for each $A>0$ and $M>0$ there is $C \geq 1$ so that for all $x < A/x_0$,
	\[
		C^{-1} \frac{1}{x\phi( 1/x)} 
		\leq \int_{\frac{x}{\phi'( f^{-1}(M/x))}}^{\infty} p(t,x) \rmd t 
		\leq 
		C\frac{1}{x\phi( 1/x)}.
	\]
	In particular, for each $A>0$ there is $C > 0$ such that for all $x < A/x_0$,
	\[
		G(x) \geq C \frac{1}{x\phi ( 1/x)}.
	\]
\end{proposition}
\begin{proof}
	For $M > 0$ and $x > 0$ we set
	\[
		I_M(x)=\int_{\frac{x}{\phi'(f^{-1}(M/x))}}^{\infty} p(t,x) \rmd t.
	\]
	Let us first show that for each $M>0$ there are $A_M>0$ and $C \geq 1$ such that for all $x<A_M/x_0$, 
	\begin{align}
	\label{eq:153}
	C^{-1} \frac{1}{x\phi( 1/x)} 
	\leq \int_{\frac{x}{\phi'( f^{-1}(M/x))}}^{\infty} p(t,x) \rmd t 
	\leq 
	C\frac{1}{x\phi( 1/x)}.
	\end{align}
	Let
	\[
		A_M = \min \big\{ M, c_0^{-1} M_0 \big\} \cdot \min\bigg\{1, \frac{x_0}{f^*(x_0)}\bigg\}
	\]
	where $M_0$ is determined in Corollary \ref{cor:1}, and $c_0$ is taken from \eqref{eq:74}. We claim that the following
	holds true.
	\begin{claim}
		\label{clm:1}
		For each $M >0$ there is $C \geq 1$ so that for all $x < A_M/x_0$,
		\begin{equation}
			\label{eq:37}
			C^{-1} \frac{1}{\phi'\big(f^{-1}(1/x)\big)}
			\leq 
			I_M(x)
			\leq
			C \frac{1}{\phi'\big(f^{-1}(1/x)\big)}.
		\end{equation}
	\end{claim}
	Suppose that
	\begin{equation}
		\label{eq:144}
		t > \frac{x}{\phi'\big(f^{-1}(M_1/x)\big)}
	\end{equation}
	with $M_1 = c_0^{-1} M_0$. Notice that for $x < A_M/x_0$, we have $x < M_1/f^*(x_0)$. Hence,
	$x_0 \leq f^{-1}\big(M_1 /x\big)$, thus by monotonicity of $\phi'$, we obtain
	\begin{align}
		\nonumber
		\frac{x}{t} 
		&\leq 
		\phi'\big(f^{-1}(M_1/x)\big) \\
		\label{eq:141}
		&\leq
		\phi'(x_0).
	\end{align}
	Moreover, for $w = (\phi')^{-1}(x/t)$ the condition \eqref{eq:144} implies that
	\[
		f^*(w) \geq M_1/x,
	\]
	which together with \eqref{eq:74} gives
	\begin{align}
		\nonumber
		t \vphi(w) = x f(w) 
		&\geq
		c_0 x f^*(w) \\
		\label{eq:140}
		&\geq
		M_0.
	\end{align}
	Now, to justify the claim, let us first consider $M \geq M_1$. In view of \eqref{eq:141} and \eqref{eq:140}
	we can apply Corollary \ref{cor:1} to get
	\[
		I_M(x)
		\approx 
		\int_{\frac{x}{\phi'( f^{-1}(M/x))}}^{\infty} \frac{1}{\sqrt{t (-\phi'' (w))}} 
		\exp\Big\{-t \big(\phi(w) - w \phi'(w)\big)\Big\} \rmd t.
	\]
	Since, by Proposition \ref{prop:3} and Remark \ref{rem:1}, for all $w>x_0$,
	\begin{align*}
		\phi(w)-w\phi'(w) 
		&\approx h(1/w) \\
		&\approx K(1/w) \\
		& \approx w^2 \big( -\phi''(w) \big),
	\end{align*}
	after the change of variables $t = x / \phi'(s)$, we can find $C_2 \geq 1$ such that for all $x < A_M/x_0$,
	\begin{align}
		\label{eq:136}
		\int_{f^{-1}(M/x)}^{\infty} 
		\exp\{ -C_2 x f(s) \} \sqrt{xf(s)} \frac{{\rm d}s}{s\phi'(s)} 
		\lesssim 
		I_M(x) 
		\lesssim 
		\int_{f^{-1}(M/x)}^{\infty} 
		\exp\{ -C_2^{-1} xf(s) \} \sqrt{xf(s)} \frac{{\rm d}s}{s\phi'(s)}.
	\end{align}
	Recall that $f^{-1}$ has the doubling property on $(f^*(x_0),\infty)$. 
	Using now Proposition \ref{prop:WLSC} and \eqref{eq:74}, we get
	\begin{align}
		\nonumber
		I_M(x) 
		&\gtrsim \int_{f^{-1}(M/x)}^{2f^{-1}(M/x)} \exp\{-C_2xf^*(s)\}\sqrt{xf^*(s)} \frac{{\rm d}s}{\phi(s)} \\
		\nonumber
		&\gtrsim \frac{1}{\phi \big( f^{-1}(M/x) \big) } f^{-1}(M/x) \\
		\label{eq:32}
		&\gtrsim \frac{1}{\phi' \big( f^{-1}(M/x) \big)},
	\end{align}
	where the implicit constants may depend on $M$. Therefore, by monotonicity of $f^{-1}$ and $\phi'$, the estimate
	\eqref{eq:32} gives
	\begin{align}
		\label{eq:147}
		I_M(x) \gtrsim \frac{1}{\phi' \big( f^{-1}(1/x) \big)}.
	\end{align}
	This proves the first inequality in \eqref{eq:37}. 
	
	We next observe that \eqref{eq:138} entails that $f^{-1}(s) \gtrsim s$ for $s > f^*(x_0)$, 
	thus, by \eqref{eq:147},
	\begin{align}
		\nonumber
		G(x) \geq I_{M_1}(x)
		&\gtrsim
		\frac{1}{\phi' (1/x)} \\
		\label{eq:145}
		&\gtrsim
		\frac{1}{x \phi(1/x)}
	\end{align}
	where the last estimate follows by Proposition \ref{prop:WLSC}.
	
	We next show the second inequality in \eqref{eq:37}. By \eqref{eq:136}, Proposition \ref{prop:WLSC} and monotonicity
	of $\phi$,
	\begin{align*}
		I_M(x) 
		&\lesssim
		\int_{f^{-1}(M/x)}^{\infty} 
		\exp\big\{ -C_2^{-1} xf(s) \big\} \sqrt{xf(s)} \frac{{\rm d}s}{\phi(s)} \\
		&\leq 
		\frac{1}{\phi \big( f^{-1}(M/x) \big)}
		\int_{f^{-1}(M/x)}^{\infty} \exp\big\{ -C_2^{-1} xf(s) \big\} \sqrt{xf(s)} \rmd s \\
		&\leq
		\frac{1}{\phi \big( f^{-1}(M/x) \big)}
		\int_{f^{-1}(M/x)}^{\infty} \exp \Big\{ -\tfrac{1}{2 C_2} xf(s) \Big\} \rmd s
	\end{align*}
	where in the last inequality we have used
	\[
		\exp \big\{ -C_2^{-1} xf(s) \big\} \sqrt{xf(s)} \leq 
		\exp \Big\{ -\tfrac{1}{2 C_2} xf(s) \Big\}.
	\]
	Since $\vphi \in \WLSC{\alpha}{c}{x_0}$, by \cite[Lemma 16]{MR3165234},
	\[
		\int_\RR  \exp\big\{ -C_2^{-1} xf(\abs{s}) \big\}\rmd s \lesssim
		f^{-1}\big(M_1/ x \big).
	\]
	Finally, the doubling property of $f^{-1}$, monotonicity of $\phi$, and Proposition \ref{prop:WLSC} give
	\begin{align*}
		I_M(x)
		&\lesssim \frac{1}{\phi \big( f^{-1}(M/x) \big)} f^{-1}\big(M_1/x\big) \\
		&\lesssim \frac{1}{\phi' \big( f^{-1}(1/x) \big)}
	\end{align*}
	where the implied constant may depend on $M$. This finishes the proof of \eqref{eq:37} for $M \geq M_1$.

	We next consider $0 < M < M_1$. By monotonicity, the lower estimate remains valid for all $M > 0$.
	Therefore, it is enough to show that for each $0 < M < M_1$, there is $C \geq 1$ such that for all
	$x < A_M/x_0$,
	\begin{align*}
		\int_{\frac{x}{\phi'( f^{-1}(M/x))}}^{\frac{x}{\phi'( f^{-1}(M_1/x))}} 
		p(t,x) \rmd t 
		\leq C \frac{1}{\phi' \big( f^{-1}(1/x) \big)}.
	\end{align*}
	By \cite[Theorem 3.1]{GS2019}, there is $t_0 > 0$ such that for all $0 < t < t_0$, 
	\begin{align*}
		p(t,x) \lesssim \varphi^{-1}(1/t).
	\end{align*}
	If $x_0=0$ then $t_0=\infty$. Since $\vphi$ is almost increasing we have
	\begin{align*}
		\frac{x}{\phi' \big( f^{-1}(M_1/x) \big)} 
		&\leq \frac{M_1}{\vphi \big( f^{-1}(M_1 /x) \big)} \\
		&\lesssim \frac{M_1}{\vphi \big( f^{-1}(M_1 x_0 /A) \big)}.
	\end{align*}
	Hence, by continuity and positivity of $p(t,x)$ and $\vphi^{-1}(1/t)$, we can take
	\[
		t_0 \geq \frac{x}{\phi' \big( f^{-1}(M_1 / x ) \big)}.
	\]
	Therefore, by the change of variables $t = x/\phi'(s)$ we get
	\begin{align*}
		\int_{\frac{x}{\phi'( f^{-1}(M/x))}}^{\frac{x}{\phi'( f^{-1}(M_1/x))}} p(t,x) \rmd t 
		&\lesssim
		\int_{\frac{x}{\phi'( f^{-1}(M/x) )}}^{\frac{x}{\phi'( f^{-1}(M_1/x))}} \vphi^{-1}(1/t) \rmd t \\
		&= 
		x
		\int_{f^{-1}(M/x)}^{f^{-1}(M_1/x)} \vphi^{-1} \bigg( \frac{\phi'(s)}{x} \bigg) f(s) 
		\frac{{\rm d}s}{s^2 \phi'(s)}.
	\end{align*}
	Next, by monotonicity and the doubling property of $f^{-1}$ and $\phi'$, we obtain
	\begin{align}
		\nonumber
		\int_{\frac{x}{\phi'( f^{-1}(M/x))}}^{\frac{x}{\phi'( f^{-1}(M_1/x))}} p(t,x) \rmd t
		&\lesssim
		\frac{1}{\big(f^{-1}(M/x)\big)^2}
		\cdot
		\frac{1}{\phi' \big( f^{-1}(M_1/x) \big)}
		\int_{f^{-1}(M/x)}^{f^{-1}(M_1/x)} \vphi^{-1} \bigg( \frac{\phi'(s)}{x} \bigg) \rmd s \\
		\label{eq:137}
		&\lesssim
		\frac{1}{\big(f^{-1}(1/x)\big)^2}
		\cdot
		\frac{1}{\phi' \big( f^{-1}(1/x) \big)}
		\int_{f^{-1}(M/x)}^{f^{-1}(M_1/x)} \vphi^{-1} \bigg( \frac{\phi'(s)}{x} \bigg) \rmd s.
	\end{align}
	Since, by \eqref{eq:74}, for $s \geq f^{-1}(M/x)$ we have
	\[
		\frac{\phi'(s)}{x} = \frac{\vphi(s)}{x f(s)} \lesssim \vphi^*(s),
	\]
	by monotonicity of $\vphi^{-1}$, Proposition \ref{prop:7}, Remark \ref{rem:5} and the doubling property of
	$f^{-1}$ and $\vphi^{-1}$, we get
	\[
		\int_{f^{-1}(M/x)}^{f^{-1}(M_1/x)} \vphi^{-1} \bigg( \frac{\phi'(s)}{x} \bigg) \rmd s
		\lesssim
		\big(f^{-1}(1/x)\big)^2,
	\]
	which together with \eqref{eq:137} gives \eqref{eq:37} for $0 < M < M_1$. This completes the proof of
	Claim \ref{clm:1}.

	Our next task is to deduce \eqref{eq:153} from Claim \ref{clm:1}. By Lemma \ref{lem:4} and Proposition \ref{prop:WLSC},
	there is a complete Bernstein function $\tilde{\phi}$ such that $\tilde{\phi} \approx \phi$, and
	\[
		f(x) \approx \tilde{f}(x) = \frac{x^2 \big(-\tilde{\phi}''(x) \big)}{\tilde{\phi}'(x)}
	\]
	for all $x > x_0$. Let $\tilde{\bfT}$ be a subordinator with the Laplace exponent $\tilde{\phi}$. By
	$\tilde{p}(t, \: \cdot \: )$ we denote the density of the probability distribution of $\tilde{T}_t$. We set
	\[
		\tilde{G}(x) = \int_0^\infty \tilde{p}(t, x) \rmd t
	\]
	and
	\[
		\tilde{I}_M(x) = \int_{\frac{x}{\tilde{\phi}'(\tilde{f}^{-1}(M/x))}}^{\infty} \tilde{p}(t,x) \rmd t.
	\]
	Fix $M >0$. By Claim \ref{clm:1}, there is $A_M > 0$ such that for all $x < A_M/x_0$,
	\begin{equation}
		\label{eq:55}
		\begin{aligned}
		\tilde{I}_M(x) 
		&\approx \frac{1}{\tilde{\phi}' \big( \tilde{f}^{-1}(1/x) \big)} \\
		&\approx \frac{1}{\phi'\big(f^{-1}(1/x)\big)}.
		\end{aligned}
	\end{equation}
	On the other hand, since $\tilde{\phi}$ is the complete Bernstein function, by \eqref{eq:145} and
	\cite[Corollary 2.6]{MR2986850}, there is $C_3 \geq 1$ such that
	\[
		C_3^{-1} \frac{1}{x\tilde{\phi}(1/x)} 
		\leq 
		\tilde{I}_M(x) 
		\leq 
		\tilde{G}(x) \leq C_3 \frac{1}{x\tilde{\phi}(1/x)}.
	\]
	Therefore, by \eqref{eq:55}, for $x < A_M/x_0$,
	\begin{equation}
		\label{eq:58}
		\begin{aligned}
		\frac{1}{\phi'\big(f^{-1}(1/x)\big)}
		\approx \tilde{I}_M(x)
		&\approx \frac{1}{x\tilde{\phi}(1/x)} \\
		&\approx \frac{1}{x \phi(1/x)},
		\end{aligned}
	\end{equation}
	and \eqref{eq:153} follows for all $A \leq A_M$. Let us now consider $A > A_M$. Observe that the functions
	\begin{align*}
		\bigg[\frac{A_M}{x_0},\frac{A}{x_0}\bigg] \ni x &\mapsto \frac{1}{x\phi(1/x)}, \intertext{and}
		\bigg[\frac{A_M}{x_0},\frac{A}{x_0}\bigg] \ni x &\mapsto 
		\int_{\frac{x}{\phi'( f^{-1}(M/x))}}^{\infty} p(t,x) \rmd t, 
	\end{align*}
	are both positive and continuous, thus they are bounded for each $A$. Therefore, at the possible expense of worsening
	the constant we can conclude the proof of the proposition.
\end{proof}

\begin{proposition}\label{prop:10}
	Suppose that $b=0$, $-\phi'' \in \WLSC{\alpha-2}{c}{x_0}$ for some $c \in (0,1]$, $x_0 \geq 0$ and $\alpha>0$,
	and that the L\'{e}vy measure $\nu({\rm d}x)$ is absolutely continuous with respect to the Lebesgue measure with a
	monotone density $\nu(x)$. Then there is $\epsilon \in (0,1)$ such that for each $A>0$, there is $C \geq 1$ such that
	for all $x < A/x_0$,
	\[
		\int_0^{\frac{x}{\phi' ( f^{-1}(1/x) )}\epsilon}
		p(t,x) \rmd t 
		\leq C \frac{1}{x\phi (1/x)}.
	\]
\end{proposition}

\begin{proof}
	In view of \eqref{eq:58}, it is enough to show that for some $\epsilon \in (0, 1)$ and all $A>0$ there is
	$C \geq 1$, such that for all $x<A/x_0$,
	\begin{align}
		\label{eq:143}
		\int_0^{\frac{x}{\phi' ( f^{-1}(1/x) )}\epsilon}
		p(t,x) \rmd t 
		\leq C \frac{1}{\phi' \big( f^{-1}(1/x) \big)}.
	\end{align}
	Let $\epsilon \in (0, 1)$ and
	\[
		A = \min \bigg\{ 1,\frac{x_0}{f^*(x_0)} \bigg\}.
	\]
	Suppose that
	\begin{equation}
		\label{eq:142}
		t \leq \frac{x}{\phi'\big(f^{-1}(1/x) \big)} \epsilon,
	\end{equation}
	that is
	\[
		t \leq \frac{1}{\vphi^*\big(f^{-1}(1/x) \big)} \epsilon.
	\]
	Hence, by monotonicity of $\vphi^{-1}$ and $\phi'$,
	\begin{align*}
		x 
		\geq \frac{1}{f^*\big(\vphi^{-1}(\epsilon/t)\big)}
		&= \frac{t}{\epsilon} \phi' \big( \vphi^{-1}(\epsilon/t) \big) \\
		&\geq \frac{t}{\epsilon} \phi' \big( \vphi^{-1}(1/t) \big).
	\end{align*}
	By Proposition \ref{prop:7} and the scaling property of $\phi'$, there are $c \in (0,1]$ and $C \geq 1$ such that
	\begin{align*}
		x 
		&\geq \frac{t}{\epsilon} \phi' \big( C\psi^{-1}(1/t) \big) \\
		&\geq \frac{t}{\epsilon}cC^{\alpha-1} \phi' \big( \psi^{-1}(1/t) \big).
	\end{align*}
	Therefore, by taking $\epsilon = (2e)^{-1}cC^{\alpha-1}$, we get
	\[
		x \geq 2 e t \phi'\big(\psi^{-1}(1/t) \big).
	\]
	Since $\nu(x)$ is the monotone density of $\nu({\rm d} x)$, by Theorem \ref{thm:4} we get
	\begin{align*}
		\int_0^{\frac{x}{\phi'( f^{-1}(1/x))} \epsilon} 
		p(t,x) \rmd t 
		&\lesssim 
		\frac{\vphi(1/x)}{x} \bigg( \frac{x}{\phi'\big(f^{-1}(1/x)\big)} \bigg)^2.
	\end{align*}
	By \eqref{eq:138}, $f^{-1}(s) \gtrsim s$ for $s>f^*(x_0)$, thus using \eqref{eq:59},
	\[
		\frac{x\vphi(1/x)}{\phi' \big( f^{-1}(1/x) \big)} 
		\leq
		\frac{\vphi(1/x)}{\vphi\big( f^{-1}(1/x) \big)} \lesssim 1,
	\]
	which entails \eqref{eq:143}. The extension to arbitrary $A$ follows by continuity and positivity argument as in the
	proof of Proposition \ref{prop:9}.
\end{proof}

It is possible to get the same conclusion as in Proposition \ref{prop:10} without imposing the existence of the monotone
density of $\nu({\rm d} x)$, however instead we need to assume the weak upper scaling property in $-\phi''$.
\begin{proposition}\label{prop:12}
	Suppose that $b=0$ and $-\phi'' \in \WLSC{\alpha-2}{c}{x_0} \cap \WUSC{\beta-2}{C}{x_0}$ for some $c \in (0,1]$, 
	$C \geq 1$, $x_0 \geq 0$ and $\tfrac{1}{2}<\alpha\leq \beta<1$. Then there is $\epsilon \in (0, 1)$ such that
	for each $A>0$, there is $C_1 \geq 1$, so that for all $x < A/x_0$,
	\begin{align}
		\label{eq:152}
		\int_0^{\frac{x}{\phi' ( f^{-1}(1/x) )} \epsilon}
		p(t,x) \rmd t 
		\leq C_1 \frac{1}{x\phi (1/x)}.
	\end{align}
\end{proposition}
\begin{proof}
	Let
	\[
		A = \min\bigg\{1, \frac{x_0}{f^*(x_0)} \bigg\}.
	\]
	By repeating the same reasoning as in the proof of Proposition \ref{prop:10}, we can see that the condition 
	\[
		t \leq \frac{x}{\phi' ( f^{-1}(1/x) )} \epsilon,
	\]
	implies
	\[
		x \geq 2 e t \phi'\big(\psi^{-1}(1/t) \big),
	\]
	for $\epsilon = (2e)^{-1} c C^{\alpha-1}$. Therefore, we can apply Theorem \ref{thm:6} to get
	\begin{equation}
		\label{eq:57}
		\int_0^{\frac{x}{\phi'( f^{-1}(1/x))} \epsilon}
		p(t,x) \rmd t 
		\lesssim
		\vphi(1/x) \int_0^{\frac{x}{\phi'( f^{-1}(1/x))}\epsilon}  t \vphi^{-1}(1/t) \rmd t
	\end{equation}
	where the implied constant may depend on $\epsilon$. Since $\alpha > \tfrac{1}{2}$, by Proposition \ref{prop:7},
	\cite[Theorem 3]{aa} and the doubling property of $\vphi^{-1}$, we obtain
	\begin{align}
		\nonumber
		\int_0^{\frac{x}{\phi'( f^{-1}(1/x))}\epsilon} t\vphi^{-1}(1/t) \rmd t 
		&\lesssim
		\bigg( \frac{x}{\phi'\big( f^{-1}(1/x) \big)} \bigg)^2 \vphi^{-1} 
		\bigg( \frac{\phi' \big( f^{-1}(1/x) \big)}{ \epsilon x} \bigg) \\
		\label{eq:139}
		&\lesssim
		\bigg( \frac{x}{\phi'\big( f^{-1}(1/x) \big)} \bigg)^2 \vphi^{-1} 
		\bigg( \frac{\phi' \big( f^{-1}(1/x) \big)}{x} \bigg).
	\end{align}
	In view of \eqref{eq:74}, we have
	\begin{align*}
		\frac{\phi' \big( f^{-1}(1/x) \big)}{x} 
		&= \frac{\vphi^*\big(f^{-1}(1/x) \big)}{x f^*\big(f^{-1}(1/x)\big)} \\
		&\lesssim \vphi^* \big( f^{-1}(1/x) \big),
	\end{align*}
	thus, by Proposition \ref{prop:7} and Remark \ref{rem:5},
	\begin{align*}
		\frac{\vphi(1/x)x^2}{\phi'\big( f^{-1}(1/x) \big)} \vphi^{-1} 
		\bigg( \frac{\phi' \big( f^{-1}(1/x) \big)}{x} \bigg)
		&\lesssim
		\frac{\vphi^*(1/x) x^2}{\phi'\big( f^{-1}(1/x) \big)} f^{-1}(1/x) \\
		&= 
		xf^{-1}(1/x) \frac{\vphi^*(1/x)}{\vphi^*\big(f^{-1}(1/x)\big)}.
	\end{align*}
	In view of Propositions \ref{prop:WLSC} and \ref{prop:1}, we have $f(s) \approx s$ for $s > x_0$, thus,
	$f^{-1}(s) \approx s$, for $s > f^*(x_0)$. Hence,
	\[
		\frac{\vphi(1/x)x^2}{\phi'\big( f^{-1}(1/x) \big)} \vphi^{-1}
		\bigg( \frac{\phi' \big( f^{-1}(1/x) \big)}{x} \bigg) \lesssim 1.
	\]
	Therefore, by \eqref{eq:57} and \eqref{eq:139} we conclude that
	\[
		\int_0^{\frac{x}{\phi' ( f^{-1}(1/x) )} \epsilon}
		p(t,x) \rmd t 
		\lesssim \frac{1}{\phi' \big( f^{-1}(1/x) \big)},
	\]
	which, by Proposition \ref{prop:9} and \eqref{eq:58}, entails \eqref{eq:152}. The extension to arbitrary $A$ follows
	by positivity and continuity argument.
\end{proof}

\begin{theorem}
	\label{thm:8}
	Let $\bfT$ be a subordinator with the Laplace exponent $\phi$. Suppose that
	\[
		\phi \in \WLSC{\alpha}{c}{x_0} \cap \WUSC{\beta}{C}{x_0}
	\]
	for some $c \in (0,1]$, $C \geq 1$, $x_0 \geq 0$, and
	$0<\alpha\leq \beta<1$. We assume that one of the following conditions holds:
	\begin{enumerate}
		\item The L\'{e}vy measure $\nu({\rm d}x)$ is absolutely continuous with respect to the Lebesgue measure with
		monotone density $\nu(x)$, or
		\item $\alpha > \tfrac12$.
	\end{enumerate}
	Then for each $A > 0$ there is $C_1 \geq 1$ such that for all $x< A/x_0$,
	\[
		C_1^{-1} \frac{1}{x\phi(1/x)} \leq G(x) \leq C_1 \frac{1}{x\phi(1/x)}.
	\]
\end{theorem}
\begin{proof}
	By Corollary \ref{cor:6}, $-\phi'' \in \WLSC{\alpha-2}{c}{x_0} \cap \WUSC{\beta-2}{C}{x_0}$. Let $p(t, \: \cdot \:)$
	be the transition density of $T_t$. In view of Propositions \ref{prop:9}, \ref{prop:10} and \ref{prop:12}, and
	\eqref{eq:58}, it is enough to show that for each $A > 0$ and $\epsilon \in (0, 1)$ there is $C_1 > 0$ such that
	for all $x<A/x_0$,
	\begin{align}
		\label{eq:150}
		\int_{\frac{x}{\phi'( f^{-1}(1/x))}\epsilon}^{\frac{x}{\phi'( f^{-1}(1/x))}} 
		p(t,x) \rmd t 
		\leq C_1 \frac{1}{\phi' \big( f^{-1}(1/x) \big)}.
	\end{align}
	By \cite[Theorem 3.1]{GS2019}, there is $t_0 > 0$ such that for all $t \in (0, t_0)$,
	\[
		p(t, x) \lesssim \vphi^{-1}(1/t).
	\]
	If $x_0 = 0$ then $t_0 = \infty$. We can take
	\[
		t_0 \geq \frac{x}{\phi'\big(f^{-1}(1/x)\big)}.
	\]
	Therefore, by monotonicity of $\vphi^{-1}$, we get
	\begin{align*}
		\int_{\frac{x}{\phi'( f^{-1}(1/x))}\epsilon}^{\frac{x}{\phi'( f^{-1}(1/x))}} 
		p(t,x) \rmd t 
		&\lesssim 
		\int_{\frac{x}{\phi'( f^{-1}(1/x))} \epsilon}^{\frac{x}{\phi'( f^{-1}(1/x))}} 
		\vphi^{-1}(1/t) \rmd t \\ 
		&\leq \frac{x}{\phi' \big( f^{-1}(1/x) \big)} \vphi^{-1} 
		\bigg( \frac{\phi' \big( f^{-1}(1/x) \big)}{\epsilon x} \bigg).
	\end{align*}
	By the doubling property of $\vphi^{-1}$, definition of $f$ and Remark \ref{rem:5},
	\begin{align*}
		\vphi^{-1} \bigg( \frac{\phi' \big( f^{-1}(1/x) \big)}{\epsilon x} \bigg) 
		&\lesssim \vphi^{-1} \bigg( \frac{\phi' \big( f^{-1}(1/x) \big)}{x} \bigg) \\ 
		&\leq f^{-1}(1/x) \\ 
		&\lesssim \frac1x,
	\end{align*}
	since by the weak upper scaling property of $-\phi''$, $f(s) \approx s$ for all $s > f^*(x_0)$. Consequently, we
	obtain \eqref{eq:150} and the theorem follows.
\end{proof}

\end{document}